\documentclass[11pt]{amsart}
\usepackage{amsmath, amssymb}
\usepackage{amsfonts}
\usepackage{mathrsfs}
\usepackage[arrow,matrix,curve,cmtip,ps]{xy}
\usepackage{tikz-cd}

\usepackage{amsthm}
\usepackage{mathtools} 
\usepackage{array}
\usepackage{mathdots}

\usepackage[top= 2.25cm, bottom=2cm, left = 2cm, right= 2cm]{geometry}
\usepackage{hyperref}
\usepackage{setspace}

\allowdisplaybreaks

\numberwithin{equation}{section}

\newtheorem{theorem}{Theorem}[section]
\newtheorem{lemma}[theorem]{Lemma}
\newtheorem{proposition}[theorem]{Proposition}
\newtheorem{corollary}[theorem]{Corollary}

\newtheorem{conjecture}[theorem]{Conjecture}

\newtheorem*{theorem*}{Theorem}
\theoremstyle{remark}
\newtheorem{remark}[theorem]{Remark}
\newtheorem{definition}[theorem]{Definition}

\numberwithin{equation}{section}

\newcommand{\lif}[1]{\widetilde{#1}}  

\newcommand{\com}[2][\theta_{0}]{#2^{#1}}

\newcommand{\End}[2]{\mathcal{E}_{#1}(#2)}

\newcommand{\tEnd}[3][\omega]{\mathcal{E}_{#2}(#3,#1)}

\newcommand{\uP}[1]{\Phi^{+}_{unit}(#1)}

\renewcommand{\P}[1]{\Phi(#1)}

\newcommand{\Pbd}[1]{\Phi_{bdd}(#1)}

\newcommand{\Psm}[1]{\Phi_{sim}(#1)}

\newcommand{\cuP}[1]{\bar{\Phi}^{+}_{unit}(#1)}

\newcommand{\cP}[1]{\bar{\Phi}(#1)}

\newcommand{\cPbd}[1]{\bar{\Phi}_{bdd}(#1)}

\newcommand{\cPdt}[1]{\bar{\Phi}_{2}(#1)}

\newcommand{\cQdt}[1]{\bar{\Psi}_{2}(#1)}

\newcommand{\cPsm}[1]{\bar{\Phi}_{sim}(#1)}

\newcommand{\cPel}[1]{\bar{\Phi}_{ell}(#1)}

\newcommand{\p}{\phi}

\newcommand{\q}{\psi}

\newcommand{\lp}{\tilde{\phi}}

\newcommand{\Pkt}[1]{\Pi_{#1}}

\newcommand{\cPkt}[1]{\bar{\Pi}_{#1}}

\newcommand{\clPkt}[1]{\tilde{\bar{\Pi}}_{#1}}

\renewcommand{\r}{\pi}

\newcommand{\lr}{\tilde{\pi}}

\renewcommand{\H}{\mathcal{H}}

\newcommand{\sH}{\bar{\mathcal{H}}}

\newcommand{\Idt}[2]{I_{disc #2}^{#1}}

\newcommand{\tIdt}[2]{I_{disc #2}^{(#1, \omega)}}

\newcommand{\Sdt}[2]{S_{disc #2}^{#1}} 

\newcommand{\cS}[1]{\bar{S}_{#1}}

\renewcommand{\S}[1]{\mathcal{S}_{#1}}

\newcommand{\+}{\oplus}

\renewcommand{\#}{\boxplus}

\renewcommand{\c}{\lambda}

\newcommand{\lG}{\widetilde{G}}

\newcommand{\lM}{\widetilde{M}}

\newcommand{\lP}{\widetilde{P}}

\newcommand{\x}{\omega}

\renewcommand{\L}[1]{{}^L#1}

\newcommand{\D}[1]{\widehat{#1}}

\newcommand{\Gal}[1]{\Gamma_{#1}}

\renewcommand{\a}{\alpha}

\renewcommand{\Im}{\text{Im}\,}

\newcommand{\lZ}{Z_{\widetilde{G}}}

\newcommand{\Z}{Z_{G}}

\newcommand{\Res}{\text{Res}}

\newcommand{\Cent}{\text{Cent}}

\newcommand{\Two}{\mathbb{Z}/2\mathbb{Z}}

\newcommand{\C}{\mathbb{C}}

\newcommand{\A}{\mathbb{A}}

\newcommand{\ep}{\phi^{\mathcal{E}}}

\newcommand{\lf}{\tilde{f}}

\newcommand{\Norm}{\text{Norm}}

\newcommand{\N}[1]{\frak{N}_{#1}}

\newcommand{\cT}[1]{\bar{T}_{#1}}

\newcommand{\Aut}{\text{Aut}}

\newcommand{\Out}{\text{Out}}

\newcommand{\Hom}{\text{Hom}}

\begin{document}
\title{Global L-packets of quasisplit $GSp(2n)$ and $GO(2n)$}

\author{Bin Xu} 
\thanks{Supported by NSFC No. 20191300979 and Tsinghua University Initiative Scientific Research Program No. 2019Z07L02016}

\address{Yau Mathematical Sciences Center and Department of Mathematics \\  Tsinghua University, Beijing, China}
\email{binxu@tsinghua.edu.cn}

\subjclass[2010]{22E50 (primary); 11F70 (secondary)}
\keywords{similitude group, twisted endoscopic transfer, L-packet, stabilized twisted trace formula}

\begin{abstract}
This is a sequel to \cite{Xu:2018} on the $L$-packets of quasisplit general symplectic and even orthogonal groups. We show the existence of global $L$-packets and establish the functoriality of endoscopic transfer for them in many cases.
\end{abstract}

\maketitle

\section{Introduction}
\label{sec: introduction}

Let $F$ be a number field and $G$ be a quasisplit symplectic or special even orthogonal group over $F$. Let $\mathbb{A}_{F}$ be the ad\`ele ring of $F$. We fix an automorphism $\theta_{0}$ of $G$ preserving an $F$-splitting. It induces a dual automorphism $\D{\theta}_0$ on the dual group $\D{G}$. When $G$ is symplectic, $\theta_{0}$ is trivial. When $G$ is special even orthogonal, we require $\theta_{0}$ to be the unique nontrivial outer automorphism induced from the conjugation of the full orthogonal group. We say two irreducible admissible representations of $G(\A_{F})$ are $\theta_0$-conjugate if they are $\theta_0$-conjugate at every place. In \cite{Arthur:2013} Arthur proved the discrete automorphic spectrum of $G(\A_{F})$ can be decomposed
\[
L^{2}_{disc}(G(F) \backslash G(\mathbb{A}_{F})) = \bigoplus_{\psi \in \cQdt{G}} L^{2}_{disc, \psi}(G(F) \backslash G(\mathbb{A}_{F}))
\]
according to the set $\cQdt{G}$ of $\D{\theta}_0$-conjugacy classes of discrete Arthur parameters of $G$. For each $\q \in \cQdt{G}$, Arthur associated a multi-set $\cPkt{\q}$ of isomorphism classes of irreducible admissible representations of $G(\mathbb{A}_{F})$ modulo $\theta_0$-conjugacy, and proved that
\[
L^{2}_{disc, \psi}(G(F) \backslash G(\mathbb{A}_{F})) = \bigoplus_{\substack{[\pi] \in \cPkt{\q} \\ \langle \cdot, \pi \rangle = \epsilon_{\q}}} \pi
\]
modulo $\theta_0$-conjugacy, where $\epsilon_{\q}$ is a character of certain two-group $\S{\q}$ associated with $\q$ and there is a map
\[
\cPkt{\q} \rightarrow \D{\S{\q}}, \quad [\pi] \mapsto \langle \cdot , \r \rangle.
\]
Let $\lG$ be the group of symplectic or orthogonal similitudes over $F$, whose derived group is $G$. We can extend $\theta_0$ to $\lG$. Let $\tilde{\zeta}$ be a character of $Z_{\lG}(\mathbb{A}_{F})/Z_{\lG}(F)$ and $\zeta$ be the restriction to $Z_{G}(\mathbb{A}_{F})$. Then we have
\[
{\rm Ind}^{\, \lG(\A_{F})}_{\, \lG(F)\lZ(\A_{F})G(\A_{F})} L^{2}_{disc}( G(F) \backslash G(\A_{F}), \zeta) \cong L^{2}_{disc}(\lG(F) \backslash \lG(\A_{F}), \lif{\zeta})
\]
(cf. \cite[Lemma 5.3]{Xu:2018}). So we can decompose 
\[
L^{2}_{disc}(\lG(F) \backslash \lG(\A_{F}), \lif{\zeta}) = \bigoplus_{\psi \in \bar{\Psi}_{2}(G, \zeta)} L^{2}_{disc, \psi}(\lG(F) \backslash \lG(\mathbb{A}_{F}), \lif{\zeta})
\]
where $\bar{\Psi}_{2}(G, \zeta) \subseteq \bar{\Psi}_{2}(G)$ is associated with the central character $\zeta$ and
\[
L^{2}_{disc, \psi}(\lG(F) \backslash \lG(\mathbb{A}_{F}), \lif{\zeta}) := {\rm Ind}^{\, \lG(\A_{F})}_{\, \lG(F)\lZ(\A_{F})G(\A_{F})} L^{2}_{disc, \q}( G(F) \backslash G(\A_{F}), \zeta).
\]

In this paper, we will focus on the subset of generic parameters $\bar{\Phi}_2(G, \zeta) = \bar{\Psi}_2(G, \zeta) \cap \bar{\Phi}(G)$. Let $\clPkt{\p, \lif{\zeta}}$ be the set of isomorphism classes of irreducible admissible representations of $\lG(\A_{F})$ modulo $\theta_0$-conjugacy with central character $\lif{\zeta}$, whose restriction to $G(\A_{F})$ have irreducible constituents contained in $\cPkt{\p}$. Let 
\[
Y = \Hom(\lG(\A_{F}) / \lG(F) \lZ(\A_{F})G(\A_{F}), \C^{\times}), \quad X = \Hom(\lG(\A_{F}) / \lZ(\A_{F})G(\A_{F}), \C^{\times}).
\] 
There is a natural group homomorphism
\(
\alpha: \S{\p} \rightarrow Y
\)
(cf. \eqref{eq: global twisted endoscopic sequence}), so that $\r \otimes \omega \cong \r$ for any $\x \in \alpha(\S{\p})$ and $[\r] \in \clPkt{\p, \lif{\zeta}}$. Define $\S{\lp} := {\rm Ker} \, \a$, and
\[
\clPkt{\p, \lif{\zeta}} \rightarrow \D{\S{\lp}}, \quad [\lr] \mapsto \langle \cdot , \lr \rangle
\]
by
\[
\langle \cdot , \lr \rangle = \langle \cdot , \r \rangle|_{\S{\lp}}
\]
where $\r$ is any irreducible constituent in the restriction of $\lr$ to $G(\A_{F})$. By \cite[Proposition 5.11]{Xu:2018},
\begin{align}
\label{eq: coarse decomposition}
L^2_{disc, \p} (\lG(F) \backslash \lG(\A_{F}), \lif{\zeta}) = \sum_{\x \in Y / \a(\S{\p})} \quad \sum_{\substack{[\lr] \in \clPkt{\p, \lif{\zeta}} / X  \\  \langle \cdot, \lr \rangle = 1}} \lr \otimes \x,
\end{align}
modulo $\theta_0$-conjugacy, where $\lr$ are taken to be the representatives of $\clPkt{\p, \lif{\zeta}} / X$ in the discrete automorphic spectrum of $\lG(\A_{F})$. Our first main result gives a refined decomposition of \eqref{eq: coarse decomposition}.

\begin{theorem}
\label{thm: global L-packet 1}
For any $\p \in \cPdt{G}$, there exists a global packet $\cPkt{\lp}$ of isomorphism classes of irreducible admissible representations for $\lG(\A_{F})$ modulo $\theta_0$-conjugacy, unique up to twisting by $Y$, such that
\[
\cPkt{\lp} = \otimes'_{v} \cPkt{\lp_v}
\]
where $\cPkt{\lp_v}$ is defined in \cite[Theorem 4.6]{Xu:2018}, and
\begin{align*}
L^{2}_{disc, \p}(\lG(F) \backslash \lG(\A_{F}), \lif{\zeta}) =  \bigoplus_{\x \in Y / \a(\S{\p})} \bigoplus_{\substack{ [\lr] \in \cPkt{\lp} \otimes \x \\ \langle \cdot, \lr \rangle = 1}} \lr
\end{align*}
modulo $\theta_0$-conjugacy.
\end{theorem}

This result was originally announced in the author's thesis \cite[Theorem 5.3.1]{Xu:thesis}, but there is an error in the proof. It is due to the application of a comparison formula (cf. \cite[Lemma 5.4.5]{Xu:thesis}, \cite[Lemma 5.25]{Xu:2018}) in the induction argument, whose validity depends on the assumption of Conjecture~\ref{conj: stable multiplicity formula}. This conjecture was only proved in certain cases (cf. \cite[Theorem 5.3.5]{Xu:thesis}, \cite[Theorem 5.21]{Xu:2018}). In \cite{Xu:2018}, we reproduce the proof of Theorem~\ref{thm: global L-packet 1} in the stable case, i.e., $\S{\lp} = 1$. In the current work, we are able to prove this in all cases by establishing a weak form of this conjecture (cf. Theorem~\ref{thm: stable multiplicity formula}).

Our second main result is on the functoriality of endoscopic transfer. Let $\lif{H}$ be an elliptic endoscopic group of $\lG$. Then the derived group $H$ of $\lif{H}$ is an elliptic endoscopic group of $G$ and it is isomorphic to $G_{1} \times G_{2}$, where $G_{i}$ is a quasisplit symplectic or special even orthogonal group. Here we also allow $G_i$ to be trivial. Let $\lG_i$ be the corresponding group of similitudes and $\lambda_i$ the similitude character. In case $G_{i}$ is trivial, let $\lG_i = \mathbb{G}_m$ and $\lambda_i = id$. Then
\[
\lif{H} \cong \Big\{(g_1, g_2) \in \lG_1 \times \lG_2 \, | \, \lambda_1(g_1) = \lambda_2(g_2) \Big\}.
\]
For $\phi_i \in \cPdt{G_i}$, let $\phi_H := \phi_1 \times \phi_2$ and we define a global $L$-packet $\cPkt{\lp_H}$ of $\tilde{H}$ to be the restriction of $\cPkt{\lp_1} \otimes \,\cPkt{\lp_2}$. Let $\p := \p_1 \boxplus \p_2 \in \cPdt{G}$ (cf. \eqref{eq: global parameter}). Then we define the global endoscopic transfer of $\cPkt{\lp_{H}}$ to be
\[
{\rm Tran} \, \cPkt{\lp_H} := \otimes'_{v} {\rm Tran} \, \cPkt{\lp_{H, v}}
\]
where ${\rm Tran} \, \cPkt{\lp_{H, v}} := \cPkt{\lp_v} $ is the local $L$-packet transfered from $\cPkt{\lp_{H, v}}$ through the local character relation (cf. \cite[Theorem 4.6]{Xu:2018}). It follows from Theorem~\ref{thm: global L-packet 1} that 
\[
{\rm Tran} \, \cPkt{\lp_H} = \cPkt{\lp} \otimes \omega
\] 
for some $\omega \in X$. The functoriality of endoscopic transfer in this case means that $\omega$ can be chosen in $Y$. We prove this under certain technical assumption.

\begin{theorem}
\label{thm: functoriality 1}
Suppose either $\phi_1$ or $\phi_2$ does not contain $GL$-type orthogonal simple parameters (cf. Definition~\ref{def: GL-type}), then 
\[
{\rm Tran} \, \cPkt{\lp_H} = \cPkt{\lp} \otimes \omega
\] 
for some $\omega \in Y$.
\end{theorem}

This result strengthens \cite[Theorem 7.4.3]{Xu:thesis} and \cite[Lemma 6.27]{Xu:2018}. To give a few examples, the assumption in this theorem is always satisfied in the following cases:
\begin{itemize}


\item $\lG = GSp(4), GSp(6), GSO(4, \eta), GSO(6, \eta)$, where $\eta$ is a quadratic idele class character;

\item $\lG = GSO(8, \eta)$ for $\eta \neq 1$;

\item $\lG = GSp(2n)$ and $\lif{H} = GSO(2n)$.

\end{itemize}

This paper is organized as follows. In Section~\ref{sec: L-packet}, we review our earlier work \cite{Xu:2018} on both the local and global theory about $\lG$. In Section~\ref{sec: stable trace formula}, we review the relevant stabilized twisted trace formulas. In Section~\ref{sec: main results}, we formulate the main results of this paper in their most general forms. In particular, we also consider the twisted version of Theorem~\ref{thm: functoriality 1}. In Section~\ref{sec: comparison formula}, we formulate some comparison formulas. In Section~\ref{sec: GL-type}, we prove the main results for parameters consisting of orthogonal simple parameters of $GL$-type (cf. Definition~\ref{def: GL-type}). In Section~\ref{sec: endoscopic expansion}, we prove the comparison formulas following the same strategy as that of \cite[Lemma 5.25]{Xu:2018}. In Section~\ref{sec: compatible lifting}, we show that the global $L$-packets of $G$ has compatible lifting to $\lG$ (cf. Corollary~\ref{cor: compatible lifting packet for product}). The notion of compatible lifting has been introduced in \cite[Section 7.3]{Xu:thesis}. It has the following implication. Let $\lif{H}'_{1}, \lif{H}'_{2}$ be endoscopic groups of $\lG$. If $\lif{H}''$ is an endoscopic group for both $\lif{H}'_{1}, \lif{H}'_{2}$ such that the following diagram of endoscopic embeddings commutes.
\[
\xymatrix{& \L{\lif{H}'}_1 \ar[dr]^{\xi_1} &  \\
\L{\lif{H}''} \ar[ur]^{\iota_1} \ar[dr]_{\iota_2} &&  \L{\lG} \\
& \L{\lif{H}'}_{2} \ar[ur]_{\xi_{2}} & 
}
\]
Then
\begin{align}
\label{eq: transitivity}
{\rm Tran}_{\xi_{1}} \circ {\rm Tran}_{\iota_1} \, \cPkt{\lp_{H''}} = {\rm Tran}_{\xi_{2}} \circ {\rm Tran}_{\iota_2} \, \cPkt{\lp_{H''}}
\end{align}
for any global $L$-packet $\cPkt{\lp_{H''}}$ of $\lif{H}''$. This is essentially a local statement. It is not trivial since one can not compose the geometric endoscopic transfer maps. It provides a reduction for the proof of Theorem~\ref{thm: functoriality 1}. As another application, we construct a bijection between local Langlands parameters and local $L$-packets for $\lG$ in the nonarchimedean case such that it is compatible with the endoscopic transfers (cf. Theorem~\ref{thm: LLC} and \cite[Theorem 7.3.5]{Xu:thesis}). This construction depends on certain choices, so it is not canonical. Nevertheless, it will suffice for our construction of the local Arthur packets for $\lG$ in a forthcoming work. It is also necessary to consider \eqref{eq: transitivity} for the twisted endoscopic groups. To do so, we have claimed incorrectly in \cite[Section 7.3]{Xu:thesis} that there is a canonical lift of endoscopic embedding from an endoscopic group of $G$ to a twisted endoscopic group of $\lG$. In fact, that construction has an obstruction by a $2$-cycle of the Galois group in the central torus $\D{D}$ of the dual group in the case when $\lG = GSO(2n, \eta)$ for $\eta \neq 1$. In the current paper, we fix this error by introducing a $1$-cochain of the Weil group in $\D{D}$, which splits the $2$-cocycle. It is the fact that we can not make a canonical choice for the twisted endoscopic embedding for $\lG$ makes the discussion in this section complicated. In the last section, we prove the main results in the cases left by Section~\ref{sec: GL-type}. The proof of Theorem~\ref{thm: global L-packet 1} is similar to that of \cite[Theorem 5.21]{Xu:2018}. In the proof of Theorem~\ref{thm: functoriality 1}, we introduce some auxiliary global parameters by doubling certain simple parameter. This is motivated by some arguments in \cite{Arthur:2013}. The proof is similar to that of \cite[Theorem 7.4.3]{Xu:thesis}.


\section{L-packets}
\label{sec: L-packet}

Let $F$ be a local or global field of characteristic zero and $\bar{F}$ be its algebraic closure. When $F$ is global, let us denote the ad\`ele ring of $F$ by $\A_{F}$, the id\`ele group by $I_{F}$. Let $\Gal{F}$ or $\Gal{}$ be the absolute Galois group of $F$ and $W_{F}$ the Weil group. For any quasisplit connected reductive group $G$ over $F$, we denote by $Z_{G}$ its center, by $A_{G}$ the split connected component of $Z_{G}$, by $\D{G}$ its complex dual group, by $Z(\D{G})$ the centre of $\D{G}$, and by $\L{G}$ its $L$-group, which is a semidirect product of $\D{G}$ with the Weil group $W_{F}$, i.e., $\D{G} \rtimes W_{F}$. Let $X^{*}(G)$ be the group of algebraic characters of $G$ over $F$ and $\mathfrak{a}_{G} = \Hom_{\mathbb{Z}}(X^{*}(G), \mathbb{R})$.

Let $G \subseteq \lG$ be two quasisplit connected reductive groups over $F$ such that they have the same derived group. Then $\lG/G$ is a torus, denoted by $D$. There is an exact sequence
\begin{align}
\label{eq: extension}
\xymatrix{1 \ar[r] & G \ar[r] & \lG \ar[r]^{\lambda}  & D \ar[r] & 1.}
\end{align}
On the dual side, we have 
\begin{align*}
\xymatrix{1 \ar[r] & \D{D} \ar[r] & \D{\lG} \ar[r]^{\bold{p}}  & \D{G} \ar[r] & 1,}
\end{align*}
where all the homomorphisms can be extended to $L$-homomorphisms of $L$-groups. Let $\Sigma$ be a finite abelian group of $F$-automorphisms of $\lG$ preserving a fixed $F$-splitting of $\lG$. We will always assume that $\c$ is $\Sigma$-invariant. This implies that $\Sigma$ also acts on $G$. Let $\lG^{\Sigma} = \lG \rtimes \Sigma$ and $G^{\Sigma} = G \rtimes \Sigma$. Since $\Sigma$ induces dual automorphisms on $\D{\lG}$ and $\D{G}$, we denote this group by $\D{\Sigma}$ and define $\D{\lG}^{\Sigma} = \D{\lG} \rtimes \D{\Sigma}$ and $\D{G}^{\Sigma} = \D{G} \rtimes \D{\Sigma}$. 

As in \cite{Xu:2018} we are mainly concerned with the case that $G$ is a quasisplit symplectic or special even orthogonal group and $\lG$ is the corresponding similitude group. We denote the split symplectic and special even orthogonal group of rank $n$ by $Sp(2n)$ and $SO(2n)$ respectively. We also denote the outer twist of $SO(2n)$ with respect to a quadratic extension $E / F$ by $SO(2n, \eta_{E/F})$, where $\eta_{E/F}$ is the quadratic (id\`ele class) character associated to $E/F$ by the local (global) class field theory. Denote $\eta_{G} = \eta_{E/F}$. The corresponding similitude groups can be defined as follows.
\[
GSp(2n) = (\mathbb{G}_{m} \times Sp(2n)) / (\Two) \,\, \text{ and } \,\, GSO(2n, \eta_{E/F}) = (\mathbb{G}_{m} \times SO(2n, \eta_{E/F})) / (\Two),
\] 
where $\Two$ is embedded diagonally into the centre of each factor. The similitude character $\c$ is square on $\mathbb{G}_{m}$ and trivial on the other factor. We fix an automorphism $\theta_{0}$ of $G$ preserving an $F$-splitting. When $G$ is symplectic, we require $\theta_{0}$ to be trivial. When $G$ is special even orthogonal, we require $\theta_{0}$ to be the unique nontrivial outer automorphism induced from the conjugation of the full orthogonal group. Clearly, $\theta_{0}^{2} = 1$, $\theta_{0}$ extends to $\lG$ by acting trivially on $\lZ$, and $\c$ is $\theta_{0}$-invariant. Let $\Sigma_{0} = \langle \theta_{0} \rangle$. For our induction arguments in the proofs, we will also need to consider 
\begin{align}
\label{eq: product}
G = G_{1} \times \cdots \times G_{q}
\end{align}
where $G_i$ is a quasisplit symplectic or special even orthogonal group. We define
\[
\lG = (\mathbb{G}_{m} \times G_{1} \times G_{2} \times \cdots \times G_{q}) / (\Two),
\]
where $\Two$ is embedded diagonally into the center of each factor. We also define a character $\c$ of $\lG$, which is square on $\mathbb{G}_{m}$ and trivial on the other factors. We define a group of automorphisms of $G$ by taking the product of $\Sigma_{0}$ on each factor, and we denote this group again by $\Sigma_{0}$. We can extend $\Sigma_0$ to $\lG$ by the trivial action on $Z_{\lG}$. We can also view $\lG$ as a subgroup of $\lG_1 \times \cdots \times \lG_q$ by 
\[
\lG \cong \Big\{ (g_i) \in \prod_{i} \lG_{i} \, | \, \lambda_{1}(g_{1}) = \cdots = \lambda_{q}(g_{q}) \Big\}.
\] 
For admissible representations $\lr_i$ of $\lG_i(F)$ in the local case (or $\lG_{i}(\A_F)$ in the global case), we define the restriction of $\otimes_{i} \, \lr_{i}$ to $\lG$ by $\widetilde{\otimes}_i \, \lr_i$.

\subsection{Local theory}
\label{subsec: local}

Suppose $F$ is local. The local Langlands group is defined as follows 
\[
L_{F} = 
\begin{cases}
W_{F} & \text{if $F$ is archimedean}, \\ 
W_{F} \times SL(2, \C) & \text{if $F$ is nonarchimedean}.
\end{cases}
\] 
Let $G$ be a quasisplit connected reductive group over $F$. A local Langlands parameter $\p$ is a $\D{G}$-conjugacy class of admissible homomorphisms from $L_{F}$ to $\L{G}$ (cf. \cite{Borel:1979}). Let $\P{G}$ be the set of local Langlands parameters and $\Pbd{G}$ be the subset of bounded parameters, i.e., the closure of $\p(W_{F})$ is compact. Let $\Pkt{}(G(F))$ be the set of isomorphism classes of irreducible admissible representations of $G(F)$ and $\Pkt{temp}(G(F))$ be the subset of tempered representations. If $\chi$ is a quasicharacter of a closed subgroup $Z_{F}$ of $\Z(F)$, we define $\H(G, \chi)$ to be the space of $\chi^{-1}$-equivariant smooth functions on $G(F)$ with compact support modulo $Z_{F}$.

For $G \subseteq \lG$ as in \eqref{eq: extension}, the projection
\(
\bold{p}
\)
induces a surjection $\P{\lG} \rightarrow \P{G}$ (cf. \cite[Section 2.2]{Xu:2018}). For $\p \in \P{G}$, let $L_{F}$ act on $\D{D}$, $\D{G}^{\Sigma}$, and $\D{\lG}^{\Sigma}$ by conjugation through $\p$. We denote the corresponding group cohomology by $H^{*}_{\p}(L_{F}, \cdot)$. Then
\(
H^{0}_{\p}(L_{F}, \D{D}) = \D{D}^{\Gal{}}
\)
and
\(
H^{1}_{\p}(L_{F}, \D{D}) = H^{1}(W_{F}, \D{D}).
\)
We define
\[
S^{\Sigma}_{\p}: = \Cent(\Im \p, \D{G}^{\Sigma}) = H^{0}_{\p}(L_{F}, \D{G}^{\Sigma}), \quad S_{\lp}^{\Sigma} := \Cent(\Im \lp, \D{\lG}^{\Sigma}) = H^{0}_{\p}(L_{F}, \D{\lG}^{\Sigma}).
\]
The short exact sequence 
\[
\xymatrix{1 \ar[r] & \D{D} \ar[r] & \D{\lG}^{\Sigma} \ar[r]  & \D{G}^{\Sigma} \ar[r] & 1}
\]
induces a long exact sequence
\begin{align*}
\xymatrix{1 \ar[r] &  \D{D}^{\Gamma} \ar[r] & S_{\lp}^{\Sigma} \ar[r] & S_{\p}^{\Sigma} \ar[r]^{\delta \quad \quad} & H^{1}(W_{F}, \D{D}),}
\end{align*}
and hence we get
\begin{align}
\label{eq: old twisted endoscopic sequence}
\xymatrix{1 \ar[r] &  S_{\lp}^{\Sigma}/\D{D}^{\Gal{}} \ar[r]^{\quad \iota} & S_{\p}^{\Sigma} \ar[r]^{\delta \quad \quad} & H^{1}(W_{F}, \D{D}).}
\end{align}
Taking the quotient of \eqref{eq: old twisted endoscopic sequence} by $Z(\D{G})^{\Gal{}}$, we get
\begin{align}
\label{eq: twisted endoscopic sequence mod center}
\xymatrix{1 \ar[r] &  \cS{\lp}^{\Sigma} \ar[r]^{\iota} & \cS{\p}^{\Sigma} \ar[r]^{\bar{\delta} \quad \quad}  & \bar{H}^{1}(W_{F}, \D{D}),}                 
\end{align}
where $\cS{\lp}^{\Sigma} =  S_{\lp}^{\Sigma}/Z(\D{\lG})^{\Gal{}}$, $\cS{\p}^{\Sigma} =  S_{\p}^{\Sigma}/Z(\D{G})^{\Gal{}}$ and $\bar{H}^{1}(W_{F}, \D{D}) = H^{1}(W_{F}, \D{D})/\delta(Z(\D{G})^{\Gamma{}})$. Since $\Im \delta$ is finite (cf. \cite[Lemma 2.1]{Xu:2018}), then $\cS{\lp}^{0} = \cS{\p}^{0}$, which are the identity components. After taking the quotient of \eqref{eq: twisted endoscopic sequence mod center} by the identity components, we get
\begin{align}
\label{eq: twisted endoscopic sequence}
\xymatrix{1 \ar[r] &  \S{\lp}^{\Sigma} \ar[r]^{\iota} & \S{\p}^{\Sigma} \ar[r]^{\bar{\delta} \quad \quad}  & \bar{H}^{1}(W_{F}, \D{D}),}                
\end{align}
where $\S{\lp}^{\Sigma} =  \cS{\lp}^{\Sigma} / \cS{\lp}^{0}$ and $\S{\p}^{\Sigma} =  \cS{\p}^{\Sigma} / \cS{\p}^{0}$.  There are natural maps from $S^{\Sigma}_{\p}, \cS{\p}^{\Sigma}$, and $\S{\p}^{\Sigma}$ to $\D{\Sigma}$, and for $\theta \in \Sigma$, we denote the preimages of $\D{\theta} \in \D{\Sigma}$ by $S^{\theta}_{\p}, \cS{\p}^{\theta}$ and $\S{\p}^{\theta}$ respectively. To understand $\bar{H}^{1}(W_{F}, \D{D})$, we also take
\[
\xymatrix{1 \ar[r] & \D{D} \ar[r] & Z(\D{\lG}) \ar[r]  & Z(\D{G}) \ar[r] & 1}
\]
and it induces
\[
\xymatrix{Z(\D{G})^{\Gal{}} \ar[r]^{\delta \quad} & H^{1}(W_{F}, \D{D}) \ar[r] & H^{1}(W_{F}, Z(\D{\lG})) \ar[r] & H^{1}(W_{F}, Z(\D{G})) \ar[r] & H^{2}(W_{F}, \D{D}) = 1.
}
\]
So
\[
\bar{H}^{1}(W_{F}, \D{D}) \cong {\rm Im}\{H^{1}(W_{F}, \D{D}) \rightarrow H^{1}(W_{F}, Z(\D{\lG}))\}
\]
and
\begin{align}
\label{eq: twist}
\xymatrix{1 \ar[r] & \bar{H}^{1}(W_{F}, \D{D}) \ar[r] & H^{1}(W_{F}, Z(\D{\lG})) \ar[r] & H^{1}(W_{F}, Z(\D{G})) \ar[r] & 1.
}
\end{align}
By the isomorphism $H^{1}(W_{F}, Z(\D{\lG})) \longrightarrow \Hom(\lG(F), \C^{\times})$ \cite[Appendix A]{Xu:2016}, we get an isomorphism
\(
r: \bar{H}^{1}(W_{F}, \D{D}) \rightarrow \Hom(\lG(F)/G(F), \C^{\times})
\)
\cite[Section 2.2]{Xu:2018}. Let us denote $r \circ \bar{\delta}$ by $\a$.

Suppose $G$ is \eqref{eq: product}. We denote the set of $\Sigma_{0}$-orbits in $\Pkt{}(G(F))$ by $\cPkt{}(G(F))$ and the set of $\D{\Sigma}_{0}$-orbits in $\P{G}$ by $\cP{G}$. Similarly, we can define $\cPkt{temp}(G(F))$, $\cPbd{G}$, and analogues of these sets for $\lG$. We denote by $\sH(G, \chi)$ (resp. $\sH(\lG, \lif{\chi})$) the subspace of $\Sigma_{0}$-invariant functions in $\H(G, \chi)$ (resp. $\H(\lG, \lif{\chi})$). By the local Langlands correspondence for classical groups cf. \cite[Theorem 1.5.1]{Arthur:2013}, we can associate any $[\p] \in \cPbd{G}$ with a finite subset $\cPkt{\p}$ of $\cPkt{temp}(G(F))$ such that 
\[
f(\p) := \sum_{\r \in \cPkt{\p}} f_{G}(\r), \quad \quad f \in \sH(G)
\]
is stable and
\begin{align}
\label{eq: disjoint decomposition}
\cPkt{temp}(G(F)) = \bigsqcup_{[\p] \in \cPbd{G}} \cPkt{\p}.
\end{align}
Moreover, with respect to a fixed $\Sigma_{0}$-stable Whittaker datum of $G$, there is a canonical pairing between $\cPkt{\p}$ and $\S{\p}$, which induces an inclusion of $\cPkt{\p}$ into the set of irreducible characters $\D{\S{\p}}$ of $\S{\p}$
\[
\cPkt{\p} \longrightarrow \D{\S{\p}}, \quad \r \mapsto \langle \cdot, \r \rangle
\]
such that the generic representation is sent to the trivial character. For $[\p] \in \cPbd{G}$, the lift $\lp$ is unique up to twists by 
\(
\bar{H}^{1}(W_{F}, \D{D})
\)
which will be identified with ${\rm Hom}(\lG(F)/G(F), \mathbb{C}^{\times})$. For any subgroup $\Sigma \subseteq \Sigma_0$, we get from \eqref{eq: twisted endoscopic sequence}
\begin{align}
\label{eq: local twisted endoscopic sequence}
\xymatrix{1 \ar[r] &  \S{\lp}^{\Sigma} \ar[r]^{\iota} & \S{\p}^{\Sigma} \ar[r]^{\a \quad \quad \quad \quad}  & \Hom(\lG(F)/G(F), \C^{\times})}.               
\end{align}
In \cite[Theorem 4.6]{Xu:2018} we have shown that there exists a lift $\cPkt{\lp} \subseteq \cPkt{temp}(\lG(F))$ unique up to twisting by ${\rm Hom}(\lG(F)/G(F), \mathbb{C}^{\times})$ such that
\[
\lf(\lp) := \sum_{[\lr] \in \cPkt{\lp}} \lf_{\lG}(\lr), \quad \quad \lf \in \sH(\lG)
\]
is stable. Moreover, we have shown for $\omega \in {\rm Hom}(\lG(F)/G(F), \mathbb{C}^{\times})$ and any $[\lr] \in \cPkt{\lp}$,
\[
[\lp \otimes \omega] = [\lp] \Leftrightarrow [\lr \otimes \omega] = [\lr] \Leftrightarrow \omega \in \a(\S{\p}^{\Sigma_0}).
\] 
As a consequence,
\begin{align}
\label{eq: disjoint decomposition similitude}
\cPkt{temp}(\lG(F)) = \bigsqcup_{[\p] \in \cPbd{G}} \, \bigsqcup_{\omega \in  {\rm Hom}(\lG(F)/G(F), \mathbb{C}^{\times})/\a(\S{\p}^{\Sigma_0})} \cPkt{\lp} \otimes \omega.
\end{align}
We can define
\[
\cPkt{\lp} \longrightarrow \D{\S{\lp}}, \quad \langle \cdot, \lr \rangle := \langle \cdot, \r \rangle|_{\S{\lp}}
\]
for any $\r \subseteq \lr|_{G}$. For $[\p] \in \cPbd{G}$, $\theta \in \Sigma_0$ and any semisimple element $s \in \cS{\p}^{\theta}$, let $\D{G}'  := \Cent(s, \D{G})^{0}$ and it can be equipped with a Galois action given by $\p$. This determines a quasisplit connected reductive group $G'$, and $\p$ will factor through $\L{G'}$ for some $\theta$-twisted endoscopic datum $(G', s, \xi)$ of $G$, and hence we get a parameter $\p' \in \cP{G'}$. In this way, we call $(G', \p')$ corresponds to $(\p, s)$, and denote this relation by $(G', \p') \rightarrow (\p, s)$. By \cite{Xu:2018}, $(G', s, \xi)$ can be lifted to a $(\theta, \x)$-twisted endoscopic datum $(\lG', \lif{s}, \lif{\xi})$ of $\lG$ for some $\omega \in {\rm Hom}(\lG(F)/G(F), \mathbb{C}^{\times})$. In sum, we have the following diagram
\begin{align}
\label{eq: lift endoscopic embedding}
\xymatrix{ \L{\lG'} \ar[r]^{\lif{\xi}}     \ar[d]   & \L{\lG}   \ar[d] \\
 \L{G'}    \ar[r]^{\xi}     &    \L{G}.}
\end{align}
There is a decomposition $G' \cong M_{l} \times G'_{-}$, where $M_l$ is a product of general linear groups and $G'_{-}$ is as \eqref{eq: product} (cf. \cite[Section 2]{Xu:2018}). Correspondingly, we can decompose $\p' = \p_{l} \times \p'_{-}$ and define $\cPkt{\p'} = \Pkt{\p_l} \otimes \cPkt{\p'_{-}}$, where $\Pkt{\p_l}$ is the singleton $L$-packet given by the local Langlands correspondence for general linear groups \cite{HarrisTaylor:2001} \cite{Henniart:2000} \cite{Scholze:2013}. Let $x$ be the image of $s$ in $\S{\p}^{\theta}$. Arthur established the following character relation
\[
f^{G'}(\p') = \sum_{\r \in \cPkt{\p}} \langle x, \r^{+}\rangle f_{G^{\theta}}(\r), \quad f \in \sH(G)
\]
(cf. \cite[Theorem 2.2.1 and Theorem 2.2.4]{Arthur:2013}), where $f^{G'}$ the Langlands-Shelstad-Kottwitz transfer of $f$,
\(
f_{G^{\theta}}(\r) := {\rm tr} \, \pi(f) \circ A_{\r}(\theta)
\)
for an intertwining operator $A_{\r}(\theta)$ between $\r$ and $\r^{\theta}$ of order two. The choice of $A_{\r}(\theta)$ determines an extension $\r^{+}$ of $\r$ to $G(F)^{+} := G(F) \rtimes \langle \theta \rangle$, which corresponds to the extension $\langle \cdot, \r^{+} \rangle$ of $\langle \cdot, \r \rangle$ to $\S{\p}^{+}$ generated by $\S{\p}^{\theta}$. We define $\cPkt{\lp'}$ to be $\Pkt{\p_{l}} \otimes \cPkt{\lp'_{-}}$. Then we can choose $\cPkt{\lp}$ such that
\[
\lf^{\lG'}(\lp') = \sum_{\lr \in \cPkt{\lp}} \lf_{\lG^{\theta}}(\lr, \omega), \quad \lf \in \sH(\lG)
\]
(cf. \cite[Theorem 4.6]{Xu:2018}), where $\lf_{\lG^{\theta}}(\lr, \x) = {\rm tr} (\lr(\lf) \circ A_{\lr}(\theta, \x))$, and $A_{\lr}(\theta, \omega)$ is the intertwining operator between $\lr \otimes \omega$ and $\lr^{\theta}$ normalized by  
\[
A_{\lr}(\theta, \omega)|_{\r} = \langle x, \r^{+} \rangle A_{\r}(\theta)
\]
for any $\r \subseteq \lr|_{G}$. In this case, we say the transfer of $\cPkt{\lp'}$ with respect to $\lif{\xi}$ is $\cPkt{\lp}$ and write ${\rm Tran}_{\lif{\xi}} \, \cPkt{\lp'} = \cPkt{\lp}$. 

\begin{remark}
\label{rk: compatible with parabolic induction}
For any proper Levi subgroup $\lif{M}'$ of $\lG'$, there exists a $\theta$-stable proper Levi subgroup $\lif{M}$ of $\lG$ and an $L$-homomorphism $\tilde{\xi}_{\lif{M}'}$ such that the following diagram commutes
\[
\xymatrix{\L{\lG}' \ar[r]^{\tilde{\xi}} & \L{\lG} \\
\L{\lif{M}}' \ar[r]^{\tilde{\xi}_{\lif{M}'}} \ar[u] & \L{\lif{M}} \ar[u]
}
\]
and $(\lif{M}', \tilde{s}, \tilde{\xi}_{\lif{M}'})$ is a $(\theta, \omega)$-twisted endoscopic datum of $\lif{M}$. Suppose $\p'$ factors through $\p'_{M'} \in \cPbd{M'}$ for the Levi subgroup $M' := \lif{M}' \cap G'$ of $G'$, then the compatibility of twisted endoscopic transfer with parabolic induction gives
\[
{\rm Tran}_{\lif{\xi}} \, \cPkt{\lp'} = {\rm Tran}_{\tilde{\xi}} \, \, {\rm Ind}^{\lG'}_{\lif{M}'} \, \cPkt{\lp'_{M'}} = {\rm Ind}^{\lG}_{\lif{M}} \,\, {\rm Tran}_{\tilde{\xi}_{\lif{M}'}} \, \cPkt{\lp'_{M'}}.
\]
If $\lif{M} \cong \lif{M}'$, then ${\rm Tran}_{\tilde{\xi}_{\lif{M}'}}$ corresponds to twisting by a character depending on $\tilde{\xi}_{\lif{M}'}$, so it also transfers representations. In this case, we can still make sense of ${\rm Tran}_{\lif{\xi}} \, \cPkt{\lp'}$. 
\end{remark}

\subsection{Global theory}
\label{subsec: global}

Suppose $F$ is global and $G$ is a quasisplit symplectic or special even orthogonal group over $F$. Arthur \cite[Theorem 1.5.2]{Arthur:2013} showed the following decomposition
\[
L^{2}_{disc}(G(F) \backslash G(\mathbb{A}_{F})) = \bigoplus_{\psi \in \cQdt{G}} L^{2}_{disc, \psi}(G(F) \backslash G(\mathbb{A}_{F})),
\]
where $\cQdt{G}$ is the set of $\D{\Sigma}_0$-conjugacy classes of the discrete global Arthur parameters of $G$. We will only consider the subset $\cPdt{G}$ of generic parameters in this paper. Let $N = 2n + 1$ if $G = Sp(2n)$ and $N = 2n$ if $G = SO(2n, \eta_{E/F})$. Then $\cPdt{G}$ will be defined in terms of automorphic representations of $GL(N, \A_{F})$ with respect to the $GL(N, \mathbb{C})$-conjugacy class of twisted endoscopic embedding $\xi_{G}: \L{G} \longrightarrow GL(N, \mathbb{C})$ satisfying
\begin{itemize}
\item $\xi_{G}|_{\D{G}}$ is the standard representation of $\D{G}$,
\item $\xi_{G}|_{W_{F}}$ is trivial if $N$ is odd, and factors through $\Gal{E/F}$ with the nontrivial element sent to a reflection if $N$ is even.
\end{itemize}
We denote by $\Psm{m}$ the set of isomorphism classes of irreducible unitary cuspidal automorphic representations of $GL(m, \A_{F})$. For $\p \in \Psm{m}$, we denote its dual by $\p^{\vee}$. When $\p = \p^{\vee}$, we say $\p$ is of orthogonal type (resp. symplectic type) if the symmetric square (resp.  skew-symmetric square) $L$-function $L(s, \p, S^2)$ (resp. $L(s, \p, \wedge^2)$) has a pole at $s = 1$. Denote the central character of $\p$ by $\eta_{\p}$. Suppose $\p = \p^{\vee} \in \Psm{m}$, then $\eta_{\p}$ is quadratic and we denote the associated quadratic extension by $E/F$. We associate $\p$ with a twisted endoscopic datum $(G_{\p}, s_{\p}, \xi_{\p})$ of $GL(m)$ as follows.
\[
\L{G_{\p}} = \begin{cases} SO(2n + 1, \mathbb{C}) \times W_{F} & \text{ if $N = 2n +1$}  \\
SO(2n, \mathbb{C}) \rtimes W_{F} & \text{ if $N = 2n$ and $\p$ is orthogonal type} \\
Sp(2n, \mathbb{C}) \times W_{F} & \text{ if $N = 2n$ and $\p$ is symplectic type}
\end{cases}
\]
where the action of $W_{F}$ factors through $\Gal{E/F}$. The twisted endoscopic embedding 
\(
\xi_{\p}: \L{G_{\p}} \longrightarrow GL(m, \mathbb{C})
\)
satisfies:
\begin{itemize}
\item $\xi_{\p}|_{\D{G}_{\p}}$ is the standard representation,
\item $\xi_{\p}|_{W_{F}}$ factors through $\Gal{E/F}$ and the image of the nontrivial element of $\Gal{E/F}$ is $-I$ if $N$ is odd, and is a reflection if $N$ is even.
\end{itemize}
It should be noted that $\xi_{\p}$ is not equivalent to $\xi_{G_{\p}}$ if $N$ is odd and $\eta_{\p} \neq 1$. For any place $v$, $\p_{v}$ is an irreducible admissible representation of $GL(m, F_{v})$. By the local Langlands correspondence of general linear groups \cite{HarrisTaylor:2001} \cite{Henniart:2000} \cite{Scholze:2013}, we can associate it with a Langlands parameter of $GL(m, F_{v})$, still denoted by $\p_{v}$. We will define $\L{G_{\p_{v}}}$ and $\xi_{\p_{v}}$ by restriction. By \cite[Corollary 6.8.1]{Arthur:2013}, $\p_{v}$ factors through $\xi_{\p_{v}}$. So we have the following diagram.
\begin{align}
\label{diag: local-global}
\xymatrix{\p_{v}: L_{F_{v}} \ar[r] \ar[rd] & \L{G_{\p_{v}}} \ar[d] \ar[r]^{\xi_{\p_{v}} \quad } & GL(m, \mathbb{C}) \ar[d] \\
& \L{G_{\p}} \ar[r]^{\xi_{\p} \quad } & GL(m, \mathbb{C})
}
\end{align}
We denote by $\P{N}$ the set of formal direct sums
\begin{align}
\label{eq: formal direct sum}
\p = l_{1}\p_{1} \# \cdots \# l_{r}\p_{r}
\end{align}
such that $\p_{i} \in \Psm{N_{i}}$ and $\sum_{i=1}^{r} l_{i}N_{i} = N$. We can assign a family of semisimple conjugacy classes in $GL(N, \mathbb{C})$ by 
\[
c(\p_{v}) :=  \underbrace{c(\p_{1, v}) \+ \cdots \+ c(\p_{1, v})}_{l_{1}} \+ \cdots \+ \underbrace{c(\p_{r, v}) \+ \cdots \+ c(\p_{r, v})}_{l_{r}}
\]
at unramified places $v$ of $\p$, where $c(\p_{i, v})$ is the Satake parameter of the local component $\p_{i, v}$. We define an involution on $\P{N}$ by
\[
\p^{\vee} = l_{1}\p^{\vee}_{1} \# \cdots \# l_{r}\p^{\vee}_{r}.
\]
Suppose $\p \in \P{N}$ such that $\p = \p^{\vee}$, we also get an involution on the set of indices by requiring $\p_{i^{\vee}} = \p_{i}^{\vee}$. Consequently $l_{i} = l_{i^{\vee}}$. This gives a disjoint decomposition of these indices 
\begin{align}
\label{eq: partition}
I_{\p} \, \bigsqcup \, J_{\p} \, \bigsqcup \, J_{\p}^{\vee}
\end{align}
where $I_{\p}$ indexes the set of self-dual simple parameters. Let $K_{\p} = I_{\p} \sqcup J_{\p}$ and $I_{\p, O}$ (resp. $I_{\p, S}$) index the self-dual simple parameters of orthogonal (resp. symplectic) type. For $i \in I_{\p}$, let $G_{i} = G_{\p_i}$ and $\xi_i = \xi_{\p_{i}}$. For $j \in J_{\p}$, let $G_{j} = GL(N_{j})$ and define $\xi_{j} : \L{G_{j}} \longrightarrow GL(2N_{j}, \mathbb{C})$ by sending $g \rtimes w$ to $\text{diag}\{g, {}^tg^{-1}\}$. Then we define a substitute global Langlands group for $\p$ following \cite[Section 1.4]{Arthur:2013} by taking the fibre product
\[
\mathcal{L}_{\p} := \prod_{k \in K_{\p}} \{ \L{G_{k}} \longrightarrow W_{F} \}.
\]
It is equipped with $W_{F} \rightarrow \mathcal{L}_{\p}$ and $L_{F_{v}} \longrightarrow \mathcal{L}_{\p}$ by \eqref{diag: local-global}. Let 
\[
\p^{\mathcal{E}} := \bigoplus_{k \in K_{\p}} l_{k}\xi_{k} : \mathcal{L}_{\p} \longrightarrow GL(N, \mathbb{C}).
\]
We define 
\begin{align}
\label{eq: global parameter}
\cP{G} := \{ \p \in \P{N} \, : \, \p = \p^{\vee} \text{ and } \ep \text{ factors through } \xi_{G} \}.
\end{align}
which substitutes for the set of $\widehat{\Sigma}_{0}$-conjugacy classes of global Langlands parameters for $G$. Write $\P{G, \p}$ for the set of $\D{G}$-conjugacy classes of $L$-homomorphisms $\p_{G}: \mathcal{L}_{\p} \rightarrow \L{G}$ such that $\xi_{G} \circ \p_{G}$ is $GL(N, \mathbb{C})$-conjugate to $\ep$. So the pair $(\p, \p_{G})$ is a substitute for the global Langlands parameter for $G$. Let $m_{\p} = |\P{G, \p}|$. For $\p \in \cP{G}$ and subgroup $\Sigma \subseteq \Sigma_{0}$, we define
\begin{align*}
S^{\Sigma}_{\p} = \Cent(\Im \p_{G}, \D{G}^{\Sigma}), \quad \cS{\p}^{\Sigma} = S^{\Sigma}_{\p} / Z(\D{G})^{\Gal{F}}, \quad \S{\p}^{\Sigma} = \cS{\p}^{\Sigma} / \com[0]{\cS{\p}}.
\end{align*}
and $S^{\theta}_{\p}, \cS{\p}^{\theta}$ and $\S{\p}^{\theta}$ for $\theta \in \Sigma_0$ as in the local case. Let  
\begin{align*}
& \cPsm{G} = \{ \p \in \cP{G} : \cS{\p}^{\Sigma_0} = 1\}, \quad \cPdt{G} = \{ \p \in \cP{G} : |\cS{\p}| < \infty \}, \quad \cP{G^{\theta}}  = \{ \p \in \cP{G}: S_{\p}^{\theta} \neq \emptyset\},  \\ 
& \cPel{G^{\theta}} = \{ \p \in \cP{G^{\theta}} : |\cS{\p, s}^{0}| < \infty \, \text{for some semisimple} \, s \in \cS{\p}^{\theta} \},
\end{align*}
where $\theta \in \Sigma_{0}$. The following lemma is a direct consequence of the computation of $S_{\p}$ (cf. \cite[Lemma 3.2]{Xu:2018}).
\begin{lemma}
\label{lemma: discrete parameter} \,
\begin{enumerate}
\item $\cPsm{G} = \Psm{N} \cap \cP{G}$.
\item Suppose $\p \in \cP{G}$, then $\p$ is in $\cPdt{G}$ if and only if $K_{\p} = I_{\p, O}$ and $l_{i} =1$ for all $i \in K_{\p}$.
\item Suppose $\p$ is in $\cPel{G^{\theta}}$ for $\theta \in \Sigma_{0}$, then $K_{\p} = I_{\p, O}$ and $l_{i} \leqslant 2$ for all $i \in K_{\p}$.
\item Suppose $G$ is special even orthogonal and $\p \in \cP{G}$, then $\p$ is in $\cP{G^{\theta_{0}}}$ if and only if there exists $i \in I_{\p, O}$ such that $N_{i}$ is odd.
\end{enumerate}
\end{lemma}

At any place $v$, we define subsets $\cPsm{G_v}, \cPdt{G_v}, \cP{G^{\theta}_v}, \cPel{G^{\theta}_v}$ of $\cPbd{G_v}$ as in the global case. For $[\p_{v}] \in \cPbd{G_{v}}$, we can view $\p_{v}$ as a representation of $L_{F_{v}}$ through the composition with $\xi_{G_v}$. Then we can decompose it into irreducible subrepresentations 
\[
\p_v = l_{1}\p_{v, 1} \oplus \cdots \oplus l_{r}\p_{v, r}.
\] 
Since $\p_v$ is self-dual, we will get a partition of indices as \eqref{eq: partition}. If $\p_{v, i}$ is self-dual, we say it is of orthogonal (resp. symplectic) type if it factors through an orthogonal (symplectic) group. Then Lemma~\ref{lemma: discrete parameter} is still valid.

For $\p \in \cP{G}$, we define $\p_{v}$ through the following diagram
\begin{align}
\label{eq: local-global}
\xymatrix{L_{F_{v}} \ar[r]^{\p_{v}}  \ar[d]    &  \L{G_{v}} \ar[d] \\
 \mathcal{L}_{\p} \ar[r]^{\p_{G}}  & \L{G}}
\end{align}
The generalized Ramanujan conjecture would imply that $[\p_{v}] \in \cPbd{G_v}$. Since we do not have this conjecture at the moment, we can only assume that $[\p_{v}] \in \cuP{G_{v}}$, which is a subset of $\cP{G_v}$ characterized by the property that
\[
\xi_{G_v} \circ \p_{v} =  \p_{1} \+ \cdots \+ \p_{r} \+ ( \nu^{a_{1}}\p_{r+1} \+ \nu^{-a_{1}}\p_{r+1} ) \cdots \+ (\nu^{a_{s}}\p_{r+s} \+ \nu^{-a_{s}}\p_{r+s})
\]
where $\p_{i}$ is unitary for $1 \leqslant i \leqslant r+s$ and $0 < a_{j} < 1/2 \text{ for } 1 \leqslant j \leqslant s$ (cf. \cite[Proposition 3.10]{Xu:2018}). 

For any $[\p_v] \in \cuP{G^{\theta}_v}$ with $\theta \in \Sigma_{0}$, $\p_v$ factors through $\p_{M_v, \lambda} := \p_{M_v} \otimes (\lambda \circ |\cdot|_{F})$ for some $\theta$-stable parabolic subgroup $P \supseteq M$, where $\p_{M_v} \in \cPbd{M_v^{\theta}}$ and $\lambda \in \mathfrak{a}^{*}_{M_v}$ lies in the open chamber determined by $P$. The normalized parabolic induction induces bijections $\cPkt{\p_v} \cong \cPkt{\p_{M_v, \lambda}}$ and $\cPkt{\lp_v} \cong \cPkt{\lp_{M_v, \lambda}}$ (cf. \cite[Proposition 3.11]{Xu:2018}). Moreover, we have the following diagram
\begin{align}
\label{eq: nontempered}
\xymatrix{1 \ar[r] &  \S{\lp_{M_v}}^{\Sigma} \ar[r]^{\iota_{M_{v}}} \ar[d]^{\cong} & \S{\p_{M_{v}}}^{\Sigma} \ar[r]^{\a^{M_{v}} \quad \quad \quad \quad \quad} \ar[d]^{\cong} & \Hom(\lG(F_{v})/G(F_{v}), \C^{\times}) \ar@{=}[d] \\
1 \ar[r] &  \S{\lp_v}^{\Sigma} \ar[r]^{\iota_v} & \S{\p_v}^{\Sigma} \ar[r]^{\a_v \quad \quad \quad \quad \quad} & \Hom(\lG(F_{v})/G(F_{v}), \C^{\times}).           
}               
\end{align}
Since the parabolic induction preserves stability and is compatible with the twisted endoscopy, the previous local results for $\cPbd{G_{v}}$ extend to $\uP{G_{v}}$ except for \eqref{eq: disjoint decomposition} \eqref{eq: disjoint decomposition similitude}.

The diagram \eqref{eq: local-global} gives rise to an inclusion $S_{\p} \hookrightarrow S_{\p_{v}}$ at any place $v$, which induces a homomorphism $\S{\p} \rightarrow \S{\p_{v}}$. By the local Langlands correspondence for $G_{v}$, we can associate $\p_{v}$ with $\cPkt{\p_{v}}$ and define the global $L$-packet by taking the restricted tensor product
\[
\cPkt{\p} := \sideset{}{'} \bigotimes_{v}  \cPkt{\p_{v}}
\]
and define the global pairing by
\[
\langle x, \r \rangle := \prod_{v} \langle x_{v}, \r_{v} \rangle, \quad x \in \S{\p}, \, [\r] \in \cPkt{\p}
\]
For $\p \in \cPdt{G}$, let $L^2_{disc, \p} (G(F) \backslash G(\A_{F}))$ be the subspace of discrete automorphic representations whose Satake parameters are sent to $c(\p_{v})$ under $\xi_{G_v}$ for almost all places. Let $\sH(G) = \otimes'_{v} \sH(G_{v})$, then there is a decomposition as $\sH(G)$-modules
\begin{align}
\label{eq: decomposition G}
L^2_{disc, \p} (G(F) \backslash G(\A_{F}))= m_{\p} \sum_{\substack{[\r] \in \cPkt{\p}  \\  \langle \cdot, \r \rangle = 1}} \r
\end{align}
(cf. \cite[Theorem 1.5.2]{Arthur:2013}).

Suppose $G = G_{1} \times \cdots \times G_{q}$ as in \eqref{eq: product} and we define $\cP{G}$ to be the set of $\p = \p_{1} \times \p_{2} \times \cdots \times \p_{q}$ such that $\p_{i} \in \cP{G_{i}}$ for $1 \leqslant i \leqslant q$. We also define 
\[
\mathcal{L}_{\p} := \prod_{i=1}^{q} \{\mathcal{L}_{\p_{i}} \longrightarrow W_{F} \}
\]
and
\[
\ep = \prod_{i = 1}^{q} \ep_{i}.
\]
Then we can define $S^{\Sigma}_{\p}$ and $\Phi(G, \p)$ as before. Note $S_{\p} = \prod_{i=1}^{q} S_{\p_{i}}$ and $m_{\p} = \prod_{i=1}^{q} m_{\p_{i}}$. Then \eqref{eq: decomposition G} still holds. Let $\tilde{\zeta}$ be a character of $Z_{\lG}(\mathbb{A}_{F})/Z_{\lG}(F)$ and $\zeta$ its restriction to $Z_{G}(\mathbb{A}_{F})$. Then we have
\[
{\rm Ind}^{\, \lG(\A_{F})}_{\, \lG(F)\lZ(\A_{F})G(\A_{F})} L^{2}_{disc}( G(F) \backslash G(\A_{F}), \zeta) \cong L^{2}_{disc}(\lG(F) \backslash \lG(\A_{F}), \lif{\zeta})
\]
(cf. \cite[Lemma 5.3]{Xu:2018}). So we can decompose 
\[
L^{2}_{disc}(\lG(F) \backslash \lG(\A_{F}), \lif{\zeta}) = \bigoplus_{\psi \in \bar{\Psi}_{2}(G, \zeta)} L^{2}_{disc, \psi}(\lG(F) \backslash \lG(\mathbb{A}_{F}), \lif{\zeta})
\]
where $\bar{\Psi}_{2}(G, \zeta) \subseteq \bar{\Psi}_{2}(G)$ is associated with the central character $\zeta$ and
\[
L^{2}_{disc, \psi}(\lG(F) \backslash \lG(\mathbb{A}_{F}), \lif{\zeta}) := {\rm Ind}^{\, \lG(\A_{F})}_{\, \lG(F)\lZ(\A_{F})G(\A_{F})} L^{2}_{disc, \q}( G(F) \backslash G(\A_{F}), \zeta).
\]
Our goal is to decompose the right hand side in terms of global $L$-packets for $\q = \p \in \cPdt{G, \zeta}$. In parallel with the local case, we also have \eqref{eq: twisted endoscopic sequence} and an isomorphism
\(
r: \bar{H}^{1}(W_{F}, \D{D}) \rightarrow \Hom(\lG(\A_{F})/ \lG(F)G(\A_{F}), \C^{\times})
\) (cf. \cite[Lemma 2.11]{Xu:2018}). Denote $r \circ \bar{\delta}$ by $\a$. For any $\Sigma \subseteq \Sigma_0$, we have 
\begin{align}
\label{eq: global twisted endoscopic sequence}
\xymatrix{1 \ar[r] &  \S{\lp}^{\Sigma} \ar[r]^{\iota} \ar[d] & \S{\p}^{\Sigma} \ar[r]^{\a \quad \quad \quad \quad \quad \quad} \ar[d] & \Hom(\lG(\A_{F})/ \lG(F)G(\A_{F}), \C^{\times}) \ar[d] \\
1 \ar[r] &  \S{\lp_v}^{\Sigma} \ar[r]^{\iota_v} & \S{\p_v}^{\Sigma} \ar[r]^{\a_v \quad \quad \quad \quad \quad} & \Hom(\lG(F_{v})/G(F_{v}), \C^{\times}).           
}               
\end{align}
Let $\clPkt{\p, \lif{\zeta}}$ be the set of isomorphism classes of irreducible admissible representations of $\lG(\A_{F})$ as $\sH(\lG)$-modules with central character $\lif{\zeta}$, whose restriction to $G(\A_{F})$ have irreducible constituents contained in $\cPkt{\p}$. It follows from \eqref{eq: global twisted endoscopic sequence} that $[\lr \otimes \omega] = [\lr]$ for any $[\lr] \in \clPkt{\p, \lif{\zeta}}$ and $\omega \in \a(\S{\p}^{\Sigma_0})$. Let 
\[
X = {\rm Hom}(\lG(\A_{F}) / \lZ(\A_{F})G(\A_{F}), \mathbb{C}^{\times}), \quad Y = {\rm Hom}(\lG(\A_{F}) / \lG(F) \lZ(\A_{F})G(\A_{F}), \mathbb{C}^{\times}).
\] 
Then $\a(\S{\p}^{\Sigma_0}) \subseteq Y$. We define the global pairing as
\[
\langle x, \lr \rangle := \prod_{v} \langle x_{v}, \lr_{v} \rangle, \quad x \in \S{\lp}, \, [\lr] \in \clPkt{\p, \lif{\zeta}}.
\]
By \cite[Corollary 5.6]{Xu:2018}, we can always choose a representative $\lr$ in the discrete automorphic spectrum of $\lG(\A_{F})$ for $[\lr] \in \clPkt{\p, \lif{\zeta}} / X$ with $\langle \cdot, \lr \rangle = 1$. Moreover, we have shown the following decomposition as $\sH(\lG)$-modules 
\begin{align}
\label{eq: discrete spectrum}
L^2_{disc, \p} (\lG(F) \backslash \lG(\A_{F}), \lif{\zeta}) = m_{\p} \sum_{\x \in Y / \a(\S{\p})} \quad \sum_{\substack{[\lr] \in \clPkt{\p, \lif{\zeta}} / X  \\  \langle \cdot, \lr \rangle = 1}} \lr \otimes \x,
\end{align}
where $\lr$ are taken to be the representatives of $\clPkt{\p, \lif{\zeta}} / X$ in the discrete automorphic spectrum of $\lG(\A_{F})$ (cf. \cite[Proposition 5.11]{Xu:2018}).

\section{Stable trace formula}
\label{sec: stable trace formula}

Suppose $F$ is global and $G$ is \eqref{eq: product}. Let $\lif{Z}_{\A_{F}} = \prod'_{v} \lif{Z}_{F_{v}}$ be a closed subgroup of $\lZ(\A_{F})$ such that $\lif{Z}_{\A_{F}} \Z(\A_{F}) = \lZ(\A_{F})$. Let $\lif{Z}_{F} = \lif{Z}_{\A_{F}} \cap \lZ(F)$ and $\lif{\chi}$ a character of $\lif{Z}_{\A_{F}} / \lif{Z}_{F}$. Let $Z_{{\A_{F}}} = \lif{Z}_{\A_{F}} \cap \Z(\A_{F})$ and $Z_{F} = \lif{Z}_{F} \cap \Z(F)$. We denote the restriction of $\lif{\chi}$ to $Z_{\A_{F}}$ by $\chi$. Let $\theta \in \Sigma_0$ and $\x \in Y$, the discrete part of the $(\theta, \x)$-twisted trace formula for $\lG$ takes the form
\begin{align}
\label{eq: twisted spectral side}
\tIdt{\lG^{\theta}}{, t}(\lf) = \sum_{ \{ \lM \} } |W(\lM)|^{-1} \sum_{w \in W^{\theta}(\lM)_{reg}} |\det(w-1)_{\mathfrak{a}^{\lG^{\theta}}_{\lM}}|^{-1} tr(M_{\lP|\theta \lP, t}(w, \lif{\chi}) I^{\theta, \x}_{\lP, t}(\lif{\chi}, \lf)), \quad \lf \in \H(\lG, \lif{\chi}).
\end{align}
We give the explanation of this formula below. The outer sum is taken over $\lG(F)$-conjugacy classes of Levi subgroups $\lM$ of $\lG$, and the inner sum is taken over elements $w$ in the Weyl set
\[
W^{\theta}(\lM) := \Norm(A_{\lM}, \lG \rtimes \theta) / \lM
\]
such that $|\det(w-1)_{\mathfrak{a}^{\lG^{\theta}}_{\lM}}|^{-1} \neq 0$. Here $\mathfrak{a}^{\lG^{\theta}}_{\lM}$ is the kernel of the canonical projection of 
\(
\mathfrak{a}_{\lM} \rightarrow \mathfrak{a}_{\lG} \rightarrow \mathfrak{a}_{\lG^{\theta}}
\) 
for 
\(
\mathfrak{a}_{\lG^{\theta}} := \mathfrak{a}_{\lG} / \{ X - \theta(X): X \in \mathfrak{a}_{\lG}\}.
\)
For any Levi subgroup $\lM$ of $\lG$, we take 
\[
\bigoplus_{\lif{\zeta}_{\lM}} \, L^{2}_{disc, t}(\lM(F) \backslash \lM(\A_{F}), \lif{\zeta}_{\lM}) 
\] 
where the central character $\lif{\zeta}_{\lM}$ extends $\lif{\chi}$ and is invariant under some element of $W^{\theta}(\lM)_{reg}$, and the archimedean infinitesimal characters of the irreducible constituents have norm $t$ on their imaginary parts. Then 
\[
I_{\lP, t}(\lif{\chi}, \lf) = \int_{\lif{Z}_{\A_{F}} \backslash \lG(\A_{F})} \lf(g) I_{\lP, t}(\tilde{\chi}, g) dg
\]
defines an operator on the space $\H_{\lP, t}(\tilde{\chi})$ of the corresponding normalized induced representation. Let $R(\theta): \H_{\theta \lP, t}(\tilde{\chi}) \rightarrow \H_{\lP, t}(\tilde{\chi})$ be induced by the action of $\theta$ on $\lG(\A_{F})$ and $R(\x): \H_{\lP, t}(\tilde{\chi}) \rightarrow \H_{\lP, t}(\tilde{\chi})$ induced by multiplying $\omega$. Finally $I^{\theta, \x}_{\lP, t}(\lif{\chi}, \lf)$ is the composition $R(\theta)^{-1} \circ R(\x) \circ I_{\lP, t}(\lif{\chi}, \lf)$ and $M_{\lP|\theta \lP, t}(w, \lif{\chi})$ is the standard intertwining operator between $\H_{\theta \lP, t}(\lif{\chi})$ and $\H_{\lP, t}(\lif{\chi})$. For the term corresponding to $\lM = \lG$, we denote 
\[
R^{\lG}_{disc, t}(\lf) := I_{\lG, t}(\lif{\chi}, \lf), \quad R^{(\lG^{\theta}, \x)}_{disc, t}(\lf) := R(\theta)^{-1} \circ R(\x) \circ R^{\lG}_{disc, t}(\lf).
\] 

The stabilization of \eqref{eq: twisted spectral side} results from the works of Arthur in the ordinary case \cite{Arthur:2001} \cite{Arthur:2002} \cite{Arthur:2003} and the works of Moeglin-Waldspurger in the twisted case \cite{MW1:2016} \cite{MW2:2016}. To describe it, let us denote by $\End{}{\lG^{\theta}, \x}$ ($\End{ell}{\lG^{\theta}, \x}$) the set of equivalence classes of twisted (elliptic) endoscopic data for $(G, \theta, \x)$. Then the stabilization takes the following form
\begin{align}
\label{eq: twisted endoscopic side}
\tIdt{\lG^{\theta}}{, t}(\lf) = \sum_{\lG' \in \tEnd{ell}{\lG^{\theta}}} \iota(\lG, \lG') \Sdt{\lG'}{, t}(\lf^{\lG'}), \quad \lf \in \H(\lG, \lif{\chi}).
\end{align}
where $\lf^{\lG'}$ is the Langlands-Shelstad-Kottwitz transfer of $\lf$. The coefficients $\iota(\lG, \lG')$ are given as follows 
\begin{align}
\label{eq: stabilization coefficient}
\iota(\lG, \lG') = | \bar{Z}(\D{\lG'})^{\Gal{}} |^{-1} | \Out_{\lG}(\lG') |^{-1} 
\end{align}
(cf. \cite[(5.13)]{Xu:2018}), where 
\[
\bar{Z}(\D{\lG'})^{\Gal{}} = Z(\D{\lG'})^{\Gal{}} Z(\D{\lG})^{\Gal{}} / Z(\D{\lG})^{\Gal{}}, \quad \Out_{\lG}(\lG') = \Aut_{\lG}(\lG') / \D{\lG'}Z(\D{\lG})^{\Gal{}}.
\]

For $\p \in \cP{G}$, we choose $\lif{\chi}$ such that $\chi$ matches the restriction of the central character of $\cPkt{\p}$. Then we define the $\p$-component of \eqref{eq: twisted spectral side} to be
\begin{align*}
\tIdt{\lG^{\theta}}{, \p}(\lf) := \sum_{\lif{c} \rightarrow c(\p)} \tIdt{\lG^{\theta}}{, t(\p), \lif{c}}(\lf), \quad \lf \in \sH(\lG, \lif{\chi})
\end{align*}
which is contributed from automorphic representations $\lr$ of $\lG(\A_{F})$, whose central characters match $\lif{\chi}$ by restriction and Satake parameters $\lif{c}_v$ projected to $c(\p_v)$ and $t(\p)$ is the norm of of the imaginary part of the archimedean infinitesimal character determined by $\p$. It follows from \cite[Lemma 5.1]{Xu:2018} that the stabilization of the $\p$-component of the twisted stable trace formula \eqref{eq: twisted endoscopic side} for $\lG$ is 
\begin{align}
\label{eq: endoscopic side component}
\tIdt{\lG^{\theta}}{, \p}(\lf) = \sum_{\lG' \in \End{ell}{\lG^{\theta}, \x}} \iota(\lG, \lG') \Sdt{\lG'}{, \p}(\lf^{\lG'}),  \quad \lf \in \sH(\lG, \lif{\chi})
\end{align}
where 
\[
\Sdt{\lG'}{, \p}(\lf^{\lG'}) = \sum_{\lif{c}' \rightarrow c(\p)} \Sdt{\lG'}{, t(\p), \lif{c}'}(\lf^{\lG'})
\]
can be defined recursively by using the ordinary stable trace formula.


\section{Statements of main results}
\label{sec: main results}

Let $F$ be a global field of characteristic zero and $G$ be \eqref{eq: product}. The following theorem proves \cite[Conjecture 5.16]{Xu:2018}, which concerns the existence of global $L$-packets of $\lG(\A_{F})$.

\begin{theorem}
\label{thm: global L-packet}
\,
\begin{enumerate}
\item For $\p \in \cP{G}$, there exists a global packet $\cPkt{\lp}$ of $\sH(\lG)$-modules of irreducible admissible representations of $\lG(\A_{F})$  satisfies the following properties:
          \begin{enumerate}
          \item $ \cPkt{\lp} = \bigotimes'_{v} \cPkt{\lp_{v}}$, where $\cPkt{\lp_{v}}$ is some lift of $\cPkt{\p_{v}}$;
          \item there exists $[\lr] \in \cPkt{\lp}$ isomorphic to an automorphic representation of $\lG$ as $\sH(\lG)$-modules.
          \end{enumerate}
Moreover, $\cPkt{\lp}$ is unique up to twisting by characters of $\lG(\A_{F}) / \lG(F)G(\A_{F})$.

\item Suppose $\p \in \cPdt{G}$, we have the following decomposition as $\sH(\lG)$-module 
\begin{align}
\label{formula: discrete spectrum}
L^{2}_{disc, \p}(\lG(F) \backslash \lG(\A_{F}), \lif{\zeta}) = m_{\p} \bigoplus_{\x \in Y / \a(\S{\p})} \bigoplus_{\substack{ [\lr] \in \cPkt{\lp} \otimes \x \\ \langle \cdot, \lr \rangle = 1}} \lr.
\end{align}
where $\lif{\zeta}$ is an extension of the central character of $\cPkt{\p}$.
\end{enumerate}
\end{theorem}

The next theorem proves \cite[Conjecture 5.19]{Xu:2018}, which describes the discrete spectrum under the twists by an outer automorphism $\theta \in \Sigma_0$ and a character $\x \in Y$. 

\begin{theorem}
\label{thm: compatible normalization}
Suppose $\p \in \cPdt{G}$ and $x \in \S{\p}^{\theta}$ with $\a(x) = \x$ for $\theta \in \Sigma_{0}$ and some character $\x$ of $\lG(\A_{F})/\lG(F)G(\A_{F})$. For $[\lr] \in \cPkt{\lp}$ with $\langle \cdot, \lr \rangle =1$, the intertwining operator 
\(
R(\theta)^{-1} \circ R(\x)
\)
restricted to the $\lr$-isotypic component $I(\lr)$ in the discrete spectrum is equal to the product of the multiplicity $m(\lr)$ and the local intertwining operators $A_{\lr_{v}}(\theta, \x_{v})$ normalized by $x_{v}$, i.e.
\begin{align}
\label{formula: theta twisted discrete spectrum}
\tIdt{\lG^{\theta}}{, \p}(\lf) = m_{\p} \sum_{\x' \in Y / \a(\S{\p})} \sum_{\substack{[\lr] \in \cPkt{\lp} \otimes \x' \\ \langle \cdot, \lr \rangle = 1}} \lf_{\lG^{\theta}}(\lr, \x),  \,\,\,\,\, \lf \in \sH(\lG, \lif{\chi}),
\end{align}
where $ \lf_{\lG^{\theta}}(\lr, \x) = \prod_{v} \lf_{\lG^{\theta}_{v}}(\lr_{v}, \x_{v})$, and it does not depend on the choice of $x$ in the $\S{\lp}$-coset.
\end{theorem}

Both theorems have been proved in \cite{Xu:2018} for $\p = \p_{1} \times \p_{2} \times \cdots \times \p_{q} \in \cPdt{G}$ with $\p_{i} \in \cPdt{G_{i}}$ such that $\S{\lp_{i}} = 1$ for all $i$. The next result concerns the functoriality of twisted endoscopic transfer, which can be formulated as follows. For $\p \in \cP{G}$, $\theta \in \Sigma_0$ and any semisimple element $s \in \cS{\p}^{\theta}$, we can define the relation $(G', \p') \rightarrow (\p, s)$ as in the local case. Define
\[
{\rm Tran} \, \cPkt{\lp'} := \otimes'_{v} \, {\rm Tran} \, \cPkt{\lp'_v}
\]
to be the transfer of a global $L$-packet $\cPkt{\lp'}$ with respect to the endoscopic embedding $\tilde{\xi}$ (cf. \eqref{eq: lift endoscopic embedding}).

\begin{conjecture}
\label{conj: functoriality}
There exists a global $L$-packet $\cPkt{\lp}$ such that
\[
{\rm Tran} \, \cPkt{\lp'} = \cPkt{\lp}.
\] 
\end{conjecture}

We are not able to prove this conjecture in all cases. This is due to the occurrence of what we call $GL$-type orthogonal simple parameters. They are defined as follows.

\begin{definition}
\label{def: GL-type}
For $\p \in \Psm{N}$ such that $\p = \p^{\vee}$. We call $\p$ is of {\bf $GL$-type} if $N$ is even and at every place $v$,  $\p_{v}$ factors through the $L$-group of a Levi subgroup $M_{\p, v}$ of $G_{\p, v}$, which is a product of general linear groups.
\end{definition}

Suppose $G$ is symplectic or special even orthogonal and $\phi \in \cP{G}$, we will decompose the parameter as 
\[
\phi = \phi_{o} \boxplus \phi_b
\]
where $\phi_{b}$ consists of all $GL$-type orthogonal simple parameters in $\phi$. The parameter $\phi$ factors through $\phi_{H} := \phi_o \times \phi_b$ for the endoscopic group
\[
H := G_o \times G_b.
\]
The embedding $\L{H} \hookrightarrow \L{G}$ induces an isomorphism
\[
S^{\Sigma_0}_{\phi_{o}} \times S_{\p_b} \cong S^{\Sigma_0}_{\phi}.
\]

\begin{theorem}
\label{thm: functoriality}
For $\p \in \cP{G^{\theta}}$ and semisimple $s \in (S^{\theta}_{\p_{o}} \times Z(\D{G}_b))/Z(\D{G})^{\Gal{}}$, let $(G', \p') \rightarrow (\p, s)$ and $\lG' \in \tEnd{}{\lG^{\theta}}$ be the lift of $G'$, then there exist global packets $\cPkt{\lp'}$ and $\cPkt{\lp}$ such that
\[
{\rm Tran} \, \cPkt{\lp'} = \cPkt{\lp}.
\] 
\end{theorem}

This result is closely related to the stable multiplicity formula conjectured by Arthur.

\begin{conjecture}[Stable multiplicity formula]
\label{conj: stable multiplicity formula}
Suppose $\p \in \cP{G}$, then 
\begin{align}
\label{formula: stable multiplicity}
\Sdt{\lG}{, \p}(\lf) = m_{\p} \sum_{\x \in Y / \a(\S{\p})} |\S{\lp}|^{-1} \sigma( \com[0]{\cS{\p}}) \lf^{\lG} (\lp \otimes \x), \,\,\,\,\,\,\, \lf  \in \sH(\lG, \lif{\chi}),
\end{align}
where $\sigma(\com[0]{\cS{\p}})$ is some constant defined through \eqref{eq: combinatorial identity} and
\(
\lf^{\lG} (\lp \otimes \x) := \prod_{v} \lf_{v}(\lp_{v} \otimes \x_{v}).
\)
\end{conjecture}

In this paper, we will only consider the case that $\phi_b \in \cPdt{G_b}$. For any unordered $m$-partition
\[
\phi_{b} = \phi_{b, 1} \boxplus \cdots \boxplus \phi_{b, m}, \quad \quad \phi_{b, i} \in \cPdt{G_{b, i}}
\]
the parameter $\phi_{b}$ factors through 
\(
\phi'_b := \phi_{b, 1} \times \cdots \times \phi_{b, m}
\) 
for $G' := G_{b, 1} \times \cdots \times G_{b, m}$ by a sequence of endoscopic embeddings
\[
\D{G}_{b, 1} \times \cdots \times \D{G}_{b, m} \xrightarrow{\xi_{m}} \cdots \xrightarrow{\xi_{1}} \D{G}_{b}
\]
We can lift each endoscopic embedding to that of similitude groups. For fixed $\cPkt{\lp_{b, i}}$, let us define $\cPkt{\lp'_b} := \cPkt{\lp_{b,1}} \tilde{\otimes} \cdots \tilde{\otimes} \,\, \cPkt{\lp_{b, m}}$ and
\[
\cPkt{}(\lp_{b, 1}, \cdots, \lp_{b, m}) := {\rm Tran}_{\lif{\xi}_1} \circ \cdots \circ {\rm Tran}_{\lif{\xi}_m} \, \cPkt{\lp'_b}
\]
We also define
\[
\cPkt{}(\lp_o; \lp_{b, 1}, \cdots, \lp_{b, m}) := {\rm Tran} \, \cPkt{\lp_o} \tilde{\otimes} \,\, \cPkt{}(\lp_{b, 1}, \cdots, \lp_{b, m})
\]
and
\[
\lf(\lp_o; \lp_{b, 1}, \cdots, \lp_{b, m}) := \sum_{\lr \in \cPkt{}(\lp_o; \lp_{b, 1}, \cdots, \lp_{b, m})} \lf(\lr)
\]
\[
\lf(\p_o; \p_{b, 1}, \cdots, \p_{b, m}) := \sum_{\omega \in Y} \, \sum_{\lr \in \cPkt{}(\lp_o; \lp_{b, 1}, \cdots, \lp_{b, m})} \lf(\lr \otimes \omega)
\]
At last, we need to define some constants by the recursive formula 
\begin{align*}
a_1 = 1, \quad \quad 
a_m  = \frac{1}{4} \sum_{k = 1}^{m-1} \begin{pmatrix} m \\ k \end{pmatrix} a_{k}a_{m-k} \quad (m > 1).
\end{align*}

\begin{theorem}
\label{thm: stable multiplicity formula}

Suppose $\p \in \cP{G}$ and $\p_b \in \cPdt{G_b}$.

\begin{enumerate}

\item If $\p = \p_b$, then

\begin{align}
\label{formula: discrete part GL-type}
\Idt{\lG_b}{, \p_b}(\lf) = m_{\p_b}  \lf(\p_{b}), \,\,\,\,\,\,\, \lf  \in \sH(\lG, \lif{\chi}),
\end{align}

\begin{align}
\label{formula: stable multiplicity GL-type}
\Sdt{\lG_b}{, \p_b}(\lf) = m_{\p_b} \sum_{\p_{b, 1}, \cdots, \p_{b, m}} (-1)^{m - 1} a_m \, \lf(\p_{b, 1}, \cdots, \p_{b, m}), \,\,\,\,\,\,\, \lf  \in \sH(\lG, \lif{\chi}),
\end{align}
where the sum is over unordered partition of $\p$.

\item If $\p$ contains at most one GL-type orthogonal simple parameter, then

\begin{align}
\label{formula: stable multiplicity}
\Sdt{\lG}{, \p}(\lf) = m_{\p} \sum_{\x \in Y / \a(\S{\p})} |\S{\lp}|^{-1} \sigma( \com[0]{\cS{\p}}) \lf^{\lG} (\lp \otimes \x), \,\,\,\,\,\,\, \lf  \in \sH(\lG, \lif{\chi}).
\end{align}

\item If $\p \neq \p_o, \p_b$, then 

\begin{align}
\label{formula: stable multiplicity mixed}
\Sdt{\lG}{, \p} = \frac{1}{4} \, {\rm Tran} \, (\Sdt{\lG_o}{, \p_o} \, \tilde{\otimes} \, \Sdt{\lG_b}{, \lp_b}),
\end{align}
where
\begin{align}
\label{formula: stable multiplicity GL-type 1}
\Sdt{\lG_b}{, \lp_b}(\lf) = m_{\p_b} \sum_{\p_{b, 1}, \cdots, \p_{b, m}} (-1)^{m - 1} a_m \, \lf(\lp_{b, 1}, \cdots, \lp_{b, m}), \,\,\,\,\,\,\, \lf  \in \sH(\lG, \lif{\chi}).
\end{align}

\end{enumerate}

\end{theorem}

We can generalize Theorem~\ref{thm: stable multiplicity formula} for $G = G_1 \times \cdots \times G_q$ and $\p = \p_1 \times \cdots \times \p_q \in \cP{G}$ with $\p_{i, b} \in \cPdt{G_{i, b}}$. The parameter $\phi$ factors through $\phi_{H} := \phi_o \times \phi_b$ for the endoscopic group
\(
H := G_o \times G_b,
\)
where 
\[
G_o = G_{1, o} \times \cdots \times G_{q, o}, \quad \quad G_b = G_{1, b} \times \cdots \times G_{q, b},
\]
and 
\[
\p_o = \p_{1, o} \times \cdots \times \p_{q, o}, \quad \quad \p_b = \p_{1, b} \times \cdots \times \p_{q, b}.
\]

\begin{theorem}
\label{thm: stable multiplicity formula product}

Suppose $G = G_1 \times \cdots \times G_q$ as \eqref{eq: product}, $\p \in \cP{G}$ and $\p_b \in \cPdt{G_b}$.

\begin{enumerate}

\item If $\p = \p_b$, then

\begin{align}
\label{formula: discrete part GL-type product}
\Idt{\lG_b}{, \p_b}(\lf) = m_{\p_b}  \lf(\p_{b}), \,\,\,\,\,\,\, \lf  \in \sH(\lG, \lif{\chi}),
\end{align}
and
\begin{align}
\label{formula: stable multiplicity GL-type product}
\Sdt{\lG}{, \p} = (\tilde{\otimes}_{i \neq 1} \Sdt{\lG_i}{, \lp_i}) \, \tilde{\otimes} \, \Sdt{\lG_1}{, \p_1}.
\end{align}
where $\Sdt{\lG_i}{, \lp_i}$ is defined as in \eqref{formula: stable multiplicity} and \eqref{formula: stable multiplicity mixed} by dropping the sum over $\omega$.

\item If each $\p_i$ contains at most one GL-type orthogonal simple parameter, then

\begin{align}
\label{formula: stable multiplicity product}
\Sdt{\lG}{, \p}(\lf) = m_{\p} \sum_{\x \in Y / \a(\S{\p})} |\S{\lp}|^{-1} \sigma( \com[0]{\cS{\p}}) \lf^{\lG} (\lp \otimes \x), \,\,\,\,\,\,\, \lf  \in \sH(\lG, \lif{\chi}).
\end{align}

\item If $\p \neq \p_o, \p_b$, then 

\begin{align}
\label{formula: stable multiplicity mixed product}
\Sdt{\lG}{, \p} = (\frac{1}{4})^{l} \, {\rm Tran} \, (\Sdt{\lG_o}{, \p_o} \, \tilde{\otimes} \, \Sdt{\lG_b}{, \lp_b}),
\end{align}
where
\(
l = \sharp \{1 \leqslant i \leqslant q \, | \, \p_i \neq \p_{i, o}, \p_{i, b}\}
\)
and
\begin{align}
\label{formula: stable multiplicity GL-type product 1}
\Sdt{\lG_{b}}{, \lp_{b}} = \tilde{\otimes}_{i} \, \Sdt{\lG_{i, b}}{, \lp_{i, b}} \, .
\end{align}

\end{enumerate}

\end{theorem}


\section{Comparison formulas}
\label{sec: comparison formula}

We would like to expand both sides of \eqref{eq: endoscopic side component} in terms of local objects. First we need to introduce some general constructions from \cite{Arthur:2013}. Suppose $S$ is a connected complex reductive group with an automorphism $\theta$, we denote $S^{\theta} = S \rtimes \theta$, which can be viewed as a connected component of the complex reductive group $S^{+} := S \rtimes \langle \theta \rangle$. We fix a maximal torus $T$ of $S$, and define the Weyl set
\[
W^{\theta}(S) = \Norm(T, S^{\theta}) / T.
\]
Let $W^{\theta}(S)_{reg}$ be the set of Weyl elements $w$ such that 
\(
\det(w-1)|_{\mathfrak{a}_{T}} \neq 0.
\) 
Moreover, let $s^{0}(w)$ denote the sign $(-1)^{n}$, where $n$ is the number of positive roots of $(S, T)$ mapped by $w$ to negative roots. Now we can assign to $S^{\theta}$ a real number
\[
i^{\theta}(S) := |W(S)|^{-1} \sum_{w \in W^{\theta}_{reg}(S)} s^{0}(w) |\det(w - 1)|_{\mathfrak{a}_{T}}^{-1},
\]
where $W(S)$ is the Weyl group of $S$. Let us denote the set of semisimple elements of $S^{\theta}$ by $S^{\theta}_{ss}$. And for any $s \in S^{\theta}_{ss}$, we write $S_{s} = \Cent(s, S)$. Let 
\(
S^{\theta}_{ell} = \{s \in S^{\theta}_{ss} : |Z(S_{s})| < \infty \}
\)
and $\mathcal{E}^{\theta}_{ell}(S)$ the $S$-conjugacy classes in $S^{\theta}_{ell}$. Arthur proved (\cite{Arthur:2013}, Proposition 4.1.1) that there exist unique constants $\sigma(S_{1})$ defined for all connected complex reductive groups $S_{1}$ satisfying
\[
\sigma(S_{1}) = \sigma(S_{1} / Z_{1}) |Z_{1}|^{-1}
\]
for any central subgroup $Z_{1}$ of $S_{1}$, such that the following equality holds
\begin{align}
\label{eq: combinatorial identity}
\sum_{s \in \mathcal{E}^{\theta}_{ell}(S)} |\pi_{0}(S_{s})|^{-1} \sigma((S_{s})^{0}) = i^{\theta}(S)
\end{align}
for any pairs $(S, \theta)$. 

Let us assume Theorem~\ref{thm: global L-packet}, ~\ref{thm: compatible normalization}, ~\ref{thm: stable multiplicity formula}, ~\ref{thm: stable multiplicity formula product} hold for the proper Levi subgroups and proper twisted endoscopic groups of $\lG$. For $\p \in \cP{G} - \cPdt{G}$, let $\bar{T}_{\p}$ be a maximal torus of $\cS{\p}^{0}$. Then $\p$ factors through $\p_{M} \in \cPdt{M}$ for a Levi subgroup $M$ of $G$ such that $\D{M} := {\rm Cent}(\bar{T}_{\p}, \D{G})$. For $\theta \in \Sigma_0$, let $\cS{\p}^{+}$ be generated by $\cS{\p}$ and $\cS{\p}^{\theta}$. Define
\[
\N{\p}^{\theta}(G) = \Norm(\cT{\p}, \cS{\p}^{\theta}) / \Cent(\cT{\p}, \com[0]{\cS{\p}})^{0}, \quad \N{\p}^{+}(G) = \Norm(\cT{\p}, \cS{\p}^{+}) / \Cent(\cT{\p}, \com[0]{\cS{\p}})^{0}, 
\]
and
\[
W^{+}_{\p} = W(\bar{T}_{\p}, \cS{\p}^{+}), \quad W^{0}_{\p} = W(\bar{T}_{\p}, \cS{\p}^{0}), \quad R^{+}_{\p} = W^{+}_{\p}/W^{0}_{\p}.
\] 
They fit into the following diagram.
\begin{align} 
\label{eq: theta twisted intertwining relation diagram}
\xymatrix @C=0.5cm @R=0.5cm{
             &                                                 & 1   \ar[d]                                        & 1  \ar[d]                                        &    \\
             &                                                 & \com[0]{W_{\p}}    \ar[d]   \ar@{=}[r]  & \com[0]{W_{\p}}   \ar[d]      &     \\
1 \ar[r]  & \S{\p_{M}}    \ar@{=}[d]     \ar[r]  & \N{\p}^{+}    \ar[d]    \ar[r]   & W_{\p}^{+}     \ar[d] \ar[r]   & 1  \\
1 \ar[r]  & \S{\p_{M}}   \ar[r]                       & \S{\p}^{+}    \ar[d]     \ar[r]  & R_{\p}^{+}     \ar[d] \ar[r]    & 1   \\
            &                                                  & 1                                                  & 1                                                   & }  
\end{align}
For $[\lr_{M}] \in \cPkt{\lp_{M}}$, 
\[
\mathcal{I}_{\lP}(\lr_{M}, \lf) := \int_{\lif{Z}_{\A_{F}} \backslash \lG(\A_{F})} \lf(g) \mathcal{I}_{\lP}(\lr_{M}, g) dg
\]
defines an operator on the space $\mathcal{H}_{\lP}(\lr_M)$ of normalized induced representation of $\lr_{M}$. For $\omega \in Y$, let
\[
\mathcal{I}_{\lP}^{\theta, \x}(\lr_{M} \otimes \x^{-1}, \lf) = R(\theta)^{-1} \circ R(\x) \circ \mathcal{I}_{P}(\lr_{M} \otimes \x^{-1}, \lif{f}).
\]
For $u \in \N{\p}^{\theta}$, let $w_u$ be the image of $u$ in $W^{+}_{\p}$ and $\theta_{u}$ a representative of $w_{u}$ in $G(F) \rtimes \theta$ preserving the $F$-splitting of $M$. Then $u$ gives an element in $\S{\p_{M}}^{\theta_u}$. We define
\[
R_{P|\theta P}(\theta_{u}, \r_{M}^{+}, \p) := \bigotimes_{v} \, R_{P_v|\theta P_v}(\theta_{u_{v}}, \r_{M_v}^{+}, \p_v)
\]
to be the normalized intertwining operator between $\mathcal{H}_{\theta P}(\r_{M}^{\theta^{-1}})$ and $\mathcal{H}_{P}(\r_{M})$. After identifying
\begin{align*}
\mathcal{H}_{\theta \lP}(\lr_{M}^{\theta^{-1}}) & \cong \bigoplus_{\r_{M} \subseteq \Res \, \lr_{M}} \mathcal{H}_{\theta P}(\r_{M}^{\theta^{-1}}) \\
\mathcal{H}_{\lP}(\lr_{M} \otimes \x^{-1}) & \cong \bigoplus_{\r_{M} \subseteq \Res \, \lr_{M}} \mathcal{H}_{P}(\r_{M})
\end{align*}
as $G(\mathbb{A}_F)$-representations, we define the normalized intertwining operator between $\mathcal{H}_{\theta \lP}(\lr_{M}^{\theta^{-1}})$ and $\mathcal{H}_{\lP}(\lr_{M} \otimes \x^{-1})$ to be 
\[
R_{\lP | \theta \lP}(u, \lr_{M}, \lp) := \bigoplus_{\r_{M} \subseteq \Res \, \lr_{M}} \langle u, \r_{M}^{+} \rangle R_{P|\theta P}(\theta_{u}, \r_{M}^{+}, \p).
\]
If $\lr_{M}$ is a discrete automorphic representation, then it follows from Theorem~\ref{thm: compatible normalization} and analogous result for general linear groups (cf. \cite[Lemma 4.2.3]{Arthur:2013}) that
\[
R_{\lP|\theta \lP}(u, \lr_{M}, \lp) = r_{P}(w_u, \p_{M})^{-1} M_{\lP|\theta \lP}(w_u, \lr_{M})
\]
where 
\[
r_{P}(w, \p_{}) = L(0, \r_{\p_{M}}, \rho^{\vee}_{w_{u}^{-1}P | P} ) \varepsilon(0, \r_{\p_{M}}, \rho^{\vee}_{w_{u}^{-1}P | P}, \psi_{F} )^{-1} L(1, \r_{\p_{M}}, \rho^{\vee}_{w_{u}^{-1}P | P})^{-1}
\]
and $\rho^{\vee}_{w_{u}^{-1}P | P}$ is the contragredient of the adjoint representation of $\L{M}$ over ${\rm Lie}(w_{u}^{-1}\D{N}_{P}) \, \cap \, {\rm Lie}(\D{N}_{P}) \backslash {\rm Lie}(w_{u}^{-1} \D{N}_{P})$. In \cite[5.4.1]{Xu:2018} we have incorrectly written the normalizing factor as a product of local normalizing factors. It should be obtained by analytic continuation of the product. This following lemma is \cite[Lemma 5.22]{Xu:2018}.

\begin{lemma}
\label{lemma: twisted spectral expansion}
Suppose $\p \in \cP{G} - \cPdt{G}$, $\theta \in \Sigma_{0}$ and $\x \in Y$, then
\begin{align}
\label{eq: twisted spectral expansion}
\tIdt{\lG^{\theta}}{, \p} (\lf) = C_{\lp} \sum_{\x' \in Y / \a(\S{\p})} \sum_{x \in \S{\p}^{\theta}(\x)} i^{\theta}_{\p}(x) \lf_{\lG^{\theta}}(\lp \otimes \x', x),     \,\,\,\,\, \lf \in \sH(\lG, \lif{\chi}).          
\end{align}
where $C_{\lp} = m_{\p} |\S{\lp}|^{-1}$ and
\[
\lf_{\lG^{\theta}}(\lp \otimes \x', x) = \sum_{[\lr_{M}] \in \cPkt{\lp_{M}} \otimes \x'} tr (R_{\lP|\theta \lP}(u, \lr_{M}, \lp) I^{\theta, \x}_{\lP}(\lr_{M} \otimes \x^{-1}, \lf)),
\]
and
\(
i^{\theta}_{\p}(x) := i^{s}(\cS{\p}^{0})
\)
for any semisimple $s \in \cS{\p}^{\theta}$.
\end{lemma}

For the expansion of the right hand side of \eqref{eq: endoscopic side component}, we introduce the following notations. Let $\p = \p_1 \times \cdots \times \p_q \in \cP{G}$ such that $\p_b \in \cPdt{G_b}$. Let
\[
J = \{1 \leqslant i \leqslant q \, | \, \p_i = \p_{i, b} \}, \quad \quad K = \{1 \leqslant i \leqslant q \, | \, \p_i \neq \p_{i, b}, \p_{i, o} \}
\]
and
\[
\p_{J} = \prod_{i \in J} \p_i, \quad \quad \p_{K} = \prod_{i \in K} \p_i.
\]
The embedding $\L{H} \hookrightarrow \L{G}$ induces an isomorphism
\[
S^{\Sigma_0}_{\phi_{o}} \times S_{\p_b} \cong S^{\Sigma_0}_{\phi}
\]
where
\[
S^{\Sigma_0}_{\phi_{o}} = S^{\Sigma_{1, 0}}_{\phi_{1, o}} \times \cdots \times S^{\Sigma_{q, 0}}_{\phi_{q, o}}, \quad \quad S_{\phi_{b}} = S_{\phi_{1, b}} \times \cdots \times S_{\phi_{q, b}} \,.
\]
This gives a projection 
\begin{align}
\label{eq: projection}
\S{\p}^{\Sigma_0} \rightarrow \S{\p_o}^{\Sigma_0}.
\end{align}
Let 
\[
\cPkt{}(\lp_o; \lp_b) = {\rm Tran} \, \cPkt{\lp_o} \tilde{\otimes} \, \cPkt{\lp_b}.
\]
For $x_{o} \in \mathcal{S}^{\theta}_{\phi_o}(\omega)$, we define
\[
\lf_{\lG^{\theta}}' (\lp_o \otimes \x', x_o; \lp_b) := \lf_{\lG^{\theta}}' (\lp \otimes \x', x) = \lf^{\lG'}(\lp' \otimes \x')
\]
for any 
\[
x \in (S^{\theta}_{\p_{o}}(\omega) \times Z(\D{G}_b))/S^{0}_{\p_o}Z(\D{G})^{\Gal{}} \subseteq \S{\p}^{\theta}(\omega)
\] 
projecting to $x_o$. 
Let
\begin{align}
\label{eq: coefficient}
e'^{\theta}_{\p_o}(x_o) = \begin{cases}
                                      i^{\theta}_{\p_o}(x_o)  & \text{ if } x_o \neq 1, \\
                                      i^{\theta}_{\p_o}(x_o) - \sigma(\com[0]{\cS{\p_o}}) & \text{ if } x_o = 1.
                                      \end{cases}
\end{align}

\begin{lemma}
\label{lemma: twisted endoscopic expansion}
Suppose $\p \in \cP{G}$ such that $\p_b \in \cPdt{G_b}$ and $\p \neq \p_b$. Let $\theta \in \Sigma_{0}$ and $\x \in Y$. 

\begin{enumerate}

\item

If $\theta =id, \x = 1$, then there are two cases.

\begin{enumerate}

\item $\p = \p_o$:

\begin{align}
\label{eq: endoscopic expansion}
\Idt{\lG}{, \p}(\lf) - \Sdt{\lG}{, \p} (\lf) = C_{\lp} \sum_{\x' \in Y / \a(\S{\p})} \sum_{x \in \S{\lp}} e'_{\p}(x) \lf_{\lG}' (\lp \otimes \x', x), \,\,\,\,\, \lf \in \sH(\lG, \lif{\chi}). 
\end{align}

\item $\p \neq \p_o$:

\begin{align}
\label{eq: endoscopic expansion 1}
\Idt{\lG}{, \p}(\lf) - \Sdt{\lG}{, \p} (\lf) & = m_{\p_{J}} C_{\lp_o} \sum_{\x' \in Y / \a(\S{\p_o})} \sum_{x_o \in \S{\lp_o}} e'_{\p_o}(x_o) \lf_{\lG}' (\lp_o \otimes \x', x_o; \lp_b) \\
& +  {\rm Tran} \, \Sdt{\lG_o}{, \p_o} \tilde{\otimes} \, \Big(\frac{m_{\p_{J}}}{m_{\p_{b}}} \, \Idt{\lG_b}{, \lp_b} - (\frac{1}{4})^{|K|} \, \Sdt{\lG_b}{, \lp_b} \Big), \,\,\,\,\, \lf \in \sH(\lG, \lif{\chi})
\end{align}
where
\[
\Idt{\lG_b}{, \lp_b}(\lf) = m_{\p_b} \lf(\lp_b).
\]

\end{enumerate}

\item 

If $\theta \neq id$ or $\x \neq 1$, then

\begin{align}
\label{eq: twisted endoscopic expansion}
\tIdt{\lG^{\theta}}{, \p} (\lf) = m_{\p_{J}} C_{\lp_o} \sum_{\x' \in Y / \a(\S{\p_o})} \sum_{x_o \in \S{\p_o}^{\theta}(\x)} e'^{\theta}_{\p_o}(x_o) \lf_{\lG^{\theta}}' (\lp_o \otimes \x', x_o; \lp_b), \,\,\,\,\, \lf \in \sH(\lG, \lif{\chi}). 
\end{align}

\end{enumerate}
\end{lemma}

This reduces to \cite[Lemma 5.25]{Xu:2018} when each $\p_i$ has at most one $GL$-type simple orthogonal parameter. The remaining case that $\p = \p_b$ will be treated in the next section. To compare \eqref{eq: endoscopic expansion 1}, \eqref{eq: twisted endoscopic expansion} with \eqref{eq: twisted spectral expansion}, we note that $\lf_{\lG^{\theta}}(\lp \otimes \x', x)$ only depends on the image $x_o$ of $x$ under \eqref{eq: projection} and $i^{\theta}_{\p}(x) = i^{\theta}_{\p_o}(x_o)$. So we can rewrite \eqref{eq: twisted spectral expansion} as
\begin{align}
\label{eq: twisted spectral expansion 1}
\tIdt{\lG^{\theta}}{, \p} (\lf) = m_{\p_J}C_{\lp_o} \sum_{\x' \in Y / \a(\S{\p_o})} \sum_{x_o \in \S{\p_o}^{\theta}(\x)} i^{\theta}_{\p_o}(x_o) \lf_{\lG^{\theta}}(\lp \otimes \x', x),     \,\,\,\,\, \lf \in \sH(\lG, \lif{\chi}).          
\end{align}
Moreover,
\begin{align}
\label{eq: coefficient}
i^{\theta}_{\p_o}(x) - e'^{\theta}_{\p_o}(x_o) = \begin{cases}
                                      0 & \text{ if } x_o \neq 1, \\
                                      \sigma(\com[0]{\cS{\p_o}}) & \text{ if } x_o = 1.
                                      \end{cases}
\end{align}
If we further assume Theorem~\ref{thm: functoriality},  then by the local intertwining relation (cf. \cite[Theorem 4.12]{Xu:2018}) we get
\begin{align}
\label{eq: global intertwining relation}
\lf_{\lG^{\theta}}' (\lp_o \otimes \x', x_o; \lp_b) = \lf_{\lG^{\theta}}(\lp \otimes \x', x).
\end{align}
This is the global intertwining relation for $\lG$.


\section{$GL$-type orthogonal parameters}
\label{sec: GL-type}

In this section we would like to prove Theorem~\ref{thm: global L-packet}, ~\ref{thm: stable multiplicity formula}, ~\ref{thm: stable multiplicity formula product} for $\p = \p_b \in \cPdt{G}$. Theorem~\ref{thm: compatible normalization} is irrelevant in this case. We can assume $G = G_1 \times \cdots \times G_q$, where $G_i = SO(2n_i)$. Let $\p = \p_1 \times \cdots \times \p_{q}$, where $\p_i \in \cPdt{G_i}$ and $\p_{i, b} = \p_i$. Then $\S{\lp} = \S{\p} = \prod_{i=1}^{q} \S{\p_{i}} = \prod_{i=1}^{q} \S{\lp_{i}}$. For $x = (x_i) \in \S{\lp}$, $\p$ factors through $\p'_x$ for the elliptic endoscopic group $G'_x$. The following lemma treats the remaining case left by Lemma~\ref{lemma: twisted endoscopic expansion}.

\begin{lemma}
\label{lemma: endoscopic side GL-type}
\[
\Idt{\lG}{, \p}(\lf) - \Sdt{\lG}{, \p} (\lf) = \sum_{1 \neq x \in \S{\lp}} (m_{\p} / m_{\p'_x})^2 \, \Sdt{\lG'_x}{, \p'_x}(\lf'), \quad \quad \lf \in \sH(\lG, \lif{\chi}).
\]
\end{lemma}

\begin{proof}
We follow the same strategy in \cite[Section 5.4.2]{Xu:2018}. Recall
\[
\Idt{\lG}{, \p}(\lf) - \Sdt{\lG}{, \p} (\lf) = \sum_{\lG' \in \End{ell}{\lG} - \{\lG\} } \iota( \lG, \lG' ) \Sdt{\lG'}{, \p}(\lf^{\lG'}) , \quad \quad \lf \in \sH(\lG, \lif{\chi}),
\]
where
\[
\iota(\lG, \lG') = | \bar{Z}(\D{\lG'})^{\Gal{}} |^{-1} | \Out_{\lG}(\lG') |^{-1}.
\]
First note
\[
|\bar{Z}(\D{\lG'})^{\Gal{}}| = |\bar{Z}(\D{G}')^{\Gal{}}| = m_{\p'}/m_{\p}, \quad \Out_{\lG}(\lG') = \Out_{G}(G').
\]
Secondly, we can turn the right hand side into a sum over 
\(
\p_{G} \in \P{G, \p}
\)
and
\[
\Big\{ (\lG', \p') : \lG' \in \End{ell}{\lG} - \{\lG\} \text{ and } \p' \in \P{G', \p_{G}} \Big\}
\]
where $\P{G', \p_{G}}$ is the set $\D{G}'$-conjugacy classes of $L$-homomorphisms $\p' : \mathcal{L}_{\p} \rightarrow \L{G}'$ such that $\xi' \circ \p'$ is $\D{G}$-conjugate to $\p_{G}$. The contribution from each $\p_{G}$-summand is the same, so we get a constant multiple $|\P{G, \p} | = m_{\p}$. Moreover, the contribution of $(\lG', \p')$ only depends on its image under 
\[
\Big\{ (\lG', \p') : \lG' \in \End{ell}{\lG} - \{ \lG \}  \text{ and } \p' \in \P{G', \p_{G}} \Big\} \longrightarrow \S{\p} - \{1\},
\]
which is
\(
m^{-1}_{\p'_x} \, \Sdt{\lG'_x}{, \p'_x}(\lf').
\)
The fiber has size
\(
| \Out_{G}(G'_x) | | S_{\p, x} / S_{\p, x} \cap \D{G}'_x Z(\D{G})^{\Gal{}} |^{-1}.
\)
Since $\p = \p_b$, we have $S_{\p, x} = S_{\p} \subseteq \D{G}'_x$. Hence $| S_{\p, x} / S_{\p, x} \cap \D{G}'_x Z(\D{G})^{\Gal{}} | = 1$. In sum, we get
\begin{align*}
\Idt{\lG}{, \p}(\lf) - \Sdt{\lG}{, \p} (\lf) = \sum_{1 \neq x \in \S{\lp}} (m_{\p} / m_{\p'_x})^2 \, \Sdt{\lG'_x}{, \p'_x}(\lf').
\end{align*}

\end{proof}

Before starting the proof of the main theorems, we make the following observations. Theorem~\ref{thm: global L-packet} implies \eqref{formula: discrete part GL-type} \eqref{formula: discrete part GL-type product}. So for Theorem~\ref{thm: stable multiplicity formula}, ~\ref{thm: stable multiplicity formula product}, it remains to show \eqref{formula: stable multiplicity GL-type} and \eqref{formula: stable multiplicity GL-type product}. At last, Theorem~\ref{thm: global L-packet} follows from the special case that $G = SO(2n)$. So now we assume Theorem~\ref{thm: global L-packet} holds for $G = SO(2n)$ with $n < N$ and $\p = \p_b \in \cPdt{G}$. In particular, we will fix $\cPkt{\lp}$. For $G = G_1 \times \cdots \times G_q$ such that $n_i < N$, it follows that Theorem~\ref{thm: global L-packet} and \eqref{formula: discrete part GL-type product} holds for $G$. We define recursively
\[
\Sdt{\lG}{, \lp} (\lf) = \Idt{\lG}{, \lp}(\lf) - \sum_{1 \neq x \in \S{\lp}} (m_{\p} / m_{\p'_x})^2 \, \Sdt{\lG'_x}{, \lp'_x}(\lf'), \quad \quad \lf \in \sH(\lG, \lif{\chi}).
\]
We will show that it admits a tensor product decomposition as \eqref{formula: stable multiplicity GL-type product 1}, and matches \eqref{formula: stable multiplicity GL-type 1} in case $G = SO(2n)$. By induction on $|\S{\lp}|$, one can easily show
\begin{align}
\label{eq: stable decomposition}
\Sdt{\lG}{, \p} (\lf) = \sum_{\omega \in Y} \Sdt{\lG}{, \lp} (\lf \otimes \omega).
\end{align}

\begin{lemma}
\label{lemma: product formula}

\begin{align}
\label{eq: product formula 1}
\Sdt{\lG}{, \lp} = \tilde{\otimes}_{i} \, \Sdt{\lG_i}{, \lp_i}
\end{align}
\begin{align}
\label{eq: product formula}
\Sdt{\lG}{, \p} = (\tilde{\otimes}_{i \neq 1} \Sdt{\lG_i}{, \lp_i}) \, \tilde{\otimes} \, \Sdt{\lG_1}{, \p_1}
\end{align}
\end{lemma}

\begin{proof}
In view of \eqref{eq: stable decomposition}, \eqref{eq: product formula} follows from \eqref{eq: product formula 1}. We can write the right hand side of \eqref{eq: product formula 1} as
\begin{align*}
\tilde{\otimes}_{i} \, \Sdt{\lG_i}{, \lp_i} = \tilde{\otimes}_i \, \Idt{\lG_i}{, \lp_i} - \sum_{1 \neq x \in \S{\lp}} (m_{\p} / m_{\p'_x})^2 \, {\rm Tran} \, (\tilde{\otimes}_i \, \Sdt{\lG'_{x_i}}{, \lp'_{x_i}}).
\end{align*}
By induction on $|\S{\lp}|$, we can assume
\[
\Sdt{\lG'_x}{, \lp'_x} = \tilde{\otimes}_i \, \Sdt{\lG'_{x_i}}{, \lp'_{x_i}}\,.
\]
So it suffices to know
\[
\Idt{\lG}{, \lp} = \tilde{\otimes}_i \, \Idt{\lG_i}{, \lp_i}
\]
which follows easily from \eqref{formula: discrete part GL-type product}.
\end{proof}

\begin{lemma}
\label{lemma: stable expansion}
For $G = SO(2n)$,
\begin{align}
\label{eq: stable expansion 1}
\Sdt{\lG}{, \lp}(\lf) = m_{\p} \sum_{\p_{1}, \cdots, \p_{m}} (-1)^{m - 1} a_m \, \lf(\lp_{1}, \cdots, \lp_{m}), \,\,\,\,\,\,\, \lf  \in \sH(\lG, \lif{\chi}),
\end{align}
\begin{align}
\label{eq: stable expansion}
\Sdt{\lG}{, \p}(\lf) = m_{\p} \sum_{\p_{1}, \cdots, \p_{m}} (-1)^{m - 1} a_m \, \lf(\p_{1}, \cdots, \p_{m}), \,\,\,\,\,\,\, \lf  \in \sH(\lG, \lif{\chi}).
\end{align}
In both cases, the sum is over unordered partition of $\p$.
\end{lemma}

\begin{proof}
In view of \eqref{eq: stable decomposition}, \eqref{eq: stable expansion} follows from \eqref{eq: stable expansion 1}. We will prove \eqref{eq: stable expansion} by induction on $n$. By definition,
\[
\Sdt{\lG}{, \lp} (\lf) = \Idt{\lG}{, \lp}(\lf) - \sum_{1 \neq x \in \S{\lp}} (m_{\p} / m_{\p'_x})^2 \, \Sdt{\lG'_x}{, \lp'_x}(\lf').
\]
For $1 \neq x \in \S{\lp}$, let $G'_x = G_{I} \times G_{II}$ and $\p'_x = \p_I \times \p_{II}$. By \eqref{eq: product formula 1} and the induction assumption, we can write
\begin{align*}
\Sdt{\lG'_x}{, \lp'_x} & = \Sdt{\lG_I}{, \lp_I} \tilde{\otimes} \, \Sdt{\lG_{II}}{, \lp_{II}} \\
& = m_{\p_I} m_{\p_{II}} \Big(\sum_{\p_{I, 1}, \cdots, \p_{I, l}} (-1)^{l - 1} a_l \,  \lf_I(\lp_{I, 1}, \cdots, \lp_{I, l}) \Big) \\
& \tilde{\otimes} \Big(\sum_{\p_{II, 1}, \cdots, \p_{II, k}} (-1)^{k - 1} a_k \,  \lf_{II}(\lp_{II, 1}, \cdots, \lp_{II, k}) \Big).
\end{align*}
Then
\[
\Sdt{\lG'_x}{, \lp'_x}(\lf') = (-1)^{k+l} m_{\p_I}m_{\p_{II}} a_{k} a_{l} \lf(\lp_{I, 1}, \cdots, \lp_{I, l}, \lp_{II, 1}, \cdots, \lp_{II, k}).
\]
Note
\[
m_\p/m_{\p'_x} = \frac{1}{2}, \quad \quad  m_{\p'_x} = m_{\p_I}m_{\p_{II}}.
\]
So
\[
- (m_{\p} / m_{\p'_x})^2 \, \Sdt{\lG'_x}{, \lp'_x}(\lf') = m_{\p} \, (-1)^{k+l -1} \, (\frac{1}{2} a_{k}a_{l}) \, \lf(\lp_{I, 1}, \cdots, \lp_{I, l}, \lp_{II, 1}, \cdots, \lp_{II, k}).
\]
At last, the association of $x$ with $\p_I, \p_{II}$ gives a bijection between $\S{\lp} - \{1\}$ and unordered $2$-partitions of $\p$. So 
\[
\Sdt{\lG}{, \lp} (\lf) - \Idt{\lG}{, \lp}(\lf)
\] 
is a sum of distributions 
\[
\lf(\lp_{1}, \cdots, \lp_{m})
\]
over unordered $m$-partitions of $\p$ for $m > 1$, with coefficients
\[
(-1)^{m-1} \frac{1}{2} \sum_{S_1, S_2} a_{|S_1|} a_{|S_2|} 
\]
where the sum is over unordered $2$-partitions $(S_1, S_2)$ of $\{\p_{1}, \cdots, \p_{m}\}$. It is easy to check that
\[
\frac{1}{2} \sum_{S_1, S_2} a_{|S_1|} a_{|S_2|} = a_{m}.
\]
The rest is clear.

\end{proof}

Lemma~\ref{lemma: product formula}, ~\ref{lemma: stable expansion} complete the proof of Theorem~\ref{thm: stable multiplicity formula}, ~\ref{thm: stable multiplicity formula product} in case $n_i < N$. To complete the induction argument, we still need to prove Theorem~\ref{thm: global L-packet} for $G = SO(2N)$. By Lemma~\ref{lemma: endoscopic side GL-type}, 
\[
\Idt{\lG}{, \p}(\lf) - \Sdt{\lG}{, \p} (\lf) = \sum_{1 \neq x \in \S{\lp}} (m_{\p} / m_{\p'_x})^2 \, \Sdt{\lG'_x}{, \p'_x}(\lf').
\]
For $1 \neq x \in \S{\lp}$, let $G'_x = G_{I} \times G_{II}$ and $\p'_x = \p_I \times \p_{II}$. By \eqref{eq: product formula} and Lemma~\ref{lemma: stable expansion}, we can write
\begin{align*}
\Sdt{\lG'_x}{, \p'_x} & = \Sdt{\lG_I}{, \p_I} \tilde{\otimes} \, \Sdt{\lG_{II}}{, \lp_{II}} \\
& = m_{\p_I} m_{\p_{II}} \Big(\sum_{\p_{I, 1}, \cdots, \p_{I, l}} (-1)^{l - 1} a_l \,  \lf_I(\p_{I, 1}, \cdots, \p_{I, l}) \Big) \\
& \tilde{\otimes} \Big(\sum_{\p_{II, 1}, \cdots, \p_{II, k}} (-1)^{k - 1} a_k \,  \lf_{II}(\lp_{II, 1}, \cdots, \lp_{II, k}) \Big).
\end{align*}
Then we see $\Sdt{\lG'_x}{, \p'_x}(\lf')$ is stable. It follows that $\Idt{\lG}{, \p}(\lf)$ is stable. Then one can argue by stability to show Theorem~\ref{thm: global L-packet}. The arguement is as follows. We define
\(
\cPkt{\lp} = \otimes'_{v} \cPkt{\lp_{v}}
\) 
such that it contains a discrete automorphic representation $\lr^{0}$. Since $\Idt{\lG}{, \p}(\lf)$ is stable, it is stable at every place. So we can take $\lf = \otimes_{w} \lf_{w}$ and fix $\otimes_{w \neq v}\lf_{w}$ for any place $v$, then by \cite[Corollary 4.8]{Xu:2018} the coefficient of $\lf_{v}(\lr_{v})$ in $\Idt{\lG}{, \p}(\lf)$ must be the same for all $\lr_{v} \in \cPkt{\lp_{v}}$. By varying $\otimes_{w \neq v}\lf_{w}$ and the linear independence of characters of $\otimes_{w \neq v} \sH(\lG_{w}, \lif{\chi}_{w})$-modules, we have that
\[
[\lr^{0}] = [\lr^{0}_{v}] \otimes (\otimes_{w \neq v} [\lr^{0}_{w}])
\]
contributes to $\Idt{\lG}{, \p}(\lf)$ if and only if all elements in
\[
\cPkt{\lp_{v}} \otimes (\otimes_{w \neq v} [\lr^{0}_{w}])
\]
also contribute to $\Idt{\lG}{, \p}(\lf)$. By repeating this argument, one can show all elements in $\cPkt{\lp}$ contribute to $\Idt{\lG}{, \p}(\lf)$. Then Theorem~\ref{thm: global L-packet} will follow from \eqref{eq: discrete spectrum} immediately. This finishes the proof of Theorem~\ref{thm: global L-packet}, ~\ref{thm: stable multiplicity formula}, ~\ref{thm: stable multiplicity formula product} in the case $\phi = \phi_b \in \cPdt{G}$.


\section{Proof of Lemma~\ref{lemma: twisted endoscopic expansion}}
\label{sec: endoscopic expansion}

In this section we will prove Lemma~\ref{lemma: twisted endoscopic expansion} by assuming Theorem~\ref{thm: global L-packet}, ~\ref{thm: compatible normalization}, ~\ref{thm: stable multiplicity formula}, ~\ref{thm: stable multiplicity formula product} for the proper Levi subgroups and proper twisted endoscopic groups of $\lG$. We follow the same strategy as in \cite[Section 5.4.2]{Xu:2018}. The goal is to expand the right hand sides of
\[
\Idt{\lG}{, \p}(\lf) - \Sdt{\lG}{, \p} (\lf) = \sum_{\lG' \in \End{ell}{\lG} - \{\lG\} } \iota( \lG, \lG' ) \Sdt{\lG'}{, \p}(\lf^{\lG'}),
\]
and
\[
\tIdt{\lG^{\theta}}{, \p}(\lf) = \sum_{\lG' \in \tEnd{ell}{\lG^{\theta}}} \iota(\lG, \lG') \Sdt{\lG'}{, \p}(\lf^{\lG'})
\]
in case $\theta \neq id$ or $\omega \neq 1$. We can turn the right hand side into a sum over 
\(
\p_{G} \in \P{G, \p}
\)
and
\[
\Big\{ (\lG', \p') : \lG' \in \End{ell}{\lG^{\theta}} - \{\lG\} \text{ and } \p' \in \P{G', \p_{G}} \Big\}
\]
where $\P{G', \p_{G}}$ is the set $\D{G}'$-conjugacy classes of $L$-homomorphisms $\p' : \mathcal{L}_{\p} \rightarrow \L{G}'$ such that $\xi' \circ \p'$ is $\D{G}$-conjugate to $\p_{G}$. The contribution from each $\p_{G}$-summand is the same, so we get a constant multiple $|\P{G, \p} | = m_{\p}$. Moreover, the contribution of $(\lG', \p')$ only depends on its image under 
\[
\Big\{ (\lG', \p') : \lG' \in \End{ell}{\lG^{\theta}} - \{\lG\}  \text{ and } \p' \in \P{G', \p_{G}} \Big\} \longrightarrow \cS{\p} \backslash (\cS{\p, ss}^{\theta}(\omega) - \{1\}),
\]
which is
\(
\frac{\iota(\lG, \lG'_x) }{m_{\p'_x}} \, \Sdt{\lG'_x}{, \p'_x}(\lf').
\)
Here 
\(
\cS{\p, ss}^{\theta}(\x) = \{ s \in \cS{\p, ss}^{\theta}: \a(s) = \x\}
\)
and the fiber has size
\(
| \Out_{G}(G'_x) | | S_{\p, x} / S_{\p, x} \cap \D{G}'_x Z(\D{G})^{\Gal{}} |^{-1}.  
\)

Suppose $G$ is $Sp(2n)$ or $SO(2n, \eta)$ and $\p \neq \p_{o}, \p_{b}$. We have $m_{\p} = m_{\p_o}$ and 
\(
S^{\Sigma_0}_{\p} \cong S^{\Sigma_0}_{\p_{o}} \times S_{\p_{b}},
\)
which induces
\[
1 \rightarrow S_{\p_{b}} \rightarrow \cS{\p}^{\Sigma_0} \rightarrow \cS{\p_o}^{\Sigma_0} \rightarrow 1.
\]
Moreover, we have
\begin{align}
\label{eq: fiberation}
\cS{\p} \backslash (\cS{\p, ss}^{\theta}(\omega) - \{1\}) \longrightarrow \cS{\p_o} \backslash \cS{\p_{o, ss}}^{\theta}(\omega) \quad x \mapsto x_o.
\end{align}
Note
\[
|{\rm Out}_{G}(G')| = \begin{cases} |{\rm Out}_{G_o}(G'_o)|  & \text{ if } \p'_o \neq \p_o \\
2 |{\rm Out}_{G_o}(G'_o)| & \text{ if } \p'_o = \p_o
\end{cases}
\]
and
\(
| S_{\p, x} / S_{\p, x} \cap \D{G}'_x Z(\D{G})^{\Gal{}} | = | S_{\p_{o}, x_o} / S_{\p_{o}, x_o} \cap \D{G}'_{o, x_{o}} Z(\D{G}_{o})^{\Gal{}} |
\)
for $S_{\p_b} \subseteq S_{\p, x} \cap \D{G}' Z(\D{G})^{\Gamma}$. At last
\[
\iota(\lG, \lG') = \begin{cases} \iota(\lG_o, \lG'_o) & \text{ if } \p'_o \neq \p_o \\
\frac{1}{4} \iota(\lG_o, \lG'_o) & \text{ if } \p'_o = \p_o
\end{cases}
\]
\[
m_{\p'} = \begin{cases} m_{\p'_o} & \text{ if } \p'_o \neq \p_o \\
2 \, m_{\p'_o} & \text{ if } \p'_o = \p_o
\end{cases}
\]
and
\[
\Sdt{\lG'}{, \p'} = (\frac{1}{4})^{l} \, {\rm Tran} \, (\Sdt{\lG'_o}{, \p'_o} \, \tilde{\otimes} \, \Sdt{\lG'_b}{, \lp'_b}),
\]
where 
\[
l = \begin{cases} 2 & \text{ if } \p'_o \neq \p_o, \p'_b \neq \p_b \\
0 & \text{ if } \p'_o = \p_o, \p'_b = \p_b \\
1 & \text{ otherwise. }
\end{cases}
\]
By the compatibility of parabolic induction with endoscopic transfer, we have
\[
{\rm Tran} \, \Sdt{\lG'}{, \p'} = (\frac{1}{4})^{l} \, {\rm Tran} \, \Big({\rm Tran} \, \Sdt{\lG'_o}{, \p'_o} \, \tilde{\otimes} \, {\rm Tran} \, \Sdt{\lG'_b}{, \lp'_b} \Big).
\]
Here the distribution in the bracket may not be stable, so its transfer is in the sense of Remark~\ref{rk: compatible with parabolic induction}. We separate the coefficients for $(\lG_o, \lG'_o)$ and sum over the fibers of \eqref{eq: fiberation}. Since $S_{\p_b}$ is in the center of $\cS{\p}^{\Sigma_0}$, the fiber of
\[
\cS{\p} \backslash \cS{\p, ss}^{\theta}(\omega) \longrightarrow \cS{\p_o} \backslash \cS{\p_{o, ss}}^{\theta}(\omega)
\]
is a $S_{\p_{b}}$-torsor. Note ${\rm Tran} \, \Sdt{\lG'_o}{, \p'_o}$ remains the same over the fiber and ${\rm Tran} \, \Sdt{\lG'_b}{, \lp'_b}$ only depends on the image under $S_{\p_b} \rightarrow \S{\p_b} = \S{\lp_b}$, whose kernel is $Z(\D{G}_b) \cong \Two$. First we sum over the fiber of $x_o \neq 1$, in which case $\p'_o \neq \p_o$ and the result is
\begin{align*}
& {\rm Tran} \, \Big( {\rm Tran} \, (\Sdt{\lG'_o}{, \p'_o}) \, \tilde{\otimes} \, |Z(\D{G}_b)| \cdot \Big( \sum_{1 \neq x \in \S{\lp_b}} (\frac{1}{4})^2 \, {\rm Tran} \, (\Sdt{\lG'_{b, x}}{, \lp'_{b, x}}) + \frac{1}{4}\Sdt{\lG_b}{, \lp_b} \Big) \Big) \\
=  & \, {\rm Tran} \, \Big({\rm Tran} \, (\Sdt{\lG'_o}{, \p'_o}) \, \tilde{\otimes} \, \frac{1}{2} \Idt{\lG_b}{, \lp_b} \Big).
\end{align*}
Next we sum over the fiber of $x_o = 1$, in which case $\p'_o = \p_o$ and the result is
\begin{align*}
& {\rm Tran} \, \Big( {\rm Tran} \, (\Sdt{\lG_o}{, \p_o}) \, \tilde{\otimes} \, \frac{1}{4} \Big(|Z(\D{G}_b)| \sum_{1 \neq x \in \S{\lp_b}} \frac{1}{4} \, {\rm Tran} \, (\Sdt{\lG'_{b, x}}{, \lp'_{b, x}})  + \Sdt{\lG_b}{, \lp_b} \Big) \Big) \\
= & \, {\rm Tran} \,  \Sdt{\lG_o}{, \p_o} \, \tilde{\otimes} \, \Big(\frac{1}{2} \Idt{\lG_b}{, \lp_b} -  \frac{1}{4}\Sdt{\lG_b}{, \lp_b} \Big).
\end{align*}
In both cases we have used Lemma~\ref{lemma: endoscopic side GL-type}.

In general, recall
\[
\p_{J} = \prod_{i \in J } \p_i, \quad  \quad  \p_{K} = \prod_{i \in K} \p_i
\]
and
\[
1 \rightarrow S_{\p_{K, b}} \times \cS{\p_{J, b}} \rightarrow \cS{\p}^{\Sigma_0} \rightarrow \cS{\p_o}^{\Sigma_0} \rightarrow 1
\]
which induces
\[
\cS{\p} \backslash \cS{\p, ss}^{\theta}(\omega) \longrightarrow \cS{\p_o} \backslash \cS{\p_{o, ss}}^{\theta}(\omega) \quad x \mapsto x_o
\]
where the fiber is a $S_{\p_{K, b}} \times \cS{\p_{J, b}}$-torsor. By the decomposition
\[
S_{\p_{K, b}} \times \cS{\p_{J, b}} = \prod_{k \in K} S_{\p_{k, b}} \times \prod_{j \in J} \cS{\p_{j, b}}
\]
we could sum over $S_{\p_{k, b}}$ and $\cS{\p_{j, b}}$ respectively. Note 
\[
m_{\p} = \prod_{i} m_{\p_i} \quad \quad m_{\p'} = \prod_i m_{\p'_i}
\] 
\[
|{\rm Out}_{G}(G')| = \prod_i |{\rm Out}_{G_i}(G'_i)| \quad \quad \iota(\lG, \lG') = \prod_i \iota(\lG_i, \lG'_i)
\]
and
\[
\Sdt{\lG'}{, \p'} = (\frac{1}{4})^{l} \, {\rm Tran} \, (\Sdt{\lG'_o}{, \p'_o} \, \tilde{\otimes} \, \Sdt{\lG'_b}{, \lp'_b})
\]
where 
\[
(\frac{1}{4})^{l} \, \Sdt{\lG'_b}{, \lp'_b} = \tilde{\otimes}_{k \in K} \, (\frac{1}{4})^{l_k} \, \Sdt{\lG'_{k, b}}{, \lp'_{k, b}} \tilde{\otimes} \, \tilde{\otimes}_{j \in J} \Sdt{\lG'_{j, b}}{, \lp'_{j, b}}
\]
and 
\[
l_k = \begin{cases} 2 & \text{ if } \p'_{k, o} \neq \p_{k, o}, \p'_{k, b} \neq \p_{k, b} \\
0 & \text{ if } \p'_{k, o} = \p_{k, o}, \p'_{k, b} = \p_{k, b} \\
1 & \text{ otherwise. }
\end{cases}
\]
By the compatibility of parabolic induction with endoscopic transfer,
\begin{align*}
{\rm Tran} \, \Sdt{\lG'}{, \p'} & = (\frac{1}{4})^{l} \, {\rm Tran} \, \Big( {\rm Tran} \, \Sdt{\lG'_o}{, \p'_o} \, \tilde{\otimes} \, {\rm Tran} \, \Sdt{\lG'_b}{, \lp'_b} \Big) \\
& =  {\rm Tran} \, \Big( {\rm Tran} \, \Sdt{\lG'_o}{, \p'_o} \, \tilde{\otimes} \, \Big(\tilde{\otimes}_{k \in K}  {\rm Tran} \, (\frac{1}{4})^{l_k} \Sdt{\lG'_{k, b}}{, \lp'_{k, b}} \Big) \tilde{\otimes} \, \Big(\tilde{\otimes}_{j \in J}  {\rm Tran} \, \Sdt{\lG'_{j, b}}{, \lp'_{j, b}} \Big) \Big).
\end{align*}
We again separate the coefficients for $(\lG_{o}, \lG'_{o})$ and sum over the fibers of
\[
\cS{\p} \backslash (\cS{\p, ss}^{\theta}(\omega) - \{1\}) \longrightarrow \cS{\p_o} \backslash \cS{\p_{o, ss}}^{\theta}(\omega) \quad x \mapsto x_o.
\]
For $x_o \neq 1$, the sum over $\cS{\p_{j, b}}$ gives $\Idt{\lG_{j, b}}{, \lp_{j, b}}$ and the sum over $S_{\p_{k, b}}$ gives $\frac{1}{2}\Idt{\lG_{k, b}}{, \lp_{k, b}}$. So the result is 
\begin{align*}
& {\rm Tran} \, \Big({\rm Tran} \, (\Sdt{\lG'_o}{, \p'_o}) \, \tilde{\otimes} \, \Big( \tilde{\otimes}_{k \in K} \frac{1}{2} \, \Idt{\lG_{k, b}}{, \lp_{k, b}} \Big) \tilde{\otimes} \Big( \tilde{\otimes}_{j \in J} \Idt{\lG_{j, b}}{, \lp_{j, b}}  \Big) \Big) \\
= \, & {\rm Tran} \, \Big({\rm Tran} \, (\Sdt{\lG'_o}{, \p'_o}) \, \tilde{\otimes} \, \frac{m_{\p_{J}}}{m_{\p_{b}}} \, \Idt{\lG_b}{, \lp_b} \Big).
\end{align*}
For $x_o = 1$, we obtain the same result after adding
\(
{\rm Tran} \, \Sdt{\lG_o}{, \p_o} \tilde{\otimes} \, (\frac{1}{4})^{|K|} \, \Sdt{\lG_b}{, \lp_b}.
\)
At last, the sum over $\cS{\p_o} \backslash (\cS{\p_{o, ss}}^{\theta}(\omega) - \{1\})$ reduces to the case for $\p_o$, which has been treated in \cite[Lemma 5.25]{Xu:2018}.


\section{Compatible lifting of $L$-packets}
\label{sec: compatible lifting}

\subsection{Factorization of local and global parameters}

Let $F$ be a local (resp. global) field of characteristic zero. Let $G$ be a quasisplit symplectic or special even orthogonal group over $F$, and $[\p] \in \cuP{G}$ (resp. $\cP{G}$). We choose
\(
\xi_{G}: \L{G} \rightarrow GL(V),
\)
whose image is contained in an orthogonal group $O(V)$ for some nondegenerate symmetric bilinear form $B$. We choose a decomposition of $V$ into subrepresentations (not necessarily irreducible) of $L_{F}$ (resp. $\mathcal{L}_{\p}$) as follows
\begin{align}
\label{eq: decomposition of parameter}
\p = \oplus_{i \in I} \p_i \oplus \oplus_{j \in J} (\p_{GL, j} \oplus \p^{\vee}_{GL, j}), \quad V = \oplus_{i \in I} V_{i} \oplus \oplus_{j \in J} (W_j \oplus W_j^{\vee})
\end{align}
such that the restriction of $B$ to $V_{i}$ is nondegenerate and to $W_j$ is zero, and $W_j^{\vee}$ is dual to $W_j$ under $B$. Let $N_i = {\rm dim} \, V_i$ ($i \in I$) and $N_{j} = {\rm dim} \, W_{j}$ ($j \in J$). Let $\eta_i := \eta_{\phi_i}$ $(i \in I)$ associated with an extension $E_{i}/F$ of degree at most two.
We would like to factorize $\p$ through a sequence of proper twisted endoscopic groups and proper Levi subgroups 
\begin{align}
\label{eq: factorization of parameter}
\p: L_{F} \text{ (resp. $\mathcal{L}_{\p}$) } \xrightarrow{\p^{(n)}} \L{G}^{(n)} \xrightarrow{\xi^{(n-1)}} \L{G}^{(n-1)} \xrightarrow{\xi^{(n-2)}} \cdots \xrightarrow{\xi^{(0)}} \L{G}
\end{align}
according to the decomposition \eqref{eq: decomposition of parameter}. By this we mean
\[
\D{G}^{(k)} = \prod_l SO(V^{(k)}_l) \times \prod_m GL(W^{(k)}_m),
\]
where 
\[
V^{(k)}_{l} = \oplus_{i \in I^{(k)}_{l}} V_{i} \oplus \oplus_{j \in J_{l}^{(k)}} (W_j \oplus W_j^{\vee}), \quad \quad W^{(k)}_m = \oplus_{i \in S^{(k)}_{m} / \langle \sigma \rangle} V^{1/2}_{i, \sigma(i)} \oplus \oplus_{j \in T^{(k)}_{m}} W_j
\]
for $I_l^{(k)}, S^{(k)}_{m} \subseteq I$ and $J^{(k)}_l, T^{(k)}_m \subseteq J$, and $\sigma$ is an involution on $S^{(k)}_{m}$ without fixed points such that $\phi_{i} \cong \phi_{\sigma(i)}$, and 
\(
V^{1/2}_{i, \sigma(i)} \subseteq V_{i} \oplus V_{\sigma(i)}
\)
is an $L_{F}$ (resp. $\mathcal{L}_{\p}$)-subrepresentation $\phi^{1/2}_{i, \sigma(i)}$, on which $B$ vanishes. The composition of $\L{G}^{(k)} \rightarrow \L{G}$ with $\xi_{G}$ factors through
\[
\L{G}^{(k)} \longrightarrow \prod_l O(V^{(k)}_l) \times \prod_m GL(W^{(k)}_m).  
\] 
Let $\phi^{(k)}: L_{F} \text{ (resp. $\mathcal{L}_{\p}$) } \rightarrow \L{G}^{(k)}$ and $\phi^{(0)} = \phi$. We write
\begin{align}
\label{eq: k-th parameter}
\p^{(k)} = \prod_{l} \p^{(k)}_l \times \prod_m \phi^{(k)}_{GL, m}.
\end{align}
Then
\[
\p^{(k)}_l  = (\oplus_{i \in I^{(k)}_{l}} \p_i \oplus \oplus_{j \in J^{(k)}_{l}} (\p_{GL, j} \oplus \p^{\vee}_{GL, j})) \otimes \eta^{(k)}_{l} \in \cuP{G^{(k)}_{l}} \, \text{(resp. $\cP{G^{(k)}_{l}}$)}
\]
for some quadratic character $\eta^{(k)}_{l}$. We require that $\p^{(n)}_{l} = \p_i \otimes \eta^{(n)}_l$ for some $i \in I$ such that $N_i \neq 1$, or $\p^{(n)}_l = (1 \oplus \eta') \eta$ for some quadratic characters $\eta, \eta'$. We also require $\xi^{(k)}$ to be given by the fiber product
\[
\xymatrix{\L{G}^{(k+1)} \ar[r] \ar[d] & \L{G}^{(k)} \ar[d] \\
O(V^{(k + 1)}_{l_1}) \times O(V^{(k + 1)}_{l_2}) \ar[r] & O(V^{(k)}_{l})
}
\]
for some $l$ or  
\[
\xymatrix{\L{G}^{(k+1)} \ar[r] \ar[d] & \L{G}^{(k)} \ar[d] \\
GL(W^{(k+1)}_{m_1}) \times O(V^{(k+1)}_{l_{1}}) \ar[r] & O(V^{(k)}_{l})
}
\]
for some $l$ or
\[
\xymatrix{\L{G}^{(k+1)} \ar[r] \ar[d] & \L{G}^{(k)} \ar[d] \\
GL(W^{(k+1)}_{m_1}) \times GL(W^{(k+1)}_{m_2}) \ar[r] & GL(W^{(k)}_m)
}
\]
for some $m$.

\subsection{Lifting local and global parameters}

When $F$ is global, we will only consider the restriction of the global parameter from $\mathcal{L}_{\p}$ to $L_{F} := W_{F}$ and denote it still by $\p$. We will give two lifts of $\phi$. First of all, we would like to choose lifts of $\p_i$ for even $N_i$ and $\p_{i} \otimes \eta_i$ for odd $N_{i} \neq 1$. Then we would like to lift $\phi_i$ to $GPin(V_{i})$. When $N$ is even, we just take the chosen lift $\lp_i$. When $N$ is odd, we take
\[
\widetilde{\p_i \otimes \eta_i} \otimes \tilde{\eta}^{\square}_{i}
\]
where
\(
\tilde{\eta}_{i}^{\square}: \Gamma_{F_{i}/F} \rightarrow Pin(V_i)
\)
is a lift of 
\(
\eta_{i}: \Gamma_{F_{i}/F} \rightarrow O(V_i), \tau_i \mapsto -I.
\)
Then we get a lift $\tilde{\p}^{\sharp}_{O}$ of $\p_{O} := \prod_{i \in I} \p_i$ through the multiplication in the Clifford algebra $Cl(\oplus_{i \in I} V_{i})$:
\[
m: \prod_{i \in I} GPin(V_i) \rightarrow GPin(\oplus_{i \in I} V_{i}).
\]
Let
\(
\tilde{\p}^{\sharp} := \prod_j \phi_{GL, j} \rtimes \lp_O^{\sharp}  
\)
be a lift of $\p$ through the natural embedding  
\[
\prod_{j \in J} GL(W_j) \times GPin(\oplus_{i \in I} V_i) \hookrightarrow GPin(V).
\] 
Note $\lp^{\sharp}$ is not a homomorphism.

Next, we want to lift $\p$ through that of $\p^{(n)}$ and $\xi^{(k)}$. To lift $\p^{(n)}$, we first lift
\(
\p^{(n)}_O := \prod_{l} \p^{(n)}_l
\)
through that of $\p^{(n)}_l$ (cf. \eqref{eq: k-th parameter}). There are three cases.

\begin{enumerate}

\item $\p^{(n)}_l = \p_i \otimes \eta_i$ for odd $N_i \neq 1$. Then we can take the chosen lift $\widetilde{\phi_i \otimes \eta_i}$.

\item $\p^{(n)}_l = \p_i \otimes \eta$ for even $N_i$ and some quadratic character $\eta$. We can take the chosen lift of $\p_i$
\[
\lp_i: L_{F} \rightarrow GSpin(V_i) \rtimes \Gal{E_{i}/F}
\]
and twist it by a lift of $\eta$
\[
\tilde{\eta}^{\square}: \Gamma_{E/F} \rightarrow Spin(V_i),
\]
where $E/F$ is the quadratic extension associated with $\eta$. Note $\lp_i \otimes \tilde{\eta}^{\square}$ is not a homomorphism. The obstruction is given by a 2-cocycle 
\(
b^{(n)}_l \in Z^{2}(\Gamma_{EE_i/F}, \Two),
\) 
which splits in $Z^{2}(W_{F}, \D{D})$ (cf. \cite[Lemma 4]{Langlands:1979}). In particular, we can choose a $1$-cochain $c^{(n)}_l: W_F \rightarrow \D{D}$ such that $\lp_i \otimes \tilde{\eta}^{\square}c^{(n)}_l$ becomes a homomorphism.

\item $\p^{(n)}_l = (1 \oplus \eta')\eta$. Note $G^{(n)}_l = SO(2, \eta')$ and $1 \oplus \eta' \in \cPbd{G^{(n)}_l}$ corresponds to the trivial representation of $G^{(n)}_l(F)$ (resp. $G^{(n)}_l(\A_{F})$). We take the trivial lift $\widetilde{1 \oplus \eta'}^{{\rm triv}} \in \cPbd{\lG^{n}_l}$, i.e., it corresponds to the trivial representation of $\lG^{(n)}_l(F)$ (resp. $\lG^{(n)}_l(\A_{F})$), and twist it by $\widetilde{\eta}^{\square}$ as in the second case. To make $(\widetilde{1 \oplus \eta'}^{{\rm triv}}) \otimes \widetilde{\eta}^{\square}$ a homomorphism, we need to further choose a $1$-cochain $c^{(n)}_l: W_F \rightarrow \D{D}$. 

\end{enumerate}
Let us denote the resulting lift of $\p^{(n)}_O$ by $\lp^{(n) \square}_{O}$. Then
\(
\prod_m \phi^{(n)}_{GL, m} \rtimes \lp^{(n) \square}_O
\) 
is a lift of $\p^{(n)}$. In order to compare it with $\lp^{\sharp}$, we need to twist it further according to $\p^{(n)}_{GL, m}$. There are two cases. For $\phi_{GL, j} \otimes \eta$ in $\p^{(n)}_{GL, m}$, we need to introduce a twist by 
\begin{align}
\label{eq: Siegel Levi 1}
(\p_{GL, j} \rtimes 1) \otimes \tilde{\eta}^{\square} / (\p_{GL, j} \otimes \eta) \rtimes 1.
\end{align}
For $\p^{1/2}_{i, \sigma(i)} \otimes \eta$ in $\p^{(n)}_{GL, m}$, let us take 
\(
\p_{i, \sigma(i)} := \p_i \oplus \p_{\sigma(i)} 
\) 
as a parameter of $SO(V_i \oplus V_{\sigma(i)})$. Then we need to introduce a twist by
\begin{align}
\label{eq: Siegel Levi}
\widetilde{(\p_{i, \sigma(i)} \otimes \eta)}^{\sharp} / (\p^{1/2}_{i, \sigma(i)} \otimes \eta) \rtimes 1.
\end{align}
Denote the product of twists from $\p^{(n)}_{GL, m}$ by $\beta^{(n)}_m$. Then we define
\(
\lp^{(n) \square} :=  \prod_m \p^{(n)}_{GL, m} \rtimes (\p^{(n) \square}_O \otimes \prod_{m} \beta^{(n)}_m).
\)

We can calculate the twists \eqref{eq: Siegel Levi 1}, \eqref{eq: Siegel Levi} explicitly. First, we compute $(\p^{1/2}_{i, \sigma(i)} \otimes \eta) \rtimes 1$. Let 
\(
V_{i, \sigma(i)} = V^{1/2}_{i, \sigma(i)} \oplus (V^{1/2}_{i, \sigma(i)})^{\vee}
\)
and
\(
\iota: GL(V^{1/2}_{i, \sigma(i)}) \longrightarrow O(V_{i, \sigma(i)}), x \mapsto (x, {}^tx^{-1}),
\)
where ${}^tx^{-1} \in GL((V^{1/2}_{i, \sigma(i)})^{\vee})$ is defined by
\[
B(v, {}^tx^{-1}v^{\vee}) = B(x^{-1}v, v^{\vee}), \quad \quad v \in V^{1/2}_{i, \sigma(i)}, \quad v^{\vee} \in (V^{1/2}_{i, \sigma(i)})^{\vee}.
\]
Since $\phi^{1/2}_{i, \sigma(i)} \oplus (\phi^{1/2}_{i, \sigma(i)})^{\vee}$ preserves $V_{i}$, the image of $\phi^{1/2}_{i, \sigma(i)}$ lies in an orthogonal group $C$ such that
\[
O(V_{i}) \xleftarrow{\cong} \iota(C) = \iota(GL(V^{1/2}_{i, \sigma(i)})) \cap (O(V_i) \times O(V_{\sigma(i)})) \xrightarrow{\cong} O(V_{\sigma(i)}).
\]
The isomorphism $O(V_{i}) \cong O(V_{\sigma(i)})$ determines an isometry $V_{i} \cong V_{\sigma(i)}$ up to multiplication by $-1$. We fix an isometry and it induces an isomorphism $Pin(V_{i}) \cong Pin(V_{\sigma(i)})$. We choose sections $x \mapsto \tilde{x}$ so that the diagram commutes
\[
\xymatrix{Pin(V_{i}) \ar[r]^{\sim} & Pin(V_{\sigma(i)}) \\
O(V_i)  \ar[r]^{\sim} \ar[u]_{x \mapsto \tilde{x}} & O(V_{\sigma(i)}) \ar[u]_{x \mapsto \tilde{x}}.
}
\]
Let $C^{0}$ be the identity component of $C$. We would like to describe the image of the composition
\[
l: C \hookrightarrow GL(V^{1/2}_{i, \sigma(i)}) \hookrightarrow GL(V^{1/2}_{i, \sigma(i)}) \times \mathbb{C}^{\times} \hookrightarrow GSpin(V_{i, \sigma(i)}).
\]
Note
\[
\xymatrix{GSpin(V_{i, \sigma(i)}) \ar[r]^{\quad \quad \lambda} & \mathbb{C}^{\times} \\
GL(V^{1/2}_{i, \sigma(i)}) \times \mathbb{C}^{\times} \ar[ur]_{(g, t) \mapsto {\rm det}(g)t^2} \ar[u] & 
}
\] 
So the image of $C^{0}$ is in $Spin(V_{i, \sigma(i)})$.

\begin{lemma}
\label{lemma: Siegel Levi}
\,
\begin{enumerate}
\item
For $x \in C^{0}$,
\(
l(x) = \widetilde{\iota(x)|_{V_{i}}} \cdot \widetilde{\iota(x)|_{V_{\sigma(i)}}}.
\)

\item
For $x \in C \backslash C^{0}$,
\(
l(x) = \sqrt{-1} \, \widetilde{\iota(x)|_{V_{i}}} \cdot \widetilde{\iota(x)|_{V_{\sigma(i)}}}  \,\, {\rm mod} \,\, \Two.
\)

\end{enumerate}
\end{lemma}
 
\begin{proof}
We first show 
\(
C^{0} \rightarrow Spin(V_{i, \sigma(i)}), x \mapsto  \widetilde{\iota(x)|_{V_{i}}} \cdot \widetilde{\iota(x)|_{V_{\sigma(i)}}}
\)
defines a group homomorphism. Let $b_{i} \in Z^{2}(C^{0}, \Two)$ (resp. $b_{\sigma(i)} \in Z^2(C^{0}, \Two)$) be the obstruction of $\widetilde{\iota(x)|_{V_{i}}}$ (resp. $\widetilde{\iota(x)|_{V_{\sigma(i)}}}$) from being a group homomorphism. Since $b_{i} = b_{\sigma(i)}$ and $\widetilde{\iota(x)|_{V_{i}}}, \widetilde{\iota(x)|_{V_{\sigma(i)}}}$ commute, then it defines a group homomorphism. It differs from $l$ by a group homomorphism
\(
C^0 \rightarrow \D{D}.
\)
Since $C^{0}$ is semisimple, then this must be trivial. The second part follows from the fact that $\lambda(\iota(x)) = -1$.

\end{proof}

\begin{lemma}
\begin{enumerate}

\item The twist \eqref{eq: Siegel Levi 1} is equal to
\begin{align}
\label{eq: Siegel Levi 2}
c_{\eta}: \Gal{E/F} \rightarrow \D{D}, \quad \tau \mapsto \begin{cases} \sqrt{-1} &  N_j \text{ odd } \\
1 & N_j \text{ even } \end{cases}
\end{align}
up to ${\rm Hom}(\Gal{E/F}, \Two)$-twists.

\item
The twist \eqref{eq: Siegel Levi} is equal to
\begin{align}
\label{eq: Siegel Levi 3}
\chi_{i} \cdot \omega_{i} \cdot c_{i} \cdot c_{\eta}  \,\, {\rm mod} \,\,  {\rm Hom}(\Gal{EE_i/F}, \Two)
\end{align}
where $\chi_i = \lambda_i \circ (\widetilde{\p_i \otimes \eta_i})$ (resp. $\chi_i = \lambda_i \circ \lp_i$),
\[
(\widetilde{\p_i \otimes \eta_i}) \otimes w_{i} = (\widetilde{\p_{\sigma(i)} \otimes \eta_i}) \quad \quad \text{(resp. $\lp_i \otimes \omega_i = \lp_{\sigma(i)}$)}
\]
for $N_i$ odd (resp. $N_i$ even), and 
\(
c_{i}: \Gal{E_{i}/F} \rightarrow \D{D}, \tau_i \mapsto \sqrt{-1}.
\)

\end{enumerate}
\end{lemma}

\begin{proof}
For (1) it suffices to to compute the composition of \eqref{eq: Siegel Levi 1} with the similitude character $\lambda$. For (2) we also get $c_{\eta}$ for the same reason. So let us assume $\eta = 1$. Let $\p_{i, \sigma(i)}(w) = \iota(x)$ for some $x \in C$. Then 
\(
(\p^{1/2}_{i, \sigma(i)} \rtimes 1 ) (w) =  l(x).
\)
Let
\(
\lp^{\sharp}_{i}(w) = u(w) \, \widetilde{\iota(x)|_{V_{i}}}
\)
and
\( 
\lp^{\sharp}_{\sigma(i)} (w) = \omega_i(w) \, u(w) \, \widetilde{\iota(x)|_{V_{\sigma(i)}}}
\)
for $u(w) \in \D{D}$. Then
\[
\widetilde{(\p_{i, \sigma(i)})}^{\sharp} (w) =  u(w)^{2} \, \omega_{i}(x) \, \widetilde{\iota(x)|_{V_{i}}} \cdot \widetilde{\iota(x)|_{V_{\sigma(i)}}}.
\]
Note
\(
\chi_i(w) = \lambda_{i} \circ \lp_{i}(w) = u(w)^{2}.
\)
So it follows from Lemma~\ref{lemma: Siegel Levi} that
\[
\widetilde{(\p_{i, \sigma(i)})}^{\sharp} (w) = c_{i}(w) \chi_{i}(w) \omega_{i}(w) (\p^{1/2}_{i, \sigma(i)} \rtimes 1 )(w) \,\, {\rm mod } \,\, {\rm Hom}(\Gal{E_i/F}, \Two).
\]
\end{proof}

Now we want to lift $\xi^{(k)}$, which is given by a twisted elliptic endoscopic embedding
\begin{align}
\label{eq: embedding}
\L{H} \rightarrow SO(V^{(k)}_{l}) \rtimes W_{F}
\end{align}
for some $l$, or embedding of maximal Levi subgroup of $SO(V^{(k)}_{l})$ or $GL(W^{(k)}_m)$. In the case of Levi subgroups, there is a natural lifting. So we focus on the endoscopic case. Suppose 
\(
\D{H} = SO(V^{(k+1)}_{l_{1}}) \times SO(V^{(k+1)}_{l_{2}}).
\) 

\begin{enumerate}

\item ${\rm dim} \, V^{(k+1)}_{l_{1}}$ is odd, ${\rm dim} \, V^{(k+1)}_{l_{2}}$ is even:  \eqref{eq: embedding} factors through
\[
SO(V^{(k+1)}_{l_{1}}) \times SO(V^{(k+1)}_{l_{2}}) \rtimes \Gal{E/F} \xrightarrow{\xi_{k}} SO(V^{(k)}_l), \quad \tau \mapsto z
\]
where $\tau$ is the nontrivial element of $\Gal{E/F}$. We choose a lift by requiring
\(
\lif{\xi}_{k}^{\, \square}: \tau \mapsto \tilde{z} \in Spin(V^{(k)}_l).
\) 
The obstruction is a $2$-cocycle 
\(
b_{k} \in Z^{2}(\Gal{E/F}, \Two),
\)
which splits in $Z^{2}(\Gal{E/F}, \D{D})$.

\item ${\rm dim} \, V^{(k+1)}_{l_{1}}$ and ${\rm dim} \, V^{(k+1)}_{l_{2}}$ are both even or odd: \eqref{eq: embedding} factors through
\[
SO(V^{(k+1)}_{l_{1}}) \times SO(V^{(k+1)}_{l_{2}}) \rtimes \Gal{E_{I} E_{II}/F} \xrightarrow{\xi_{k}} SO(V^{(k)}_l) \rtimes \Gal{E/F}, \quad \tau_{i} \mapsto z_{i}
\]
where $\tau_i$ is the nontrivial element of $\Gal{E_{i}/F}$ ($i = I, II$). We choose a lift by requiring
\(
\lif{\xi}_{k}^{\, \square}: \tau_{i} \mapsto \tilde{z}_{i} \in Spin(V^{(k)}_l) \rtimes \Gamma_{E/F}.
\)
The obstruction is a $2$-cocycle 
\(
b_{k} \in Z^{2}(\Gal{E_{I}E_{II}/F}, \Two), 
\)
which splits in $Z^{2}(W_{F}, \D{D})$.

\end{enumerate}
Let us denote the resulting lift by 
\(
\tilde{\xi}^{(k)\square}: \L{\lG}^{(k+1)} \rightarrow \L{\lG}^{(k)}.
\)
Now we can define a lift of $\p$
\[
\lp^{\square} := \tilde{\xi}^{(0)\square} \circ \cdots \circ \lp^{(n)\square}.
\]
To get to a homomorphism, we still need to choose $1$-cochains 
\(
c^{(n)}, c^{(k)}: W_F \rightarrow \D{D}
\) 
for $\lp^{(n) \square}$ and $\tilde{\xi}^{(k) \square}$. Take $\lp^{(n)} := c^{(n)}\lp^{(n)\square}$ and $\tilde{\xi}^{(k)} := c^{(k)} \tilde{\xi}^{(k)\square}$. Then we define
\[
\lp = \tilde{\xi}^{(0)} \circ \cdots \circ \lp^{(n)}.
\]
Let $c = \prod_{k = 0}^{n} c^{(k)}$ and $L = \prod_{i \in I} E_i$.

\begin{lemma}
$\tilde{\p}^{\sharp} = u \cdot \tilde{\p}^{\square}$ for some $1$-cochain $u \in C^{1}(\Gal{L/F}, \Two)$.
\end{lemma}

\begin{proof}
To use induction in the proof, we define 
\[
\tilde{\p}^{(k) \sharp}_{l} := (\widetilde{\p^{(k)}_{l} \otimes \eta^{(k)}_l})^{\sharp} \otimes \tilde{\eta}_l^{(k) \square}
\] 
and define $\tilde{\p}^{(k) \square}_{l}$ as $\tilde{\p}^{\square}$. By induction, we can assume for $k > 0$
\[
\lp^{(k) \sharp}_{l} = \lp^{(k) \square}_{l} \,\, {\rm mod} \,\, C^{1}(\Gal{L/F}, \Two).
\]

Suppose $G^{(1)}$ is a twisted elliptic endoscopic group of $G$ as in case (1). We have 
\[
\xymatrix{& L_{F} \ar[dd]^{\lp^{(1) \square}_{1} \times \lp^{(1) \square}_{2}} \ar[dddl]_{(\lp^{(1) \sharp}_{1} \otimes \tilde{\eta}_1^{(1)\square} )\times \lp^{(1) \sharp}_{2}} \ar[dddr]^{\tilde{\p}^{(1) \square}} & \\
&&\\
& GSpin(V^{(1)}_{1}) \times GSpin(V^{(1)}_{2}) \rtimes \Gal{E/F} \ar[dr]^{p^{(1)}} \ar[dl] & \\
GPin(V^{(1)}_{1}) \times GPin(V^{(1)}_{2}) \ar[dr]_{m} & & G(Spin(V^{(1)}_{1}) \times Spin(V^{(1)}_{2})) \rtimes \Gal{E/F} \ar[dl]^{\tilde{\xi}^{\square}_1} \\
& GPin(V). &
}
\]
Here $\eta^{(1)}_1$ is associated with $E/F$.
Since
\[
\lp^{(1) \sharp}_{i} = \lp^{(1) \square}_{i} \,\, {\rm mod} \,\, C^{1}(\Gal{L/F}, \Two),
\] 
then the left diagram commutes up to twists by $C^{1}(\Gal{L/F}, \Two)$. Note that
\[
\lp^{\sharp} = m \circ ((\widetilde{\p^{(1)}_{1} \otimes \eta^{(1)}_1})^{\sharp} \times \lp^{(1) \sharp}_{2}) \,\, {\rm mod} \,\, C^{1}(\Gal{L/F}, \Two).
\]
Hence 
\[
\tilde{\p}^{\sharp} = \tilde{\p}^{\square} \,\, {\rm mod} \,\, C^{1}(\Gal{L/F}, \Two).
\] 
The argument for case (2) is similar. 

Now let us suppose $G^{(1)}$ is a maximal Levi subgroup of $G$, then by definition
\[
\lp^{(1) \square} = \p^{(1)}_{GL, 1} \times (\lp^{(1) \square}_{1} \otimes \beta^{(1)}_{1}).
\]
We can let $\lp^{(1)\square}$ factor through the Levi subgroup
\[
\prod_{j \in T^{(1)}_{1}} GL(W_{j}) \times \prod_{i \in S^{(1)}_{1} / \langle \sigma \rangle} GL(V^{\frac{1}{2}}_{i, \sigma(i)}) \times GSpin(V^{(1)}_{1}) \rtimes \Gal{E/F}.
\]
Note this is also the Levi subgroup of
\[
^{L}\widetilde{H} := G(\prod_{j \in T^{(1)}_{1}} Spin(W_{j} \oplus W^{\vee}_{j}) \times \prod_{i \in S^{(1)}_{1} / \langle \sigma \rangle} Spin(V_{i} \oplus V_{\sigma(i)}) \times Spin(V^{(1)}_{1})) \rtimes \Gal{E/F}.
\]
onto which we have a projection from 
\[
^{L}\lif{\widetilde{H}} := \prod_{j \in T^{(1)}_{1}} GSpin(W_{j} \oplus W^{\vee}_{j}) \times \prod_{i \in S^{(1)}_{1} / \langle \sigma \rangle} GSpin(V_{i} \oplus V_{\sigma(i)}) \times GSpin(V^{(1)}_{1}) \rtimes \Gal{E/F}.
\]
Note
\[
\lp^{\sharp} =  m \circ (\prod_{j \in T^{(1)}_{1}} (\p_{GL, j} \rtimes 1) \times \prod_{i \in S^{(1)}_{1} / \langle \sigma \rangle} \tilde{\p}^{\sharp}_{i, \sigma(i)} \times \lp^{(1) \sharp}_{1}) \,\, {\rm mod} \,\, C^{1}(\Gal{L/F}, \Two).
\]
By induction, we may assume 
\[
\lp^{(1)\square}_{1} = \lp^{(1) \sharp}_{1} \,\, {\rm mod} \,\, C^{1}(\Gal{L/F}, \Two)
\]
Since 
\[
\beta^{(1)}_{1} = \prod_{i \in S^{(1)}_{1} / \langle \sigma \rangle} \tilde{\p}_{i, \sigma(i)}^{\sharp} / \p^{1/2}_{i, \sigma(i)} \rtimes 1,
\]
then the result follows from the commutative diagram
\[
\xymatrix{
L_{F} \ar[r] \ar[dr] & ^{L}\lif{H} \ar[r] & GPin(V) \\
& ^{L}\lif{\lif{H}}. \ar[u] \ar[ur]_{m}
}
\]

\end{proof}

As a direct consequence, we have

\begin{proposition}
\label{prop: compatible lifting parameter}
Suppose we get two lifts $\lp^{\square}$ and ${}'\lp^{\square}$ from different factorizations of $\p$ as \eqref{eq: factorization of parameter}. Then there exists a $1$-cochain $u \in C^{1}(\Gal{L/F}, \Two)$ such that
\(
\lp^{\, \square} = u \cdot {}'\lp^{\, \square}.
\)
\end{proposition}

Suppose we construct two lifts $\lp, {}'\lp$ by choosing $1$-cochains $\{c^{(k)}\}$ and $\{{}'c^{(k)}\}$ respectively. We would like to define a homomorphism
\(
\chi^{c, {}'c}: W_{F} \rightarrow \D{D}
\)
such that
\(
\chi^{c, {}'c} = u \cdot {}'c/c
\)
for some 1-cochain $u \in C^1(\Gal{L/F}, \Two)$. The existence of $\chi^{c, {}'c}$ follows from Proposition~\ref{prop: compatible lifting parameter}. From this definition, we also see that $\chi^{c, {}'c}$ is uniquely determined up to twists by ${\rm Hom}(\Gal{L/F}, \Two)$. Note $c, {}'c$ split two cocycles $b, {}'b$ of $\Gal{L/F}$ in $\Two$ respectively by our construction. So the existence of $\chi^{c, {}'c}$ is also equivalent to $b = {}'b$ in $H^{2}({\Gal{L/F}, \Two})$. Since ${\rm Hom}(\Gal{L/F}, \Two) \subseteq \a(\S{\p}^{\Sigma_0})$ (cf. \cite[Lemma 6.9]{Xu:2018}), we can draw the following consequence.

\begin{corollary}
\label{cor: compatible lifting parameter}
Suppose $F$ is local and let $\lp, {}'\lp$ be two lifts of $\p \in \cuP{G}$ obtained by choosing $1$-cochains $\{c^{(k)}\}$ and $\{{}'c^{(k)}\}$ respectively, then $\lp \otimes \chi^{c, {}'c} = {}'\lp$ in $\cuP{\lG}$.
\end{corollary}

\subsection{Lifting local and global L-packets}

Let $F$ be local (resp. global), we would like to state the representation theoretic analogue of Corollary~\ref{cor: compatible lifting parameter}. The idea is to lift $\cPkt{\p^{(n)}}$ and $\xi^{(k)}$. The strategy is parallel with the case of parameters. First of all, we would like to choose lifts of $\cPkt{\p_i}$ for even $N_i$ and $\cPkt{\p_{i} \otimes \eta_i}$ for odd $N_{i} \neq 1$. We will fix them no matter of the factorization of $\p$. To lift $\cPkt{\p^{(n)}}$, let
\(
\p^{(n)}_O = \prod_{l} \p^{(n)}_l.
\)
We first lift $\cPkt{\p^{(n)}_O}$ through that of $\cPkt{\p^{(n)}_l}$. There are three cases.

\begin{enumerate}

\item $\p^{(n)}_l = \p_i \otimes \eta_i$ for odd $N_i \neq 1$. Then we can take the chosen lift $\cPkt{\widetilde{\phi_i \otimes \eta_i}}$.

\item $\p^{(n)}_l = \p_i \otimes \eta$ for even $N_i$. We can take the chosen lift of $\cPkt{\lp_i}$ and twist it by $\tilde{\eta} := \tilde{\eta}^{\square}c^{
(n)}_{l}$ for some $1$-cochain $c^{(n)}_{l}$ of $W_{F}$ in $\D{D}$.

\item $\p^{(n)}_l = (1 \oplus \eta') \otimes \eta$. We take the trivial lift $\cPkt{\widetilde{1 \oplus \eta'}^{{\rm triv}}}$ and twist it by $\widetilde{\eta} := \widetilde{\eta}^{\square}c^{(n)}_{l}$ for some $1$-cochain $c^{(n)}_{l}$ of $W_{F}$ in $\D{D}$.

\end{enumerate}
Let us denote the resulting lift by $\cPkt{\lp^{(n)}_O}$ with $1$-cochain $c^{(n)}_O := \prod_{l} c^{(n)}_l$. Then
\(
\boxtimes_m \pi_{\phi^{(n)}_{GL, m}} \boxtimes \cPkt{\lp^{(n)}_O} 
\) 
is a lift of $\cPkt{\p^{(n)}}$. For the purpose of comparison, we also need to introduce some twists and $1$-cochains according to $\p^{(n)}_{GL, m}$. There are two cases.\begin{itemize}

\item For $\phi_{GL, j} \otimes \eta$ in $\p^{(n)}_{GL, m}$, we introduce a $1$-cochain $c_{\eta}$ of $\Gal{E/F}$ in $\D{D}$ (cf. \eqref{eq: Siegel Levi 2}). 

\item For $\p^{1/2}_{i, \sigma(i)} \otimes \eta$ in $\p^{(n)}_{GL, m}$, we introduce a twist by $\chi_{i} \, \omega_{i}$, where $\chi_i$ is the central character of $\cPkt{\widetilde{\p_{i} \otimes \eta_i}}$ (resp. $\cPkt{\lp_i}$), and 
\[
\cPkt{\widetilde{\p_{i} \otimes \eta_i}} \otimes \omega_i = \cPkt{\widetilde{\p_{\sigma(i)} \otimes \eta_i}} \quad \quad ({\rm resp.} \, \cPkt{\lp_i} \otimes \omega_i = \cPkt{\lp_{\sigma(i)}})
\] 
for $N_{i}$ odd (resp. even). We also introduce a $1$-cochain $c_{\eta}c_{i}$ of $\Gal{EE_i/F}$ in $\D{D}$ (cf. \eqref{eq: Siegel Levi 3}). 
\end{itemize}
We define $c^{(n)}$ to be the product of $c^{(n)}_{O}$ with these $1$-cochains. Denote the product of twists from $\p^{(n)}_{GL, m}$ by $\chi^{(n)}_m$. Then we define
\[
\cPkt{\lp^{(n)}} := \boxtimes_m \pi_{\phi^{(n)}_{GL, m}} \boxtimes (\cPkt{\lp^{(n)}_O} \otimes \prod_m \chi^{(n)}_m).
\]
We lift $\xi^{(k)}$ as before. Then we can define
\[
\cPkt{\lp} := {\rm Tran}_{\tilde{\xi}^{(0)}} \circ \cdots \circ {\rm Tran}_{\tilde{\xi}^{(n-1)}} \cPkt{\lp^{(n)}}
\]
with respect to some $1$-cochains $\{c^{(k)}\}$. Let $\cPkt{\lp}$ and $\cPkt{{}'\lp}$ be two lifts obtained by choosing $1$-cochains $\{c^{(k)}\}$ and $\{{}'c^{(k)}\}$ respectively. Then $c, {}'c$ split two cocycles $b, {}'b$ of $\Gal{L/F}$ in $\Two$ respectively by our construction. 

\begin{lemma} 
$b = {}'b$ in $H^{2}({\Gal{L/F}, \Two})$. 
\end{lemma}

\begin{proof}
The $1$-cochains $\{c^{(k)}\}$ and $\{{}'c^{(k)}\}$ can be taken to be the same as for the lifting of $\p$. Then the result follows from Proposition~\ref{prop: compatible lifting parameter} and the discussion after that. 
\end{proof}

As a consequence, we can define a homomorphism
\(
\chi^{c, {}'c}: W_{F} \rightarrow \D{D}
\)
such that
\(
\chi^{c, {}'c} = u \cdot {}'c/c
\)
for some 1-cochain $u \in C^{1}(\Gal{L/F}, \Two)$. It is uniquely determined up to twists by ${\rm Hom}(\Gal{L/F}, \Two)$.

\begin{proposition}
\label{prop: compatible lifting packet}
Let $\cPkt{\lp}$ and $\cPkt{{}'\lp}$ be two lifts obtained by $1$-cochains $\{c^{(k)}\}$ and $\{{}'c^{(k)}\}$ respectively, then 
\begin{align}
\label{eq: compatible lifting}
\cPkt{\lp} \otimes \chi^{c, {}'c} = \cPkt{{}'\lp}.
\end{align}
\end{proposition}

We say $\cPkt{\p}$ has {\bf compatible lifting} with respect to the decomposition \eqref{eq: decomposition of parameter} if \eqref{eq: compatible lifting} holds. We can generalize this notion to $G = G(n_{1}) \times G(n_{2}) \times \cdots \times G(n_{q})$, which is an immediate consequence of Proposition~\ref{prop: compatible lifting packet}. 

\begin{corollary}
\label{cor: compatible lifting packet for product}
Compatible lifting of $L$-packets holds for $\lG$.
\end{corollary}

\begin{proof}
For $[\p] \in \cuP{G}$ (resp. $\cP{G}$), we write $\p = \p_{1} \times \p_{2} \times \cdots \times \p_{q}$ such that $[\p_{i}] \in \cuP{G(n_{i})}$ (resp. $\cP{G(n_i)}$) for $1 \leqslant i \leqslant q$. Note 
\[
\lG \subseteq \lG(n_{1}) \times \lG(n_{2}) \times \cdots \times \lG(n_{q}) 
\]
Any factorization of $\p$ gives rise to a factorization of $\p_{i}$ for $1 \leqslant i \leqslant q$, which determines the packets $\cPkt{\lp_{i}}$, and the corresponding packet $\cPkt{\lp}$ will be the restriction of $\bigotimes_{i = 1}^{q} \cPkt{\lp_{i}}$ to $\lG$. By Proposition~\ref{prop: compatible lifting packet}, 
\(
\cPkt{\lp_i} \otimes \chi^{c_i, {}'c_{i}} = \cPkt{{}'\lp_{i}}.
\)
Since $\chi^{c, {}'c} = \prod_{i} \chi^{c_i, {}'c_{i}}$, then
\(
\cPkt{\lp} \otimes \chi^{c, {}'c} = \cPkt{{}'\lp}.
\)
\end{proof}

\subsection{Proof of Proposition~\ref{prop: compatible lifting packet}}

In this section we shall give the proof of Proposition~\ref{prop: compatible lifting packet}. First we observe that the local case implies the global case. In the archimedean case, the Local Langlands correspondence is known and is compatible with twisted endoscopic transfer and parabolic induction \cite{Langlands:1989} \cite{Shelstad1:2008} \cite{Mezo:2013}. So the result follows from that for lifting parameters (cf. Corollary~\ref{cor: compatible lifting parameter}). 

Now let us assume $F$ is nonarchimedean. Note if any $\p_i$ or $\p_{GL, j}$ in \eqref{eq: factorization of parameter} is not irreducible, we can always factorize them further. This means that we can reduce to the case that they are all irreducible and $\p_{GL, j}$ are not of orthogonal types. Moreover, when $\p$ is unramified, the fundamental lemma for spherical Hecke algebras \cite{Hales:1995} \cite{LMW:2018} implies the compatibility of the local Langlands correspondence for unramified representations with the twisted endoscopic transfer. We should point out that the twisted case has been shown only for sufficiently large residue characteristic. The compatibility with parabolic induction is clear. So the result would follow from the same argument as in the archimedean case. 

After these reductions, the proof of compatible lifting of $L$-packets proceeds in the same way as our proof of the main local theorem in \cite{Xu:2018}. We assume that it holds for $\lG = GSp(2n)$ and $GSO(2n + 2, \eta)$, when $n < N$. When $n = N$, we first show that the case of non-elliptic parameters follows immediately from this assumption.

\begin{lemma}
\label{lemma: compatible lifting of L-packets for non-elliptic parameters}
Suppose $[\p] \in \cuP{G} - \cPel{G}$, then the packet $\cPkt{\p}$ has compatible lifting.
\end{lemma}

\begin{proof}
For the decomposition of $\p$ in \eqref{eq: decomposition of parameter}, we can assume $J \neq \emptyset$ or $\p_1 \cong \p_2 \cong \p_3$. The idea is to compare any factorization of $\p$ to certain fixed ones. For a given factorization, we will assume that $G^{(1)}$ is a twisted elliptic endoscopic group of $G$ such that ${\rm dim} \, V^{(1)}_{1}$ is odd and ${\rm dim} \, V^{(1)}_{2}$ is even, and $\p_{GL, 1}$ (resp. $\p_{i} \oplus \p_{j}$ for $\{i, j\} \subseteq \{1, 2, 3\}$) is in $\p^{(1)}_{1}$. We will add superscripts to the left of $\p$ to indicate different factorizations. The arguments for the other cases are similar, so we do not include them here.

If $J \neq \emptyset$, we factorize $\p$ in the following four ways. From top to bottom, the first is the given factorization and the last is the fixed one.
\[
\begin{tikzcd}
& & \arrow{d}{^{0}\xi^{(1)}} & \\
\arrow{r} & {}^1\p^{(2)}_2 \times \Big((\p_{GL, 1} \otimes \eta) \rtimes {}^1\p^{(2)}_1 \Big) \arrow{r}{{}^1\xi^{(1)}} \arrow[dashrightarrow]{dd} \arrow[dashrightarrow]{rdd}{'\xi^{(1)}} & {}^1\p^{(1)}_2 \times {}^1\p^{(1)}_1 = {}^1\p^{(1)} = {}^0\p^{(1)} \arrow{rd}{{}^1\xi^{(0)} = {}^0\xi^{(0)}} & \\
& && \p \\
\arrow{r} & ^{2}\p^{(2)}_2 \times (\p_{GL, 1} \rtimes  {}^2\p^{(2)}_1) \arrow{r}{^{2}\xi^{(1)}} & \p_{GL, 1} \rtimes {}^2\p^{(1)}_1 = {}^2\p^{(1)} = {}^3\p^{(1)} \arrow{ur}[swap]{^{2}\xi^{(0)} = {}^3\xi^{(0)}} & \\
& & \arrow{u}[swap]{{}^3\xi^{(1)}} & 
\end{tikzcd}
\]
By our assumption, we have ${}^1\p^{(1)}_2 = {}^1\p^{(2)}_2 = {}^2\p^{(2)}_2$, ${}^1\p^{(2)}_1 = {}^2\p^{(2)}_1$ and $\eta = {}^1\eta^{(1)}_1 = {}^1\eta^{(2)}_1$. By induction and Corollary~\ref{cor: compatible lifting packet for product}, we can assume $\cPkt{{}^1\p^{(1)}}$ has compatible lifting, then 
\begin{align}
\label{eq: descent}
\cPkt{{}^0\lp} \otimes \chi^{{}^0c, {}^1c} = \cPkt{{}^1\lp} \, .
\end{align}
Next we would like to show
\begin{align}
\label{eq: parallel factorization}
\cPkt{{}^1\lp} \otimes \chi^{{}^1c, {}^2c} = \cPkt{{}^2\lp}
\end{align}
where ${}^2c$ is chosen according to ${}^1c$ as follows. The factorization of ${}^1\p^{(2)}$ gives rise to a factorization of ${}^2\p^{(2)}$. We will choose ${}^1c^{(k)} = {}^2c^{(k)} \, \text{ if } 2 \leqslant k < n$, and ${}^1c^{(n)} = {}^2c^{(n)} c_{\eta}$. It follows
\(
\cPkt{{}^1\lp^{(2)}} = \cPkt{{}^2\lp^{(2)}} \otimes (\eta \circ {\rm det}_{GL(W_1^{(2)})}).
\)
By the compatibility of twisted endoscopic transfer with parabolic induction, 
\[
\cPkt{{}^1\lp} = {\rm Tran}_{^{2}\lif{\xi}^{(0)}} \, {\rm Tran}_{'\lif{\xi}^{(1)}} \, \cPkt{{}^1\lp^{(2)}}.
\]
Moreover,
\[
{\rm Tran}_{{}^2\lif{\xi}^{(1)}} \, \cPkt{^{2}\lp^{(2)}} = {\rm Tran}_{'\lif{\xi}^{(1)}} \, \cPkt{{}^1\lp^{(2)}} \otimes \chi^{{}^1c^{(0)}c_{\eta}, {}^2c^{(1)}}.
\]
Since 
\[
{}^2c^{(1)}/{}^1c^{(0)}c_{\eta} = {}^2c^{(1)} \, {}^2c^{(n)}/{}^1c^{(0)}{}^1c^{(n)} = {}^2c/{}^1c,
\] 
then 
\[
\chi^{{}^1c^{(0)}c_{\eta}, {}^2c^{(1)}} = \chi^{{}^1c, {}^2c} \, \text{ mod ${\rm Hom}(\Gal{L/F}, \Two)$ }
\]
This proves \eqref{eq: parallel factorization}. By induction, we can also assume $\cPkt{{}^2\p^{(1)}}$ has compatible lifting, then
\begin{align}
\label{eq: descent 1}
\cPkt{{}^2\lp} \otimes \chi^{{}^2c, {}^3c} = \cPkt{{}^3\lp} \, .
\end{align}
So the result follows from \eqref{eq: descent}, \eqref{eq: parallel factorization} and \eqref{eq: descent 1}.

If $\p_1 \cong \p_2 \cong \p_3$, we factorize $\p$ in the following four ways. From top to bottom, the first is the given factorization and the last is one of the three fixed ones indexed by $\{i, j\} \subseteq \{1, 2, 3\}$.
\[
\begin{tikzcd}
& & \arrow{d}{^{0}\xi^{(1)}} & \\
\arrow{r} & {}^1\p^{(2)}_2 \times \Big((\p^{\frac{1}{2}}_{i,j} \otimes \eta) \rtimes {}^1\p^{(2)}_1 \Big) \arrow{r}{{}^1\xi^{(1)}} \arrow[dashrightarrow]{dd} \arrow[dashrightarrow]{rdd}{'\xi^{(1)}} & {}^1\p^{(1)}_2 \times {}^1\p^{(1)}_1 = {}^1\p^{(1)} = {}^0\p^{(1)} \arrow{rd}{{}^1\xi^{(0)} = {}^0\xi^{(0)}} & \\
& && \p \\
\arrow{r} & ^{2}\p^{(2)}_2 \times (\p^{\frac{1}{2}}_{i,j} \rtimes  {}^2\p^{(2)}_1) \arrow{r}{^{2}\xi^{(1)}} & \p^{\frac{1}{2}}_{i,j} \rtimes {}^2\p^{(1)}_1 \arrow{ur}[swap]{^{2}\xi^{(0)} = ^{3}\xi^{(0)}} = {}^2\p^{(1)} = {}^3\p^{(1)} & \\
& & \arrow{u}[swap]{{}^3\xi^{(1)}} & 
\end{tikzcd}
\]
By our assumption, ${}^1\p^{(1)}_2 = {}^1\p^{(2)}_2 = {}^2\p^{(2)}_2$, ${}^1\p^{(2)}_1 = {}^2\p^{(2)}_1$ and $\eta = {}^1\eta^{(1)}_1 = {}^1\eta^{(2)}_1$. Here $\{i, j\}$ depends on ${}^0\p^{(1)}$. As in the previous case, we can show
\[
\cPkt{{}^0\lp} \otimes \chi^{{}^0c, {}^3c} = \cPkt{^{3}\lp} \, .
\]
It remains to compare the three fixed factorizations.
\[
\begin{tikzcd}
\arrow{r} & \p^{\frac{1}{2}}_{i,j} \rtimes \p^{(1)}_1 \arrow{rd}{\xi^{(0)}} & \\
&& \p \\
\arrow{r} & \p^{\frac{1}{2}}_{i,k} \rtimes {}'\p^{(1)}_1 \arrow{ur}[swap]{{}'\xi^{(0)}} & 
\end{tikzcd}
\]
Note $\D{G}^{(1)}, {}'\D{G}^{(1)}$ are Levi subgroups of $\D{G}$ conjugate under $S_{\p}$, and we have $\p^{\frac{1}{2}}_{i, j} \cong \p^{\frac{1}{2}}_{i, k}$ and $\p^{(1)}_1 \cong {}'\p^{(1)}_1$ under the conjugation. It follows that there exists $g \in G(F)$ such that $g(G^{(1)})g^{-1} = {}'G^{(1)}$ and $\cPkt{\p^{(1)}}^{g} = \cPkt{{}'\p^{(1)}}$. We can identify $\p^{(1)}_1$ and ${}'\p^{(1)}_1$, and take the same factorization and $1$-cochains for them. Recall that for $\p^{\frac{1}{2}}_{i, j}$ (resp. $\p^{\frac{1}{2}}_{i, k}$), we need to introduce a twist by $\omega_{i, j} \chi_i$ (resp. $\omega_{i, k} \chi_i$), where $\chi_i$ is the central character of $\cPkt{\lp_i}$ and 
\(
\cPkt{\lp_i} \otimes \omega_{i, j} = \cPkt{\lp_j}, \cPkt{\lp_i} \otimes \omega_{i, k} = \cPkt{\lp_k} \text{(resp. $\cPkt{\lif{\p_i \otimes \eta_i}} \otimes \omega_{i, j} = \cPkt{\lif{\p_j \otimes \eta_j}}, \cPkt{\lif{\p_i \otimes \eta_i}} \otimes \omega_{i, k} = \cPkt{\lif{\p_k \otimes \eta_k}}$)}
\)
for $N_i$ even (resp. odd). Then
\[
\cPkt{\lp^{(1)}_1}^{g} = \cPkt{{}'\lp^{(1)}_1} \otimes \omega_{k, j}
\]
and it follows that 
\[
\cPkt{\lp} = \cPkt{{}'\lp} \otimes \omega_{j, k} \otimes (\omega_{i, j} \chi_i / \omega_{i, k} \chi_i) = \cPkt{{}'\lp}.
\]
So we have \eqref{eq: compatible lifting} for the three fixed factorizations. This finishes the proof.
\end{proof}

At last, we can treat the elliptic parameters. Suppose 
\[
\p = \p_{1} \+ \cdots \+ \p_{q} \+ 2\p_{q+1} \+ \cdots \+ 2\p_{r} \in \cPel{G}.
\]

\begin{lemma}
\label{lemma: exceptional case}
If $\p_{i} = \eta_{i}$, and $r \leqslant 3$ (or $r \leqslant 4$ when $G$ is symplectic), then $\cPkt{\p}$ has compatible lifting.
\end{lemma}

The proof is similar to the general case with one difference. So we will consider the general case first.

\begin{lemma}
\label{lemma: compatible lifting of L-packets for elliptic parameters}
Suppose $[\p] \in \cPel{G}$, then the packet $\cPkt{\p}$ has compatible lifting.
\end{lemma}

\begin{proof}
By applying \cite[Lemma 6.18]{Xu:2018} to $\p$, we get a global lift 
\[
\dot{\p} = \dot{\p}_{1} \+ \cdots \+ \dot{\p}_{q} \+ 2\dot{\p}_{q+1} \+ \cdots \+ 2\dot{\p}_{r} \in \cPel{\dot{G}}
\]
satisfying the following properties.
\begin{enumerate}

\item $\dot{\p}_{u} = \p$, $\S{\dot{\p}} = \S{\p}$, $\S{\lif{\dot{\p}}} = 1$.

\item At nonarchimedean place $v \neq u$, $\dot{\p}_{v}$ is a direct sum of quasicharacters of $\dot{F}_{v}^{\times}$ with at most one ramified quasicharacter counted without multiplicities modulo the unramified quasicharacters.

\item $\Sigma_0$-strong multiplicity one holds for $\lif{\dot{\p}}$ at the place $u$, i.e., for any automorphic representation $\tilde{\dot{\r}}$ of $\lif{\dot{G}}$ such that $[\tilde{\dot{\r}}_{v}] \in \cPkt{\lif{\dot{\p}}_{v}}$ for all $v \neq u$, $[\tilde{\dot{\r}}]$ must be also in $\cPkt{\lif{\dot{\p}}}$.

\end{enumerate}
Suppose $\cPkt{\lif{\dot{\p}}}$ and $\cPkt{{}'\lif{\dot{\p}}}$ are obtained from two different factorizations of $\dot{\p}$ as in \eqref{eq: factorization of parameter} with $1$-cochains $\dot{c}$ and ${}'\dot{c}$ respectively, then they are global $L$-packets of $\lif{\dot{G}}$ by the functoriality of twisted endoscopic transfer (cf. \cite[Lemma 6.28]{Xu:2018}). By Lemma~\ref{lemma: compatible lifting of L-packets for non-elliptic parameters} and Lemma~\ref{lemma: exceptional case},
\(
\cPkt{\lif{\dot{\p}}_v} \otimes \chi^{\dot{c}, {}'\dot{c}}_{v} = \cPkt{\lif{{}'\dot{\p}}_{v}}
\)
for all $v \neq u$. By $\Sigma_0$-strong multiplicity one at the place $u$, we have
\(
\cPkt{\lif{\dot{\p}}} \otimes \chi^{\dot{c}, {}'\dot{c}} = \cPkt{{}'\lif{\dot{\p}}}
\)
and hence
\(
\cPkt{\lif{\dot{\p}}_u} \otimes \chi^{\dot{c}, {}'\dot{c}}_{u} = \cPkt{{}'\lif{\dot{\p}}_{u}}.
\)
At last, one just needs to notice that the factorizations of $\dot{\p}$ one-to-one correspond to the factorizations of $\p$.
\end{proof}

To apply this proof to Lemma~\ref{lemma: exceptional case}, one notes $\lif{\dot{\p}}_v$ may be elliptic at nonarchimedean place $v \neq u$. Nevertheless, we have
\(
\cPkt{\lif{\dot{\p}}_v} \otimes \chi^{\dot{c}, {}'\dot{c}}_{v} = \cPkt{{}'\lif{\dot{\p}}_{v}}
\)
for almost all places by the fundamental lemma for spherical Hecke algebras. By \cite[Lemma 6.13]{Xu:2018}), we know that $\Sigma_0$-strong multiplicity one holds for $\lif{\dot{\p}}$, i.e., for any automorphic representation $\tilde{\dot{\r}}$ of $\lif{\dot{G}}$ such that $[\tilde{\dot{\r}}_{v}] \in \cPkt{\lif{\dot{\p}}_{v}}$ for almost all places $v$, $[\lif{\dot{\r}}]$ must be also in $\cPkt{\lif{\dot{\p}}}$. Hence
\(
\cPkt{\lif{\dot{\p}}} \otimes \chi^{\dot{c}, {}'\dot{c}} = \cPkt{{}'\lif{\dot{\p}}}.
\)
This completes the proof of Lemma~\ref{lemma: exceptional case}.

\subsection{Towards local Langlands correspondence for similitude groups}

Suppose $F$ is nonarchimedean, and $G$ is a quasisplit symplectic or special even orthogonal group over $F$. We would like to show

\begin{theorem}
\label{thm: LLC}
One can associate any $\lp \in \cPbd{\lG}$ with an $L$-packet $\cPkt{\lp}$ such that 
\begin{enumerate}

\item $\cPkt{\lp \otimes \omega} = \cPkt{\lp} \otimes \omega$ for any $\omega \in H^{1}(W_{F}, Z(\D{\lG})) \cong {\rm Hom}(\lG(F), \mathbb{C}^{\times})$;

\item the central character of $\cPkt{\lp}$ is $\chi_{\lp}$, which is associated with the composition of $\lp$ and $\D{\lG} \rightarrow \D{Z}_{\lG}$;

\item it is compatible with twisted endoscopic transfer and parabolic induction.

\end{enumerate}
\end{theorem}

For the compatibility with twisted endoscopic transfer, we first need to construct the correspondence for the twisted endoscopic groups. It suffices to consider $G = G_1 \times \cdots \times G_q$ \eqref{eq: product} and $\p = \p_1 \times \cdots \times \p_q$ for $\p_{i} \in \cPbd{G_i}$. We define
\[
\lp = {\bf p} \circ (\lp_1 \times \cdots \times \lp_q) \mapsto \cPkt{\lp} = \cPkt{\lp_1} \tilde{\otimes} \cdots \tilde{\otimes} \, \cPkt{\lp_q}.
\]
where ${\bf p}: \L{(\prod_{i = 1}^{q} \lG_{i})} \rightarrow \L{\lG}$ is dual to the inclusion $\lG \hookrightarrow \prod_{i = 1}^{q} \lG_{i}$.

\begin{proof}

The proof is by construction, which depends on our choices for simple parameters.
\begin{itemize}

\item For $N \neq 1$ and $\p \in \cPsm{G}$, we fix a correspondence 
\(
\lp \mapsto \cPkt{\lp}
\) 
such that (1) and (2) hold. Here (1) can be achieved by \cite[Corollary 4.2]{Xu:2018} and Theorem~\ref{thm: twist}, and (2) is due to \cite[Proposition 6.27]{Xu:2016}.

\item Associate
\(
(\widetilde{1 \oplus \eta})^{\rm triv} \otimes \omega
\)
with $\cPkt{\widetilde{1 \oplus \eta}^{\rm triv}} \otimes \omega$.
\end{itemize}
For $\p \in \cPbd{G}$, we take a decomposition of $\p$ as \eqref{eq: decomposition of parameter} such that all $\p_i, \p_{GL, j}$ are irreducible and $\p_{GL, j}$ are not of orthogonal type. Then we also take a factorization of $\p$ as \eqref{eq: factorization of parameter}. We can construct a lift $\lp$ with respect to any choice of lifts $\lp_i$ for even $N_i$ (resp. $\widetilde{\p_i \otimes \eta_i}$ for odd $N_i \neq 1$) and $1$-cochains $\{c^{(k)}\}$. Correspondingly, we will construct $\cPkt{\lp}$ with respect to the same factorization of $\p$, $\cPkt{\lp_i}$ (resp. $\cPkt{\widetilde{\p_i \otimes \eta_i}}$) associated with $\lp_i$ (resp. $\widetilde{\p_i \otimes \eta_i}$), and the same $1$-cochains. Then we define
\[
\lp \otimes \omega \mapsto \cPkt{\lp} \otimes \omega
\]
for $\omega \in \bar{H}^{1}(W_{F}, \D{D}) \cong \Hom(\lG(F)/G(F), \C^{\times})$. It is not hard to see that this is independent of choices of lifts $\lp_i$ (resp. $\widetilde{\p_i \otimes \eta_i}$ if $N_i$ is odd) and $\{c^{(k)}\}$. It is also independent of the factorization by Corollary~\ref{cor: compatible lifting parameter} and Proposition~\ref{prop: compatible lifting packet}. At last, it is independent of the decomposition of $\p$, since any two decompositions are conjugate under $S_{\p}$. 

To show (3), note
\(
\lp = \lif{\xi}^{(0)} \circ \lp^{(1)}
\)
and 
\(
\cPkt{\lp} = {\rm Tran}_{\lif{\xi}^{(0)}} \, \cPkt{\lp^{(1)}}.
\)
So it remains to show that $\lp^{(1)}$ is associated with $\cPkt{\lp^{(1)}}$, which is clear by our construction and choices for simple parameters. Finally, (1) and (2) follows from (3) by extending the arguments of \cite[Lemma 5.2 and Lemma 4.1]{Xu:2016} to the twisted case, which is straightforward.
\end{proof}

As a consequence of Theorem~\ref{thm: LLC} (1), we have

\begin{corollary}
For $\p \in \cPbd{G}$ and $\omega \in H^{1}(W_{F}, Z(\D{\lG})) \cong {\rm Hom}(\lG(F), \mathbb{C}^{\times})$, 
\begin{align*}
[\lp \otimes \omega] = [\lp] \text{ if and only if } \cPkt{\lp} \otimes \omega = \cPkt{\lp}.
\end{align*}
\end{corollary}

This corollary is an enhancement of \cite[Corollary 4.2]{Xu:2018} when $G$ is orthogonal. It generalizes Theorem~\ref{thm: twist}.


\section{Proof of main theorems}
\label{sec: proof of main theorems}

Let $F$ be global. We denote $G(n) := Sp(2n), SO(2n+2, \eta)$ and $\lG(n) := GSp(2n), GSO(2n+2, \eta)$ over $F$. Let 
\(
G = G(n_{1}) \times G(n_{2}) \times \cdots \times G(n_{q}).
\)
We assume Theorem~\ref{thm: global L-packet}, \ref{thm: compatible normalization}, ~\ref{thm: functoriality}, \ref{thm: stable multiplicity formula}, \ref{thm: stable multiplicity formula product} hold when $n_{i} < N$ for all $1 \leqslant i \leqslant q$. In the proof, we will always treat the symplectic case first, and then take the results to enhance the induction assumptions.

\subsection{Proof of theorem~\ref{thm: global L-packet}}

\subsubsection{Discrete parameters}
\label{subsubsec: discrete parameter}

Let $G = G(N)$,
\(
\p = \p_1 \boxplus \cdots \boxplus \p_r \in \cPdt{G}
\)
and $\p \neq \p_b$. Suppose $\p \neq \p_o$, then by \eqref{eq: endoscopic expansion 1} 
\begin{align*}
\Idt{\lG}{, \p}(\lf) - \Sdt{\lG}{, \p} (\lf) & = C_{\lp_o} \sum_{\x' \in Y / \a(\S{\p_o})} \sum_{x_o \in \S{\lp_o} - \{1\}} \lf_{\lG}' (\lp_o \otimes \x', x_o; \lp_b) \\
& + {\rm Tran} \, \Sdt{\lG_o}{, \p_o} \tilde{\otimes} \, (\frac{1}{2} \Idt{\lG_b}{, \lp_b} - \frac{1}{4} \Sdt{\lG_b}{, \lp_b} ), \quad \lf \in \sH(\lG, \lif{\chi}).
\end{align*}
We add
\begin{align*}
2 \cdot C_{\lp_o} \sum_{\x \in Y / \a(\S{\p_o})} \sum_{x_o \in \S{\lp_o} - \{1\}}  \sum_{ \substack{ [\lr] \in \cPkt{}(\lp_o; \lp_b) \\ \langle x_o, \lr \rangle = -1}} \lf_{\lG}(\lr \otimes \x)
\end{align*}
to both sides, then the right hand side becomes stable. It follows that
\begin{align}
\label{eq: global L-packet for discrete parameter}
\Idt{\lG}{, \p}(\lf) +  2 \cdot C_{\lp_o} \sum_{\x \in Y / \a(\S{\p_o})} \sum_{x_o \in \S{\lp_o} - \{1\}}  \sum_{ \substack{ [\lr] \in \cPkt{}(\lp_o; \lp_b) \\ \langle x_o, \lr \rangle = -1}} \lf_{\lG}(\lr \otimes \x)  
\end{align}
is also stable. Then we can argue by stability as follows. We define
\(
\cPkt{\lp} = \otimes'_{v} \cPkt{\lp_{v}}
\) 
such that it contains a discrete automorphic representation $\lr^{0}$. Since \eqref{eq: global L-packet for discrete parameter} is stable, it is stable at every place. So we can take $\lf = \otimes_{w} \lf_{w}$ and fix $\otimes_{w \neq v}\lf_{w}$ for any place $v$, then by \cite[Corollary 4.8]{Xu:2018} the coefficient of $\lf_{v}(\lr_{v})$ in \eqref{eq: global L-packet for discrete parameter} must be the same for all $\lr_{v} \in \cPkt{\lp_{v}}$. By varying $\otimes_{w \neq v}\lf_{w}$ and the linear independence of characters of $\otimes_{w \neq v} \sH(\lG_{w}, \lif{\chi}_{w})$-modules, we have that
\[
[\lr^{0}] = [\lr^{0}_{v}] \otimes (\otimes_{w \neq v} [\lr^{0}_{w}])
\]
contributes to \eqref{eq: global L-packet for discrete parameter} if and only if all elements in
\[
\cPkt{\lp_{v}} \otimes (\otimes_{w \neq v} [\lr^{0}_{w}])
\]
also contribute to \eqref{eq: global L-packet for discrete parameter}. By repeating this kind of argument, one can show all elements in $\cPkt{\lp}$ contribute to \eqref{eq: global L-packet for discrete parameter}. Note for any $[\lr] \in \cPkt{\lp}$ such that $\langle \cdot, \lr \rangle = 1$, it can only contribute to $\Idt{\lG}{, \p}(\lf)$, which means it belongs to the discrete automorphic spectrum. Then \eqref{formula: discrete spectrum} follows from this observation and \eqref{eq: discrete spectrum}. Moreover, we get
\(
\cPkt{\lp} = \cPkt{}(\lp_o; \lp_b),
\) 
if there exists $x_o \in \S{\lp_o} - \{1\}$ and $[\lr] \in \cPkt{}(\lp_o; \lp_b)$ such that $\langle x_o, \lr \rangle = -1$, i.e., 
\begin{align}
\label{eq: local-global component group}
\S{\lp_o} \longrightarrow \prod_{v}\S{\lp_{o, v}}
\end{align}
is not trivial. 
Suppose $\p = \p_o$, then by \eqref{eq: endoscopic expansion}
\begin{align*}
\Idt{\lG}{, \p}(\lf) - \Sdt{\lG}{, \p} (\lf) = C_{\lp} \sum_{\x' \in Y / \a(\S{\p})} \sum_{x \in \S{\lp} - \{1\}} \lf_{\lG}' (\lp \otimes \x', x), \quad \lf \in \sH(\lG, \lif{\chi}).
\end{align*}
We add
\begin{align*}
2 \cdot C_{\lp} \sum_{\x \in Y / \a(\S{\p})} \sum_{x \in \S{\lp} - \{1\}}  \sum_{ \substack{ [\lr] \in \cPkt{\lp_{x}} \\ \langle x, \lr \rangle = -1}} \lf_{\lG}(\lr \otimes \x)
\end{align*}
to both sides, where $\cPkt{\lp_{x}}$ is a global packet transferred from $\cPkt{\lp'}$ for $x \in \S{\lp} - \{1\}$. By the argument on stability again, we get the global $L$-packet $\cPkt{\lp}$, with which \eqref{formula: discrete spectrum} holds. Moreover, 
\begin{align}
\label{eq: functoriality discrete}
\cPkt{\lp} = \cPkt{\lp_x}
\end{align}
if there exists $[\lr] \in \cPkt{\lp_x}$ such that $\langle x, \lr \rangle = -1$, i.e., $x$ has nontrivial image in $\prod_{v}\S{\lp_{v}}$. By Proposition~\ref{prop: compatible lifting packet}, \eqref{eq: functoriality discrete} holds for all $x \neq 1$ if it holds for one. So it suffices to know that \eqref{eq: local-global component group} is not trivial. 

\subsubsection{General case}

Suppose $\p \in \cP{G} - \cPdt{G}$, then $\p$ factors through $\p_{M} \in \cPdt{M}$ for some proper Levi subgroup $M$ of $G$. Then by our induction assumption, we have a global $L$-packet $\cPkt{\lp_{M}}$ for $\lM$, and we can define $\cPkt{\lp}$ to be the set of irreducible constituents induced from $\cPkt{\lp_{M}}$. At last, let
\[
G = G(n_{1}) \times G(n_{2}) \times \cdots \times G(n_{q}),
\]  
with $n_{i} \leqslant N$ for $1 \leqslant i \leqslant q$ and $\p \in \cP{G}$. The global $L$-packet is the restriction of the global $L$-packet $\bigotimes_{i = 1}^{q} \cPkt{\lp_{i}}$ of $\lG(n_{1}) \times \lG(n_{2}) \times \cdots \times \lG(n_{q})$, with which \eqref{formula: discrete spectrum} holds.

\subsection{Proof of theorem~\ref{thm: compatible normalization}}
\label{subsec: proof of compatible normalization}

Next, we consider Theorem~\ref{thm: compatible normalization} for $G$. For $\p \in \cPdt{G}$,
\[
\tIdt{\lG^{\theta}}{, \p}(\lf) = tr R^{(\lG^{\theta}. \x)}_{disc, \p}(\lf) = \sum_{\x'} \sum_{\substack{ [\lr] \in \cPkt{\lp} \otimes \x' \\ \langle \cdot, \lr \rangle = 1}} m(\lr, \theta, \x) \lf_{\lG^{\theta}}(\lr, \x), \quad \lf \in \sH(\lG, \lif{\chi}).
\]
Here the sum of $\x'$ is taken over
\(
Y / \prod^{aut}_{v} \a(\S{\p_{v}}^{\Sigma_{0}})
\)
with
\(
\prod^{aut}_{v} \a(\S{\p_{v}}^{\Sigma_{0}}) := \{ \x \in Y :  \x_{v} \in \a(\S{\p_{v}}^{\Sigma_0}) \text{ for all } \, v \},
\)
and $m(\lr, \theta, \x)$ is some integer, whose absolute value is less than or equal to
\(
m(\lp) :=  m_{\p} \, |\prod^{aut}_{v} \a(\S{\p_{v}}^{\Sigma_{0}}) | \, | \a(\S{\p})|^{-1}.
\)
By \eqref{eq: twisted endoscopic expansion}, we have
\[
\tIdt{\lG^{\theta}}{, \p}(\lf) = |\S{\lp_o}|^{-1} \sum_{\x'} \sum_{x'_o \in \S{\p_o}^{\theta}(\x)} m(\lp) \lf'_{\lG^{\theta}}(\lp_o \otimes \x', x'_o; \lp_b),
\]
where the sum of $\x'$ is again over
\(
Y / \prod^{aut}_{v} \a(\S{\p_{v}}^{\Sigma_{0}}).
\)
Therefore
\begin{align*}
\sum_{\x'} \sum_{\substack{ [\lr] \in \cPkt{\lp} \otimes \x' \\ \langle \cdot, \lr \rangle = 1}} m(\lr, \theta, \x) \lf_{\lG^{\theta}}(\lr, \x) =   |\S{\lp_o}|^{-1} \sum_{\x'} \sum_{x'_o \in \S{\p_o}^{\theta}(\x)} m(\lp) \lf'_{\lG^{\theta}}(\lp_o \otimes \x', x'_o; \lp_b). 
\end{align*}
There are two cases. If $\p \neq \p_o$, then 
\begin{align*}
\sum_{\x'} \sum_{\substack{ [\lr] \in \cPkt{\lp} \otimes \x' \\ \langle \cdot, \lr \rangle = 1}} m(\lr, \theta, \x) \lf_{\lG^{\theta}}(\lr, \x) & =  |\S{\lp_o}|^{-1} \sum_{\x'} \sum_{y_o \in \S{\lp_o}} m(\lp)  \sum_{[\lr] \in \cPkt{}(\lp_o; \lp_{b}) \otimes \x'} \langle y, \lr \rangle \lf_{\lG^{\theta}}(\lr, \x) \\
& = m(\lp) \sum_{\x'} \sum_{\substack{[\lr] \in \cPkt{}(\lp_o; \lp_{b}) \otimes \x' \\ \langle \cdot, \lr \rangle = 1}} \lf_{\lG^{\theta}}(\lr, \x),
\end{align*}
where $\lf_{\lG^{\theta}}(\lr, \x)$ is normalized by $x$. By the linear independence of twisted characters of $\bar{\mathcal{H}}(\lG, \lif{\chi})$-modules, we get 
\(
\cPkt{\lp} = \cPkt{}(\lp_o; \lp_b).
\)
If $\p = \p_o$, we define a global packet $\cPkt{\lp_{x'}}$ transferred from $\cPkt{\lp'}$ for any $x' \in \S{\p}^{\theta}(\x)$. Note $\S{\p}^{\theta}(\x) = x \cdot \S{\lp}$, then
\begin{align}
\label{eq: compatible normalization 2}
\sum_{\x'} \sum_{\substack{ [\lr] \in \cPkt{\lp} \otimes \x' \\ \langle \cdot, \lr \rangle = 1}} m(\lr, \theta, \x) \lf_{\lG^{\theta}}(\lr, \x) =   |\S{\lp}|^{-1} \sum_{\x'} \sum_{y \in \S{\lp}} m(\lp)  \sum_{[\lr] \in \cPkt{\lp_{xy}} \otimes \x'} \langle y, \lr \rangle \lf_{\lG^{\theta}}(\lr, \x),
\end{align}
where $\lf_{\lG^{\theta}}(\lr, \x)$ is normalized by $x$. This implies 
\begin{align*}
\sum_{\x'} \sum_{\substack{ [\lr] \in \cPkt{\lp} \otimes \x' \\ \langle \cdot, \lr \rangle = 1}} m(\lr, \theta, \x) \lf_{\lG^{\theta}}(\lr, \x) =   |\S{\lp}|^{-1}  \sum_{\x'} \sum_{y \in \S{\lp}} m(\lp)  \sum_{\substack{ [\lr] \in \cPkt{\lp_{xy}} \otimes \x' \\ \langle \cdot, \lr \rangle = 1}} \lf_{\lG^{\theta}}(\lr, \x). 
\end{align*}
It follows from the linear independence of twisted characters of $\bar{\mathcal{H}}(\lG, \lif{\chi})$-modules that we can choose $\cPkt{\lp_{x'}} = \cPkt{\lp}$ for all $x' \in \S{\p}^{\theta}(\x)$. In both cases, we have
\begin{align*}
\sum_{\substack{ [\lr] \in \cPkt{\lp} \\ \langle \cdot, \lr \rangle = 1}} m(\lr, \theta, \x) \lf_{\lG^{\theta}}(\lr, \x) =   m(\lp) \sum_{\substack{ [\lr] \in \cPkt{\lp} \\ \langle \cdot, \lr \rangle = 1}} \lf_{\lG^{\theta}}(\lr, \x). 
\end{align*}
So $m(\lr, \theta, \x) = m(\lp)$. Hence
\begin{align*}
\tIdt{\lG^{\theta}}{, \p}(\lf) = m_{\p} \sum_{\x' \in Y / \a(\S{\p})} \sum_{\substack{ [\lr] \in \cPkt{\lp} \otimes \x' \\ \langle \cdot, \lr \rangle = 1}} \lf_{\lG^{\theta}}(\lr, \x).
\end{align*}

As a consequence, we have shown the functoriality for 
\[
s \in (S_{\p_o}^{\Sigma_0} \times Z(\D{G}_b))/Z(G)^{\Gal{}} - (S_{\lp_o} \times Z(\D{G}_b)/Z(G)^{\Gal{}}.
\]
By the compatible lifting of global $L$-packets again (cf. Corollary~\ref{cor: compatible lifting packet for product}), this implies the functoriality for all 
\(
s \in (S_{\p_o}^{\Sigma_0} \times Z(\D{G}_b))/Z(G)^{\Gal{}},
\) 
whenever $\cS{\p_o}^{\Sigma_0} \neq \cS{\lp_o}$, i.e. $\S{\p_o}^{\Sigma_0} \neq \S{\lp_o}$.

\subsection{Proof of Theorem~\ref{thm: functoriality}, ~\ref{thm: stable multiplicity formula}, ~\ref{thm: stable multiplicity formula product}}

For Theorem~\ref{thm: stable multiplicity formula}, ~\ref{thm: stable multiplicity formula product}, it remains to show part (2) and (3) (cf. Section~\ref{sec: GL-type}), which are direct consequences of Theorem~\ref{thm: functoriality} and the comparison formulas (cf. Section~\ref{sec: comparison formula}). So we will mainly focus on Theorem~\ref{thm: functoriality}.

\subsubsection{Non-discrete parameters}

\begin{lemma}
\label{lemma: functoriality of twisted endoscopic transfer for non-discrete parameters}
Suppose $G = G(N)$ and $\p \in \cP{G} - \cPdt{G}$, then Theorem~\ref{thm: functoriality} holds for $\lp$.
\end{lemma}

\begin{proof}
Suppose $\p_o \in \cP{G_o} - \cPdt{G_o}$, then $\p$ factors through $\p_M$ for a Levi subgroup $M$ of $G$ such that $\p_{M, b} = \p_b$. Note $\cPkt{\lp}$ is obtained by parabolic induction of $\cPkt{\lp_M}$. So the theorem follows from compatible lifting of $L$-packets (cf. Proposition~\ref{prop: compatible lifting packet}). Suppose $\p_b \in \cP{G_b} - \cPdt{G_b}$, then for any $s \in (S^{\theta}_{\p_{o}} \times Z(\D{G}_b))/Z(\D{G})^{\Gal{}}$ we have $\p'_{s, b} \in \cP{G'_{s, b}} - \cPdt{G'_{s, b}}$. So $\p'_s$ factors through a Levi subgroup of $G'_s$. Then the theorem follows from the compatibility of parabolic induction with twisted endoscopic transfer.
\end{proof}

To treat the discrete case, we need to strengthen our induction assumptions as follows.

\begin{lemma}
\label{lemma: functoriality of twisted endoscopic transfer for non-discrete parameters 1}
Suppose $G = G(n)$ and 
\[
\p = \p_{1} \boxplus \cdots \boxplus \p_{q} \boxplus (2 \p_{q+1} \boxplus \cdots \boxplus 2 \p_{r}) \in \cPel{G^{\theta}}
\]
for some $\theta \in \Sigma_0$ such that $\sum_{i = 1}^{r} N_{i} < N$, then Theorem~\ref{thm: global L-packet}, \ref{thm: functoriality} hold for $\lp$.
\end{lemma}

\begin{proof}
The case of discrete parameters follows from our induction assumption. For non-discrete parameters, we can construct the global $L$-packets by parabolic induction, which shows Theorem~\ref{thm: global L-packet}. By repeating the argument in Lemma~\ref{lemma: functoriality of twisted endoscopic transfer for non-discrete parameters}, we can reduce Theorem~\ref{thm: functoriality} to the discrete case.
\end{proof}

\begin{corollary}
\label{cor: functoriality of twisted endoscopic transfer for general group 1}
Suppose 
\(
G = G(n_{1}) \times \cdots \times G(n_{q})
\)  
and $\p = \p_{1} \times \cdots \times \p_{q} \in \cP{G} - \cPdt{G}$ such that $\p_{i} \in \cP{G(n_i)}$ satisfies the condition in Lemma~\ref{lemma: functoriality of twisted endoscopic transfer for non-discrete parameters 1}, then Theorem~\ref{thm: global L-packet}, ~\ref{thm: functoriality}, ~\ref{thm: stable multiplicity formula}, ~\ref{thm: stable multiplicity formula product} hold for $\lp$.
\end{corollary}

\begin{proof}
For Theorem~\ref{thm: global L-packet}, we can construct the $L$-packets by parabolic induction. Next we show Theorem~\ref{thm: functoriality}. Since the restriction of global $L$-packet $\bigotimes_{i = 1}^{q} \cPkt{\lp_{i}}$ of $\lG(n_{1}) \times \lG(n_{2}) \times \cdots \times \lG(n_{q}) $ to $\lG$ is also a global $L$-packet, then the functoriality of twisted endoscopic transfer for $\lp$ follows from that for $\lp_{i}$. At last we prove Theorem~\ref{thm: stable multiplicity formula}, ~\ref{thm: stable multiplicity formula product} by induction on $\sum_{i = 1}^{q} n_i$. It remains to show part (2) and (3), which follows from the comparison of \eqref{eq: endoscopic expansion}, \eqref{eq: endoscopic expansion 1} with \eqref{eq: twisted spectral expansion 1} using \eqref{eq: coefficient} and \eqref{eq: global intertwining relation}.
\end{proof}

\subsubsection{Discrete parameters}

\begin{lemma}
\label{lemma: functoriality of twisted endoscopic transfer for discrete parameters}
Suppose $G = G(N)$ and $\p \in \cPdt{G}$, then Theorem~\ref{thm: functoriality} holds for $\lp$.
\end{lemma}

Suppose 
\(
\p = \p_1 \boxplus \cdots \boxplus \p_r \in \cPdt{G}
\)
and $\p \neq \p_b$. From the discussions in the end of Section~\ref{subsubsec: discrete parameter} and Section~\ref{subsec: proof of compatible normalization}, it remains to consider the case that \eqref{eq: local-global component group} is trivial and $\S{\p_o}^{\Sigma_0} = \S{\lp_o}$. Without loss of generality, we can make the following assumptions.
\begin{enumerate}

\item If $G$ is symplectic, then $N_1$ is odd and $\eta_i = 1$ for all $i$.

\item If $G$ is special even orthogonal and $\eta_G \neq 1$, then $N_i$ is even for all $i$ and $\eta_1 = \eta_G$.

\item If $G$ is special even orthogonal and $\eta_G = 1$, then $N_i$ is even and $\eta_i = 1$ for all $i$, and $\p_1$ is not of $GL$-type (cf. Definition~\ref{def: GL-type}).

\end{enumerate}
In all cases we take an auxiliary parameter as follows
\[
\p^{+} := 2\p_{1} \# \p_{2} \# \cdots \# \p_{r} \in \cP{G^{+}}.
\]
Let $\varepsilon \in \S{\p^{+}}^{\Sigma_0}$ such that $\p^{+}$ factors through 
\[
\p^{+}_{\varepsilon} := (\p_{1} \# \p_{2} \# \cdots \# \p_{r}) \times \p_{1} \in \cPdt{G^{+}_{\varepsilon}}.
\]
Let $\theta = \theta_0$ in case (1) and $\theta = 1$ in case (2) and (3). By \eqref{eq: twisted spectral expansion 1},
\begin{align}
\label{eq: twisted spectral expansion 2}
I^{(\lG^{+, \theta}, \eta_1)}_{disc, \p^{+}} (\lf) = C_{\lp^{+}_o} \sum_{\x \in Y / \langle \eta_1 \rangle} \sum_{x_o \in \S{\p^{+}_o}^{\theta}(\eta_1)} i^{\theta}_{\p^{+}_o}(x_o) \lf_{\lG^{+, \theta}}(\lp^{+}_o \otimes \x, x_o; \lp_b) ,     \,\,\,\,\, \lf \in \sH(\lG^{+}, \lif{\chi}^{+}).
\end{align}
The idea is to compare this with \eqref{eq: twisted endoscopic expansion} in case (1), (2) and with \eqref{eq: endoscopic expansion}, \eqref{eq: endoscopic expansion 1} in case (3). The main problem is that we can not apply \eqref{eq: twisted endoscopic expansion}, \eqref{eq: endoscopic expansion}, \eqref{eq: endoscopic expansion 1} to $\p^{+}$ directly, since Theorem~\ref{thm: stable multiplicity formula product} is only applicable for $\p^{+}_{x}$ with $x \in \S{\p^{+}}^{\Sigma_0} - \{ \varepsilon, 1 \}$ by our induction assumption and Corollary~\ref{cor: functoriality of twisted endoscopic transfer for general group 1}. So we need to compute them separately here. In view of the proof of Lemma~\ref{lemma: twisted endoscopic expansion}, it suffices to determine the contribution of ${\rm Tran} \, \Sdt{\lG^{+}_{\epsilon}}{, \p^{+}_{\epsilon}}$. First, the coefficient of this term is
\[
d := m_{\p^{+}} \frac{\iota(\lG^{+}, \lG^{+}_{\epsilon}) }{m_{\p^{+}_{\epsilon}}} | \Out_{G^{+}}(G^{+}_{\epsilon}) | | S_{\p^{+}, \epsilon} / S_{\p^{+}, \epsilon} \cap \D{G}^{+}_{\epsilon} Z(\D{G}^{+})^{\Gal{}} |^{-1}.
\]
Next we will give a formula for $\Sdt{\lG^{+}_{\epsilon}}{, \p^{+}_{\epsilon}}$. If $\p \neq \p_o$, we can apply \eqref{eq: endoscopic expansion 1} to $\p^{+}_{\epsilon}$ and \eqref{eq: endoscopic expansion} to $\p^{+}_{\epsilon, o}$, and get
\begin{align*}
\Sdt{\lG^{+}_{\epsilon}}{, \p^{+}_{\epsilon}} & = \frac{1}{4} \, \Sdt{\lG^{+}_{\epsilon, o}}{, \p^{+}_{\epsilon, o}} \tilde{\otimes} \, \Sdt{\lG_b}{, \lp_b} \\ & + \Big( \Idt{\lG^{+}_{\epsilon}}{, \p^{+}_{\epsilon}} - \frac{1}{2} {\rm Tran} (\Idt{\lG^{+}_{\epsilon, o}}{, \p^{+}_{\epsilon, o}} \tilde{\otimes} \, \Idt{\lG_b}{, \lp_b})\Big).
\end{align*}
If $\p = \p_o$, let us define $\cPkt{\lp^{+}_{\varepsilon, -}}$ to be the packet $\cPkt{\lp^{+}_{\varepsilon, x}}$ of $\lG^{+}_{\varepsilon}$ transferred from the $L$-packet  of endoscopic group $\lG^{+}_{\varepsilon, x}$ for any $x \in \S{\lp^{+}_{\varepsilon}} - \{1\}$, and we have 
\(
\lf_{\lG^{+}_{\varepsilon}}(\lp^{+}_{\varepsilon, -} \otimes \x') = \lf'_{\lG^{+}_{\varepsilon}}(\lp^{+}_{\varepsilon} \otimes \x', x).
\) 
Then we can apply \eqref{eq: endoscopic expansion} to $\p^{+}_{\epsilon}$, and get
\begin{align*}
\Sdt{\lG^{+}_{\epsilon}}{, \p^{+}_{\epsilon}}(\lf) & =  C_{\lp^{+}_{\epsilon}} \sum_{\x' \in Y'} \lf_{\lG^{+}_{\varepsilon}}(\lp^{+}_{\varepsilon, -} \otimes \x') 
\\ & + m_{\p^{+}_{\epsilon}} \Big( \sum_{\x' \in Y'} \lf_{\lG^{+}_{\varepsilon}}(\lp^{+}_{\varepsilon} \otimes \x') - \sum_{\x' \in Y'} \lf_{\lG^{+}_{\varepsilon}}(\lp^{+}_{\varepsilon, -} \otimes \x') \Big), \quad \lf \in \sH(\lG^{+}_{\epsilon}, \lif{\chi}^{+}_{\epsilon}).
\end{align*}
In case (1) and (2), the substitutes for \eqref{eq: twisted endoscopic expansion} can be described as follows. If $\p \neq \p_o$, then
\begin{align*}
I^{(\lG^{+, \theta}, \eta_1)}_{disc, \p^{+}} (\lf) & = C_{\lp^{+}_o} \sum_{\x \in Y / \langle \eta_1 \rangle} \sum_{x \in \S{\p^{+}_o}^{\theta}(\eta_1)} e'^{\theta}_{\p^{+}_o}(x_o) \lf_{\lG^{+, \theta}}' (\lp^{+}_o \otimes \x, x_o; \lp_b) \\ & + d \cdot {\rm Tran}\Big( \Idt{\lG^{+}_{\epsilon}}{, \p^{+}_{\epsilon}} - \frac{1}{2} {\rm Tran} (\Idt{\lG^{+}_{\epsilon, o}}{, \p^{+}_{\epsilon, o}} \tilde{\otimes} \, \Idt{\lG_b}{, \lp_b})\Big), \quad \lf \in \sH(\lG^{+}, \lif{\chi}^{+}).
\end{align*}
If $\p = \p_o$, then 
\begin{align*}
I^{(\lG^{+, \theta}, \eta_1)}_{disc, \p^{+}} (\lf) & = C_{\lp^{+}} \sum_{\x \in Y / \langle \eta_1 \rangle } \sum_{x \in \S{\p^{+}}^{\theta}(\eta_1)} e'^{\theta}_{\p^{+}}(x) \lf_{\lG^{+, \theta}}' (\lp^{+} \otimes \x, x) \\ & + d \cdot m_{\p^{+}_{\epsilon}} \Big( \sum_{\x' \in Y'} \lf^{\lG^{+}_{\varepsilon}}(\lp^{+}_{\varepsilon} \otimes \x') - \sum_{\x \in Y'}\lf^{\lG^{+}_{\varepsilon}}(\lp^{+}_{\varepsilon, -} \otimes \x') \Big), \quad \lf \in \sH(\lG^{+}, \lif{\chi}^{+}),
\end{align*}
where $\lf_{\lG^{+, \theta}}' (\lp^{+} \otimes \x, \epsilon)$ is redefined to be the transfer of $\cPkt{\lp^{+}_{\epsilon, -}}$ rather than $\cPkt{\lp^{+}_{\epsilon}}$. In case (3), the substitutes for \eqref{eq: endoscopic expansion}, \eqref{eq: endoscopic expansion 1} can be described as follows. If $\p \neq \p_o$, then
\begin{align*}
I^{\lG^{+}}_{disc, \p^{+}} (\lf) - S^{\lG^{+}}_{disc, \p^{+}}(\lf) & = C_{\lp^{+}_o} \sum_{\x \in Y} \sum_{x \in \S{\p^{+}_o}} e'_{\p^{+}_o}(x_o) \lf_{\lG^{+}}' (\lp^{+}_o \otimes \x, x_o; \lp_b) \\
& + d \cdot {\rm Tran}\Big( \Idt{\lG^{+}_{\epsilon}}{, \p^{+}_{\epsilon}} - \frac{1}{2} {\rm Tran} (\Idt{\lG^{+}_{\epsilon, o}}{, \p^{+}_{\epsilon, o}} \tilde{\otimes} \, \Idt{\lG_b}{, \lp_b})\Big), \quad \lf \in \sH(\lG^{+}, \lif{\chi}^{+}).
\end{align*}
If $\p = \p_o$, then
\begin{align*}
I^{\lG^{+}}_{disc, \p^{+}} (\lf) - S^{\lG^{+}}_{disc, \p^{+}}(\lf) & = C_{\lp^{+}} \sum_{\x \in Y} \sum_{x \in \S{\p^{+}}} e'_{\p^{+}}(x) \lf_{\lG^{+}}' (\lp^{+} \otimes \x, x) \\
& + d \cdot m_{\p^{+}_{\epsilon}} \Big( \sum_{\x' \in Y'} \lf^{\lG^{+}_{\varepsilon}}(\lp^{+}_{\varepsilon} \otimes \x') - \sum_{\x' \in Y'} \lf^{\lG^{+}_{\varepsilon}}(\lp^{+}_{\varepsilon, -} \otimes \x') \Big), \quad \lf \in \sH(\lG^{+}, \lif{\chi}^{+}).
\end{align*}

To make the comparison with \eqref{eq: twisted spectral expansion 2}, we note in all cases
\[
i^{\theta}_{\p^{+}_o}(x_o) \lf_{\lG^{+, \theta}}(\lp^{+}_o \otimes \x, x_o; \lp_b) = e'^{\theta}_{\p^{+}_o}(x_o) \lf_{\lG^{+, \theta}}' (\lp^{+}_o \otimes \x, x_o; \lp_b), \quad \lf \in \sH(\lG^{+}, \lif{\chi}^{+})
\]
for all $x_o \in \S{\p^{+}_o}^{\theta}(\eta_1)$. In case (1) and (2), we get
\begin{align}
\label{eq: vanishing 1}
{\rm Tran}\Big( \Idt{\lG^{+}_{\epsilon}}{, \p^{+}_{\epsilon}} - \frac{1}{2} {\rm Tran} (\Idt{\lG^{+}_{\epsilon, o}}{, \p^{+}_{\epsilon, o}} \tilde{\otimes} \, \Idt{\lG_b}{, \lp_b})\Big) = 0
\end{align}
when $\p \neq \p_o$, and
\begin{align}
\label{eq: vanishing 2}
\sum_{\x' \in Y'} \lf^{\lG^{+}_{\varepsilon}}(\lp^{+}_{\varepsilon} \otimes \x') - \sum_{\x' \in Y’} \lf^{\lG^{+}_{\varepsilon}}(\lp^{+}_{\varepsilon, -} \otimes \x') = 0, \quad \lf \in \sH(\lG^{+}, \lif{\chi}^{+})
\end{align}
when $\p = \p_o$. In case (3), we get
\begin{align}
\label{eq: vanishing 3}
{\rm Tran}\Big( \Idt{\lG^{+}_{\epsilon}}{, \p^{+}_{\epsilon}} - \frac{1}{2} {\rm Tran} (\Idt{\lG^{+}_{\epsilon, o}}{, \p^{+}_{\epsilon, o}} \tilde{\otimes} \, \Idt{\lG_b}{, \lp_b})\Big) = \frac{-1}{d} S^{\lG^{+}}_{disc, \p^{+}}
\end{align}
when $\p \neq \p_o$, and
\begin{align}
\label{eq: vanishing 4}
\sum_{\x' \in Y'} \lf^{\lG^{+}_{\varepsilon}}(\lp^{+}_{\varepsilon} \otimes \x') - \sum_{\x' \in Y'}\lf^{\lG^{+}_{\varepsilon}}(\lp^{+}_{\varepsilon, -} \otimes \x') =  \frac{-1}{d \, m_{\p^{+}_{\epsilon}}} S^{\lG}_{disc, \p^{+}}(\lf), \quad \lf \in \sH(\lG^{+}, \lif{\chi}^{+})
\end{align}
when $\p = \p_o$. 

We will further show that \eqref{eq: vanishing 3} and \eqref{eq: vanishing 4} both vanish, hence reduce to \eqref{eq: vanishing 1} and \eqref{eq: vanishing 2} respectively. By compatible lifting of $L$-packets (cf. Proposition~\ref{prop: compatible lifting packet}), the left hand sides of \eqref{eq: vanishing 3} and \eqref{eq: vanishing 4} become
\begin{align}
\label{eq: vanishing 5}
\sum_{\x \in Y/\langle \eta_1 \rangle} \Big( \sum_{[\lr] \in \cPkt{\lp^{+}_{-}} \otimes \x} \langle \varepsilon, \lr \rangle \lf_{\lG^{+}}(\lr) - \sum_{[\lr] \in \cPkt{\lp^{+}} \otimes \x} \langle \varepsilon, \lr \rangle \lf_{\lG^{+}}(\lr)\Big)
\end{align}
where $\cPkt{\lp^{+}_{-}}$ denotes the packet of $\lG^{+}$ transferred from the $L$-packet $\cPkt{\lp^{+}_{\varepsilon}}$ of $\lG^{+}_{\varepsilon}$. By the argument on stability as in the proof of Theorem~\ref{thm: global L-packet}, we can conclude the coefficients of characters of all $[\lr] \in \cPkt{\lp^{+}}$ must be the same. Since $\p_{1}$ is not of $GL$-type, then there exist $[\lr_{1}], [\lr_{2}] \in \cPkt{\lp^{+}}$ such that 
\(
\langle \varepsilon, \lr_{1} \rangle = 1, \langle \varepsilon, \lr_{2} \rangle = -1.
\)
So \eqref{eq: vanishing 5} would not be stable unless it is zero. 

At last, we would like to show $\cPkt{\lp} = \cPkt{}(\lp_o; \lp_b)$ if $\p \neq \p_o$, and $\cPkt{\lp} = \cPkt{\lp_{x}}$ for $1 \neq x \in \S{\lp}$ if $\p = \p_o$. Since $\a(\S{\p^{+}_{v}}^{\Sigma_0}) = \a(\S{\p^{+}_{\varepsilon, v}}^{\Sigma_0})$ for all $v$, then \eqref{eq: vanishing 1} implies that
\begin{align*}
\Idt{\lG^{+}_{\epsilon}}{, \p^{+}_{\epsilon}} - \frac{1}{2} {\rm Tran} (\Idt{\lG^{+}_{\epsilon, o}}{, \p^{+}_{\epsilon, o}} \tilde{\otimes} \, \Idt{\lG_b}{, \lp_b}) = 0.
\end{align*}
Since $\cPkt{\lp^{+}_{\varepsilon}} $ (resp. $\cPkt{\lp^{+}_{\epsilon, o}}$) is the restriction of $\cPkt{\lp} \otimes \cPkt{\lp_{1}}$ (resp. $\cPkt{\lp_{o}} \otimes \cPkt{\lp_{1}}$) and $\a(\S{\p_{1, v}}^{\Sigma_0}) \subseteq \a(\S{\p_{v}}^{\Sigma_0})$ for all $v$, then 
\(
\cPkt{\lp} = \cPkt{}(\lp_o; \lp_b). 
\) 
The other case follows from \eqref{eq: vanishing 2} by the same argument.

\begin{corollary}
\label{cor: functoriality of twisted endoscopic transfer for general group}
Suppose 
\(
G = G(n_{1}) \times \cdots \times G(n_{q})
\)  
and $\p = \p_{1} \times \cdots \times \p_{q} \in \cP{G}$ such that $n_i \leqslant N$, then Theorem~\ref{thm: functoriality} ~\ref{thm: stable multiplicity formula}, ~\ref{thm: stable multiplicity formula product} hold for $\lp$.
\end{corollary}

\begin{proof}
The proof is the same as Corollary~\ref{cor: functoriality of twisted endoscopic transfer for general group 1} under Lemma~\ref{lemma: functoriality of twisted endoscopic transfer for non-discrete parameters}, \ref{lemma: functoriality of twisted endoscopic transfer for discrete parameters}, except for (2), (3) of Theorem~\ref{thm: stable multiplicity formula}, ~\ref{thm: stable multiplicity formula product} in the discrete case, where we need to compare \eqref{eq: endoscopic expansion}, \eqref{eq: endoscopic expansion 1} with \eqref{formula: discrete spectrum} instead.
\end{proof}




\appendix

\section{}
\label{sec: twist}

Let $F$ be a $p$-adic field and $G = SO(2n, \eta')$ and $\p \in \cPsm{G}$. We would like to show

\begin{theorem}
\label{thm: twist}
For $\omega \in H^{1}(W_{F}, Z(\D{\lG})) \cong {\rm Hom}(\lG(F), \mathbb{C}^{\times})$, 
\begin{align}
\label{eq: twist equivariant}
[\lp \otimes \omega] = [\lp] \text{ if and only if } \cPkt{\lp} \otimes \omega = \cPkt{\lp}.
\end{align}
\end{theorem}

Dual to \eqref{eq: twist}, we have
\[
\xymatrix{1 \ar[r] & {\rm Hom}(\lG(F)/G(F), \mathbb{C}^{\times}) \ar[r] & {\rm Hom}(\lG(F), \mathbb{C}^{\times}) \ar[r] & \ar[r] {\rm Hom}(G(F), \mathbb{C}^{\times}) & 1.
}
\]
By \cite[Corollary 4.2]{Xu:2018}, \eqref{eq: twist equivariant} holds for $\p \in \cPbd{G}$ and $\omega \in {\rm Hom}(\lG(F)/G(F), \mathbb{C}^{\times})$. Since we have assumed $\p \in \cPsm{G}$, then any nontrivial $\omega \in {\rm Hom}(\lG(F), \mathbb{C}^{\times})$ such that $[\lp \otimes \omega] = [\lp]$ must be nontrivial on $G(F)$. In order to describe ${\rm Hom}(G(F), \mathbb{C}^{\times})$, we lift them to characters of $GSpin(2n, \eta')(F)$.
\[
\xymatrix{&&1 \ar[d] && \\
&& \mathbb{G}_m \ar[rd]^{z \mapsto z^2} \ar[d] &&  \\
1 \ar[r] & Spin(2n, \eta') \ar[r] \ar[rd] & GSpin(2n, \eta') \ar[r] \ar[d] & \mathbb{G}_m \ar[r] & 1\\
&& SO(2n, \eta') \ar[d]& &\\
&&1 &&
}
\]
Let us denote $G_{sc} := Spin(2n, \eta')$ and $\lG_{sc} := GSpin(2n, \eta')$. Then
\[
{\rm Hom}(G(F), \mathbb{C}^{\times}) \cong {\rm Hom}(\lG_{sc}(F)/F^{\times}G_{sc}(F), \mathbb{C}^{\times}) \cong {\rm Hom}(F^{\times}/F^{\times 2}, \mathbb{C}^{\times}).
\]

\begin{lemma}
\label{lemma: twist}
For $\eta \in {\rm Hom}(F^{\times}/F^{\times 2}, \mathbb{C}^{\times})$, $\cPkt{\p \otimes \eta} = \cPkt{\p} \otimes \eta$.
\end{lemma}

\begin{proof}
The packet $\cPkt{\p}$ is determined by its $\theta_{2n}$-twisted endoscopic transfer $\r^{GL}_{\p}$ to $GL(2n, F)$ through the character relation 
\[
f^{G}(\p) = f_{GL(2n)^{\theta_{2n}}}(\pi^{GL}_{\p}), \quad f \in C^{\infty}_{c}(GL(2n, F))
\]
(cf. \cite[Theorem 2.2.1]{Arthur:2013}). Since $\r^{GL}_{\p \otimes \eta} \cong \r^{GL}_{\p} \otimes \eta$ and $(f \eta)^{G} = f^{G} \eta$, then
\[
f^{G}(\p \otimes \eta) = f_{GL(2n)^{\theta_{2n}}}(\r^{GL}_{\p} \otimes \eta) = (f\eta)_{GL(2n)^{\theta_{2n}}}(\r^{GL}_{\p}) = (f^{G}\eta)(\p).
\]
This finishes the proof.
\end{proof}

Suppose $1 \neq \eta \in {\rm Hom}(F^{\times}/F^{\times 2}, \mathbb{C}^{\times})$ is associated with a quadratic extension $E/F$ and $[\phi \otimes \eta] = [\phi]$, then $\cPkt{\p} \otimes \eta = \cPkt{\p}$ by Lemma~\ref{lemma: twist}. There exists a unique extension $\tilde{\eta} \in {\rm Hom}(\lG(F), \mathbb{C}^{\times})$ of $\eta$ such that $[\lp \otimes \tilde{\eta}] = [\lp]$. In order to prove Theorem~\ref{thm: twist}, it suffices to show that $\cPkt{\lp} \otimes \tilde{\eta} = \cPkt{\lp}$. 

Note $\r^{GL}_{\p}$ is an essential discrete series representation of $GL(2n, F)$. We denote its Langlands parameter again by $\p$. Since $\pi^{GL}_{\p} \otimes \eta \cong \pi^{GL}_{\p}$, there exists an essential discrete series representation $\pi_{E}$ of $GL(n, E)$, which is not $\Gamma_{E/F}$-conjugation invariant, such that $\pi^{GL}_{\p}$ is the automorphic induction of $\pi_{E}$ (cf. \cite{Henniart-Herb:1995} \cite{AC:1989}). Denote the Langlands parameter of $\pi_E$ by $\p_{E}$. We can also view $\pi_{E}$ as a representation of ${\rm Res}_{E/F} \, GL(n)(F)$ and denote the corresponding Langlands parameter by $\p_{E/F}$. Since
\(
\p = {\rm Ind}^{L_{F}}_{L_{E}} \, \p_{E}
\) 
(cf. \cite{Henniart:2001}), then $\p$ factors through $\p_{E/F}$. Moreover, 
\[
\p|_{L_{E}} = \p^{c}_{E} \oplus \p_{E}.
\]
where $\p^{c}_{E}$ is the $\Gamma_{E/F}$-conjugate of $\p_{E}$. Since $\p$ is self-dual of orthogonal type, then $\p_{E}$ is either self-dual of orthogonal type or conjugate orthogonal (cf. \cite[Section 2]{Mok:2014}). It follows that $\p_{E/F}$ factors through $\p_{H} \in \cPbd{H}$ for a twisted elliptic endoscopic group $H$ of ${\rm Res}_{E/F} \, GL(n)$. It also induces an embedding of $\L{H}$ into $\L{G}$, through which we can view $H$ as a twisted elliptic endoscopic group of $G$. 
\[
\xymatrix{\p: L_{F} \ar[r] \ar[rd] & \L{G} \ar[r] & ^{L}GL(2n) \\
& \L{H} \ar[r] \ar[u] & ^{L}{\rm Res}_{E/F} \, GL(n) \ar[u]}
\]
More precisely, we have the following three cases.
\begin{itemize}

\item If $\p^{c}_{E} = \p^{\vee}_{E}$, then $H = U_{E/F}(n)$ is an $\eta$-twisted endoscopic group of $G$. In this case,
\[
\eta' =  \begin{cases}
1 & \text{ if $n$ is even}, \\
\eta & \text{ if $n$ is odd}.
\end{cases}
\]

\item If $\p_{E} = \p^{\vee}_{E}$ and $n$ is even, then $H = {\rm Res}_{E/F} \, SO(n, \eta_{\p_{E}})$ is an $\eta$-twisted endoscopic group of $G$. In this case, $\eta' = \eta_{\p_{E}}|_{F^{\times}}$.

\item If $\p_{E} = \p^{\vee}_{E}$ and $n$ is odd, then $H = {\rm Res}_{E/F} \, Sp(n-1)$ is a $(\theta_0, \eta)$-twisted endoscopic group of $G$. In this case, $\eta' = \eta \cdot \eta_{\p_{E}}|_{F^{\times}}$.

\end{itemize}
We can lift $H$ to an $\tilde{\eta}$-twisted (resp. $(\theta_0, \tilde{\eta})$-twisted) elliptic endoscopic group $\widetilde{H}$ of $\lG$ (cf. \cite[Proposition 3.1]{Xu:2016}). Then $[\lp \otimes \tilde{\eta}] = [\lp]$. 

Next we can construct a globalization of $\p_{H}$, namely $\dot{\p}_{\dot{H}} \in \cPsm{\dot{H}}$ such that $\dot{\p}_{\dot{H}, u} = \p_{H}$ following \cite[Proposition 6.3.1]{Arthur:2013} \cite[Proposition 7.3.1]{Mok:2014}. It gives rise to a self-dual or conjugate self-dual $\dot{\p}_{\dot{E}} \in \Psm{GL(n)_{\dot{E}}}$ such that $\dot{\p}_{\dot{E}, w} = \p_{E}$ for $w|u$. By automorphic induction \cite{AC:1989} \cite{Henniart:2012}, we obtain a self-dual $\dot{\p} \in \Psm{GL(2n)}$ such that $\dot{\p} \otimes \dot{\eta} = \dot{\p}$. Since $\dot{\p}_{u} = \p \in \cPsm{G}$, then $\dot{\p} \in \cPsm{\dot{G}}$, where $\dot{G} = SO(2n, \dot{\eta}')$. 
Let us lift $\dot{H}$ to an $\tilde{\dot{\eta}}$-twisted (resp. $(\theta_0, \tilde{\dot{\eta}})$-twisted) elliptic endoscopic group $\widetilde{\dot{H}}$ of $\widetilde{\dot{G}}$. The following lemma is the key step of the proof.

\begin{lemma}
\label{lemma: comparison}
If $n$ is even, then
\[
I^{(\widetilde{\dot{G}}, \tilde{\dot{\eta}})}_{disc, \dot{\p}} (\tilde{\dot{f}}) = 4 \, \iota(\widetilde{\dot{G}}, \widetilde{\dot{H}}) S^{\widetilde{\dot{H}}}_{disc, \dot{\p}_{\dot{H}}} (\tilde{\dot{f}}^{\widetilde{\dot{H}}}), \quad \tilde{\dot{f}} \in \sH(\widetilde{\dot{G}}, \lif{\chi}).
\]
If $n$ is odd and $\p^{c}_{E} = \p^{\vee}_{E}$, then
\[
I^{(\widetilde{\dot{G}}, \tilde{\dot{\eta}})}_{disc, \dot{\p}} (\tilde{\dot{f}}) = 2 \, \iota(\widetilde{\dot{G}}, \widetilde{\dot{H}}) S^{\widetilde{\dot{H}}}_{disc, \dot{\p}_{\dot{H}}} (\tilde{\dot{f}}^{\widetilde{\dot{H}}}), \quad \tilde{\dot{f}} \in \sH(\widetilde{\dot{G}}, \lif{\chi}).
\]
If $n$ is odd and $\p_{E} = \p^{\vee}_{E}$, then
\[
I^{(\widetilde{\dot{G}}^{\theta_0}, \tilde{\dot{\eta}})}_{disc, \dot{\p}} (\tilde{\dot{f}}) = 4 \, \iota(\widetilde{\dot{G}}, \widetilde{\dot{H}}) S^{\widetilde{\dot{H}}}_{disc, \dot{\p}_{\dot{H}}} (\tilde{\dot{f}}^{\widetilde{\dot{H}}}), \quad \tilde{\dot{f}} \in \sH(\widetilde{\dot{G}}, \lif{\chi}).
\]

\end{lemma}

\begin{proof}
By the twisted stable trace formula, we have 
\begin{align*}
I^{(\widetilde{\dot{G}}, \tilde{\dot{\eta}})}_{disc, \dot{\p}} (\tilde{\dot{f}}) = \sum_{\widetilde{\dot{G}}' \in \End{ell}{\widetilde{\dot{G}}, \tilde{\dot{\eta}}}}\iota(\widetilde{\dot{G}}, \widetilde{\dot{G}}') S^{\widetilde{\dot{G}}'}_{disc, \dot{\p}} (\tilde{\dot{f}}') , \quad \tilde{\dot{f}} \in \sH(\widetilde{\dot{G}}, \lif{\chi}) \\
\text{\Big(resp. $I^{(\widetilde{\dot{G}}^{\theta_0}, \tilde{\dot{\eta}})}_{disc, \dot{\p}} (\tilde{\dot{f}}) = \sum_{\widetilde{\dot{G}}' \in \End{ell}{\widetilde{\dot{G}}^{\theta_0}, \tilde{\dot{\eta}}}}\iota(\widetilde{\dot{G}}, \widetilde{\dot{G}}') S^{\widetilde{\dot{G}}'}_{disc, \dot{\p}} (\tilde{\dot{f}}'), \quad \tilde{\dot{f}} \in \sH(\widetilde{\dot{G}}, \lif{\chi})$. \Big)}
\end{align*}
The $\D{\Sigma}_0$-conjugacy class of isomorphism classes of $\dot{\eta}$-twisted (resp. $(\theta_0, \dot{\eta})$-twisted) elliptic endoscopic data of $\dot{G}$ can be classified by $(\dot{G}', \dot{K}/\dot{E})$, where $[\dot{K}:\dot{E}] \leqslant 2$ and $\dot{G}' = U_{\dot{E}/\dot{F}}(n_1) \times {\rm Res}_{\dot{E}/\dot{F}} \, SO(n_2, \eta_{\dot{K}/\dot{E}})$ (resp. $U_{\dot{E}/\dot{F}}(n_1) \times {\rm Res}_{\dot{E}/\dot{F}} \, Sp(n_2 - 1)$) for $n = n_1 + n_2$ and $n_2$ is even (resp. $n_2$ is odd), subject to the conditions that
\(
\dot{\eta}' = \dot{\eta}^{n} \cdot \eta_{\dot{K}/\dot{E}}|_{\A^{\times}_{\dot{F}}}
\)
and if $\dot{G}' = U_{\dot{E}/\dot{F}}(n)$, then $\dot{K} = \dot{E}$. The twisted endoscopic embedding $\xi$ can be chosen to satisfy the following diagram
\[
\xymatrix{\L{\dot{G}} \ar[rr] && \L{GL(2n)} \\
\L{\dot{G}'} \ar[r]^{\xi_{\dot{E}/\dot{F}} \quad \quad \quad \quad \quad \quad} \ar[u]^{\xi} &  \L{{\rm Res}_{\dot{E}/\dot{F}} (GL(n_1) \times GL(n_2))} \ar[r] &  \L{{\rm Res}_{\dot{E}/\dot{F}} GL(n)}, \ar[u]
}
\]
where $\xi_{\dot{E}/\dot{F}}$ is given by the fiber product over $W_{\dot{F}}$ of
\[
\xi_{\chi_{(-1)^{n_1 - 1}}}: \L{U_{\dot{E}/\dot{F}}(n_1)} \rightarrow \L{{\rm Res}_{\dot{E}/\dot{F}} GL(n_1)}
\]
(cf. \cite[Section 2]{Mok:2014}) and
\begin{align*}
\xi_{\dot{K}/\dot{E}}: \L{{\rm Res}_{\dot{E}/\dot{F}} \, SO(n_2, \eta_{\dot{K}/\dot{E}})} \rightarrow \L{{\rm Res}_{\dot{E}/\dot{F}} GL(n_2)} \\
\text{ \Big(resp. $\xi_{\dot{K}/\dot{E}}: \L{{\rm Res}_{\dot{E}/\dot{F}} \, Sp(n_2 - 1)} \rightarrow \L{{\rm Res}_{\dot{E}/\dot{F}} GL(n_2)} $\Big).}
\end{align*}
By \cite[Proposition 2.7]{Xu:2018}), there is a one to one correspondence
\begin{align*}
\End{ell}{\dot{G}, \dot{\eta}} \cong \bigsqcup_{\tilde{\dot{\eta}}|_{\dot{G}} = \dot{\eta}} \End{ell}{\widetilde{\dot{G}}, \tilde{\dot{\eta}}} \\
\text{\Big(resp. $\End{ell}{\dot{G}^{\theta_0}, \dot{\eta}} \cong \bigsqcup_{\tilde{\dot{\eta}}|_{\dot{G}} = \dot{\eta}} \End{ell}{\widetilde{\dot{G}}^{\theta_0}, \tilde{\dot{\eta}}}$ \Big).} 
\end{align*}
It is easy to show that if $S^{\widetilde{\dot{G}}'}_{disc, \dot{\p}}  \neq 0$, then $S^{\dot{G}'}_{disc, \dot{\p}} \neq 0$ (cf. \cite[Lemma 5.7]{Xu:2018}). Suppose $c(\dot{\p})$ is the image of $c(\dot{\p}')$ for some $\dot{\p}' \in \cPdt{\dot{G}'}$. Compose $\dot{\p}'$ with the embeddings
\[
\L{\dot{G}'} \rightarrow  \L{{\rm Res}_{\dot{E}/\dot{F}} (GL(n_1) \times GL(n_2))} \rightarrow  \L{{\rm Res}_{\dot{E}/\dot{F}} GL(n)}, 
\]
we get $\dot{\p}_{\dot{E}, 1} \boxplus \dot{\p}_{\dot{E}, 2} \in \P{GL(n)_{\dot{E}}}$, where $\dot{\p}_{\dot{E}, 1}$ is conjugate self-dual and $\dot{\p}_{\dot{E}, 2}$ is self-dual. 
If the image of $\dot{\p}'$ in $\cPdt{\dot{G}}$ is equal to $\dot{\p}$, then 
\[
\dot{\p}_{\dot{E}, 1} \boxplus \dot{\p}_{\dot{E}, 2} = \dot{\p}_{\dot{E}} \text{ or } \dot{\p}^{c}_{\dot{E}}.
\]
Hence, it is necessary that $\dot{G}' = \dot{H}$ and $\dot{\p}' = \dot{\p}_{\dot{H}}$ or $\dot{\p}^{c}_{\dot{H}}$, where $\dot{\p}^{c}_{\dot{H}}$ is the $\D{\theta}_{c}$-conjugate of $\dot{\p}_{\dot{H}}$ for an automorphism $\theta_c$ of $\dot{H}$ and it gives rise to $\dot{\p}^{c}_{\dot{E}}$.

If $n$ is odd and $\dot{H} = U_{\dot{E}/\dot{F}}(n)$, the corresponding $\D{\Sigma}_0$-conjugacy class of isomorphism classes of endoscopic data consists of only one isomorphism class. Moreover, ${\rm Out}_{\lif{\dot{G}}_v}(\lif{\dot{H}}_v) = 1$ and
\[
S^{\widetilde{\dot{H}}}_{disc, \dot{\p}^{c}_{\dot{H}}} (\tilde{\dot{f}}^{\lif{\dot{H}}}) = S^{\widetilde{\dot{H}}}_{disc, \dot{\p}_{\dot{H}}} ((\tilde{\dot{f}}^{\theta_0})^{\lif{\dot{H}}}), \quad \tilde{\dot{f}} \in \H(\lif{\dot{G}}, \lif{\chi}).
\]
In other cases, it follows from $\theta_c \in {\rm Out}_{\lif{\dot{G}}_v}(\lif{\dot{H}}_v)$ that
\[
S^{\widetilde{\dot{H}}}_{disc, \dot{\p}^{c}_{\dot{H}}} (\tilde{\dot{f}}^{\lif{\dot{H}}}) = S^{\widetilde{\dot{H}}}_{disc, \dot{\p}_{\dot{H}}} (\tilde{\dot{f}}^{\lif{\dot{H}}}), \quad \tilde{\dot{f}} \in \H(\lif{\dot{G}}, \lif{\chi}).
\]
Moreover, the corresponding $\D{\Sigma}_0$-conjugacy class of isomorphism classes of endoscopic data consists of two isomorphism classes. We can lift them to twisted endoscopic data of $\lif{\dot{G}}$ and denote the transfers by $\tilde{\dot{f}}^{\lif{\dot{H}}}$ and $\tilde{\dot{f}}^{\lif{\dot{H}}_1}$ respectively. Then 
\[
\tilde{\dot{f}}^{\lif{\dot{H}}_1} = (\tilde{\dot{f}}^{\theta_0})^{\lif{\dot{H}}}, \quad \tilde{\dot{f}} \in \H(\lif{\dot{G}}, \lif{\chi})
\]
and $\iota(\lif{\dot{G}}, \lif{\dot{H}})$ remains the same for the two endoscopic data. So in all cases we can conclude the lemma by restricting to $\sH(\lif{\dot{G}}, \lif{\chi})$. 

\end{proof}

\begin{remark}
We can determine the coefficient $\iota(\lif{\dot{G}}, \lif{\dot{H}})$ in the lemma using \eqref{eq: stabilization coefficient} as follows.
\[
\iota(\lif{\dot{G}}, \lif{\dot{H}}) = \begin{cases} 1 & \text{ if $\dot{\eta}' = \dot{\eta}$ and $\dot{H} = U_{\dot{E}/\dot{F}}(n)$, } \\ 1/2 & \text{ if $\dot{\eta}' = 1$ and $\dot{H} = {\rm Res}_{\dot{E}/\dot{F}} \, Sp(n - 1)$ or $U_{\dot{E}/\dot{F}}(n)$, } \\
1/4 & \text{ otherwise. } 
\end{cases}
\]
\end{remark}

To finish the proof of Theorem~\ref{thm: twist}, we still need to show the distributions in Lemma~\ref{lemma: comparison} are nonzero. By the stable trace formula of $\lif{\dot{H}}$, we have
\[
S^{\widetilde{\dot{H}}}_{disc, \dot{\p}_{\dot{H}}} (\tilde{\dot{f}}^{\widetilde{\dot{H}}}) = I^{\widetilde{\dot{H}}}_{disc, \dot{\p}_{\dot{H}}} (\tilde{\dot{f}}) - \sum_{\widetilde{\dot{H}}' \in \End{ell}{\widetilde{\dot{H}}} - \{\lif{\dot{H}}\}} \iota(\widetilde{\dot{H}}, \widetilde{\dot{H}}') S^{\widetilde{\dot{H}}'}_{disc, \dot{\p}_{\dot{H}}} (\tilde{\dot{f}}'), \quad \tilde{\dot{f}} \in \H(\lif{\dot{H}}, \lif{\chi}_{\lif{\dot{H}}}).
\]
Since $\dot{\p}_{\dot{H}}$ does not factor through any proper elliptic endoscopic group of $\dot{H}$, then $S^{\dot{H}'}_{disc, \dot{\p}_{\dot{H}}}(\dot{f}') = 0$ for $\dot{H}' \neq H$, which implies that $S^{\widetilde{\dot{H}}'}_{disc, \dot{\p}_{\dot{H}}}({\tilde{\dot{f}}'}) = 0$ (cf. \cite[Lemma 5.7]{Xu:2018}). Therefore,
\[
S^{\widetilde{\dot{H}}}_{disc, \dot{\p}_{\dot{H}}} (\tilde{\dot{f}}^{\widetilde{\dot{H}}}) = I^{\widetilde{\dot{H}}}_{disc, \dot{\p}_{\dot{H}}} (\tilde{\dot{f}}) = tr R^{\widetilde{\dot{H}}}_{disc, \dot{\p}_{\dot{H}}} (\tilde{\dot{f}}), \quad \tilde{\dot{f}} \in \H(\lif{\dot{H}}, \lif{\chi}_{\lif{\dot{H}}}).
\]
So there exists $\tilde{\dot{f}} \in \sH(\lif{\dot{G}}, \lif{\chi})$ such that $S^{\widetilde{\dot{H}}}_{disc, \dot{\p}_{\dot{H}}} (\tilde{\dot{f}}^{\widetilde{\dot{H}}}) \neq 0$, hence $I^{(\widetilde{\dot{G}}, \tilde{\dot{\eta}})}_{disc, \dot{\p}} (\tilde{\dot{f}}) \neq 0$ (resp. $I^{(\widetilde{\dot{G}}^{\theta_0}, \tilde{\dot{\eta}})}_{disc, \dot{\p}} (\tilde{\dot{f}}) \neq 0$) by Lemma~\ref{lemma: comparison}. It follows that there exists $[\tilde{\dot{\r}}] \in \cPkt{\tilde{\dot{\p}}}$ such that $\tilde{\dot{\r}} \cong \tilde{\dot{\r}} \otimes \tilde{\dot{\eta}}$ (resp. $\tilde{\dot{\r}}^{\theta_0} \cong \tilde{\dot{\r}} \otimes \tilde{\dot{\eta}})$. In particular, $\tilde{\dot{\r}}_u \cong \tilde{\dot{\r}}_u \otimes \tilde{\eta}$ (resp. $\tilde{\dot{\r}}_u^{\theta_0} \cong \tilde{\dot{\r}}_u \otimes \tilde{\eta}$). Hence $\cPkt{\lp} = \cPkt{\lp} \otimes \tilde{\eta}$.

\bibliographystyle{amsalpha}

\bibliography{reps}

\end{document}